\documentclass[11pt,a4paper]{amsart}
\usepackage[mathcal]{eucal}
\usepackage{amsfonts}
\usepackage{amsbsy}
\usepackage{amssymb}
\usepackage{amsmath}
\usepackage{graphicx}
\usepackage[latin2]{inputenc}

\newtheorem{thm}{Theorem}[section]
\newtheorem{cor}[thm]{Corollary}
\newtheorem{lem}[thm]{Lemma}

\newcounter{hypothA}
\newcounter{VECHIhypothF}
\newcounter{hypothF}
\newcounter{hypothE}
\newcounter{hypothL}
\newcounter{hypothH}
\theoremstyle{definition}

\theoremstyle{remark}
\newtheorem{rem}[thm]{Remark}
\numberwithin{equation}{section}

\begin{document}

\title[On the construction ...]
{On the construction of a family of well - posed approximate
formulations for the stationary Stokes problem using an extended
system}%
\author[C\u{a}t\u{a}lin - Liviu Bichir]%
{C\u{a}t\u{a}lin - Liviu Bichir}%

\address{independent researcher, 74 Nicolae B\u{a}lcescu Street, Flat V1, Apartment 5, Gala\c{t}i, 800001, Romania}%
\email{catalinliviubichir@yahoo.com}%





\begin{abstract}
We introduce an exact parameterized extended system such that,
under adequate data, between the components of its solution, there
is the solution of the weak formulation of the homogeneous
Dirichlet problem for the stationary Stokes equations. In the
extended system, we introduce the momentum equation together with
two other forms of this one. This allows us to reformulate, for
the stationary case, the consistent pressure Poisson equation of
Sani, Shen, Pironneau, Gresho [\textit{Int. J. Numer. Meth.
Fluids}, \textbf{50} (2006), pp. 673-682], from the unsteady case.
In this way, we can retain the information we need for the
approximate pressure on the boundary. We obtain a parameterized
perturbed pressure Poisson equation for the stationary Stokes
problem. We prove that to solve the stationary Stokes problem is
equivalent to solve a problem for the momentum equation, the
parameterized equation and the equation that defines the Laplace
operator acting on velocity. The approximation of this last
problem give a family of well - posed approximate formulations.
The solution of each element of this family approximates the
solution of the stationary Stokes problem. Some necessary variants
of existing results on general Banach spaces are also developed.
These concern the density of subspaces related to isomorphisms and
the connection between the exact and the approximate problems.
First, the well - posedness of the exact extended system is proved
in some dense subspaces. The well posed approximate problem does
not necessitate the discrete inf-sup condition. The paper is also
related to the existing discussions on the boundary conditions for
the pressure and on the imposing of the incompressibility
constraint on the boundary.
\end{abstract}
\maketitle

\section{Introduction}
\label{sectiunea_0_introduction_Stokes}

Using the information from the momentum equation for the
approximate pressure on the boundary and avoiding the discrete
inf-sup condition, we obtain a parameterized family of approximate
problems of the weak formulation of the homogeneous Dirichlet
problem for the stationary Stokes equations. In the case of the
finite element method, under some hypotheses, the corresponding
matrix form is
\begin{equation}
\label{ecStokes_ecPoisson_variational_01}
      \left[\begin{array}{cc}
      A_{h}
         &  B_{h}^{T} \\
      G_{h}(\alpha)
         &  D_{h}
      \end{array}\right]
      \left[\begin{array}{c}
      \bar{\textbf{u}}_{h} \\
      \bar{p}_{h}
      \end{array}\right]
      = \left[\begin{array}{c}
      \bar{\textbf{f}} \\
      \bar{\psi}
      \end{array}\right] \ , \quad \alpha \neq 1 \ ,
\end{equation}
with two block Poisson preconditioners $A_{h}$ and $D_{h}$ and a
parameter $\alpha$, instead of the well known matrix problem
\begin{equation}
\label{ecStokes_ecPoisson_variational_01_STOKES}
      \left[\begin{array}{cc}
      A_{h}
         &  B_{h}^{T} \\
      B_{h}
         &  0
      \end{array}\right]
      \left[\begin{array}{c}
      \bar{\textbf{u}}_{h} \\
      \bar{p}_{h}
      \end{array}\right]
      = \left[\begin{array}{c}
      \bar{\textbf{f}} \\
      \bar{g}
      \end{array}\right] \ .
\end{equation}

The Stokes problem is well posed and one of the hypotheses that is
satisfied is the inf-sup (or the LBB) condition. The approximate
problem also needs this condition in the discrete formulation
\cite{CLBichir_bib_Atkinson_Han2009, CLBichir_bib_Babuska1971,
CLBichir_bib_Bochev_GunzburgerLSFEM2009,
CLBichir_bib_Brenner_Scott2008, CLBichir_bib_Brezzi1974,
CLBichir_bib_Brezzi_Fortin1991, CLBichir_bib_Ciarlet2002,
CLBichir_bib_De_Fi_Mu2004, CLBichir_bib_A_Ern2005,
CLBichir_bib_Gir_Rav1986, CLBichir_bib_Gunzburger1989,
CLBichir_bib_Quarteroni_Valli2008}. The associated matrix problem
can be solved by direct methods
\cite{CLBichir_bib_Elman_Silvester_Wathen2005,
CLBichir_bib_Nocedal_Wright2006}. Usually, the problem is
reformulated using penalty methods or augmented Lagrangian
variational methods or with methods that lead to well posed
approximate problems without the need of the inf-sup condition
such as stabilization methods and least-squares finite element
methods (\cite{CLBichir_bib_Barth_Bochev_Gunzburger_Shadid2004} -
\cite{CLBichir_bib_Brenner_Scott2008},
\cite{CLBichir_bib_Brezzi_Douglas1988,
CLBichir_bib_Brezzi_Fortin1991,
CLBichir_bib_Brezzi_Pitkaranta1984, CLBichir_bib_Cu_Se_St1986},
\cite{CLBichir_bib_De_Fi_Mu2004} -
\cite{CLBichir_bib_Elman_Silvester_Wathen2005},
\cite{CLBichir_bib_Gir_Rav1986, CLBichir_bib_Glowinski1984,
CLBichir_bib_Gunzburger1989, CLBichir_bib_Peyret_Taylor1983,
CLBichir_bib_Quarteroni_Valli2008, CLBichir_bib_Temam1979,
CLBichir_bib_Thomasset1981} and the references cited therein).

Another method is to replace the Stokes problem (where the
equations are the momentum equation and the incompressibility
constraint) by a problem for the momentum equation and the Poisson
equation for the pressure
\cite{CLBichir_bib_Bochev_GunzburgerLSFEM2009,
CLBichir_bib_De_Fi_Mu2004, CLBichir_bib_Gir_Rav1986,
CLBichir_bib_Glowinski1984,
CLBichir_bib_Glowinski_Pironneau1979_Numer_Math,
CLBichir_bib_Gresho_Sani1987, CLBichir_bib_Gunzburger1989,
CLBichir_bib_Quarteroni_Valli2008,
CLBichir_bib_Sani_Shen_Pironneau_Gresho2006,
CLBichir_bib_Waters_Fix_Cox2004}. This method is accompanied by
discussions on the boundary conditions for the pressure and on the
imposing of the incompressibility constraint on the boundary.

In this paper, we first formulate a parameterized family of exact
extended systems (related to the problem we study, on infinite
dimensional spaces). Under adequate data, the solution
$(\textbf{u},p)$ of the weak formulation of the homogeneous
Dirichlet problem for the stationary Stokes equations is between
the components of the solution of each element of this family. To
solve this problem for the stationary Stokes equations becomes
equivalent to solve a such exact extended system. Each extended
system is approximated by a well posed formulation without the use
of the discrete inf-sup (or the LBB) condition. There results a
solution set of approximate solutions of Stokes problem. In the
case of the finite element method, under some hypotheses, this
solution set is given by a family of well - posed problems
(\ref{ecStokes_ecPoisson_variational_01}). The results are
obtained for $(\textbf{u},p)$ $\in$ $\textbf{H}_{0}^{1}(\Omega)$
$\times$ $(H^{1}(\Omega) \cap L_{0}^{2}(\Omega))$.

A component of the exact extended system is the momentum equation.
Some other components are constructed on the basis of the
Glowinski - Pironneau decomposition for the pressure
(\cite{CLBichir_bib_De_Fi_Mu2004, CLBichir_bib_Gir_Rav1986,
CLBichir_bib_Glowinski1984,
CLBichir_bib_Glowinski_Pironneau1979_Numer_Math,
CLBichir_bib_Waters_Fix_Cox2004}). The extended system also
includes a reformulation, for the stationary case, of the
consistent pressure Poisson equation of Sani, Shen, Pironneau,
Gresho, \cite{CLBichir_bib_Sani_Shen_Pironneau_Gresho2006}, from
the unsteady case. We use this one for the approximation of the
pressure on the boundary. This equation is based on the second
writing of the momentum equation. We construct the (discrete)
Laplace operator acting on pressure by connecting this last
equation and the rest of the equations. The third writing of the
momentum equation leads to the connection between the velocity and
the Laplace operator acting on velocity. All these considerations
allow us to prove that to solve the weak formulation of the
homogeneous Dirichlet problem for the stationary Stokes equations
is equivalent to solve a problem for the momentum equation, the
parameterized perturbed pressure Poisson equation and the equation
that defines the Laplace operator acting on velocity. The
conditions for the approximation of the Laplace operator acting on
velocity on the boundary become essential since they determine the
approximation of the pressure on the boundary. For instance, from
the paper \cite{CLBichir_bib_Bochev_Gunzburger2004} of Bochev and
Gunzburger, we retain a modality to evaluate the discrete Laplace
operator acting on velocity in the case of finite element method.

A density result for isomorphisms, from the functional analysis
book by Sire\c{t}chi \cite{CLBichir_bib_Sir_AF1982}, allows us to
provide proofs on dense subspaces. In order to prove the well -
posedness of the exact extended system (where the parameter is
fixed), we restrict the problem to some dense subspaces. In this
way, we can combine, in the proof, the utilization of Glowinski -
Pironneau of a potential related to velocity
(\cite{CLBichir_bib_De_Fi_Mu2004, CLBichir_bib_Gir_Rav1986,
CLBichir_bib_Glowinski1984,
CLBichir_bib_Glowinski_Pironneau1979_Numer_Math,
CLBichir_bib_Waters_Fix_Cox2004}) and a method used by Ciarlet
(\cite{CLBichir_bib_Ciarlet2002}) for the variational formulation
of the biharmonic problem. We do not impose the homogeneous
Neumann condition for the potential, which is usually imposed in
literature. This condition arises naturally related to the dense
subspaces in a stage of the proof. Finally, the well - posedness
is obtained on the initial (extended) spaces using the result from
\cite{CLBichir_bib_Sir_AF1982}. No condition for the pressure on
the boundary is necessary. No incompressibility constraint on the
boundary is necessary. This proof and the modality we consider the
approximate pressure on the boundary relate the paper to the above
mentioned discussions on these conditions on the  boundary.

In the exact extended system, some of the equations are kept in
the exact form, others are approximated. In order to study the
approximate extended system that results, we formulate a general
theorem on the well-posedness of a such approximate problem. This
theorem is a variant of a theorem of Kantorovich and Akilov
\cite{CLBichir_bib_Ka_Ak1986}, modified by some formulations of
Girault and Raviart \cite{CLBichir_bib_Gir_Rav1986} that we use,
and by our formulations. Resulting from the conditions for the
Galerkin method and the conclusion of C\'ea's lemma,
\cite{CLBichir_bib_Atkinson_Han2009,
CLBichir_bib_Bochev_GunzburgerLSFEM2009,
CLBichir_bib_Brenner_Scott2008, CLBichir_bib_Ciarlet2002,
CLBichir_bib_Dautray_Lions_vol6_1988, CLBichir_bib_A_Ern2005,
CLBichir_bib_Ili1980, CLBichir_bib_Quarteroni_Valli2008,
CLBichir_bib_Temam1979, CLBichir_bib_E_Zeidler_NFA_IIA,
CLBichir_bib_E_Zeidler_NFA_IIB}, some practical conditions are
given. The approximate part of the problem is related to the
corresponding exact one as in the case of nonlinear problems from
\cite{CLBichir_bib_Brezzi_Rappaz_Raviart1_1980,
CLBichir_bib_Cr_Ra1990, CLBichir_bib_Gir_Rav1986}, some elliptic
operators allowing the transformations. The study of the
approximate problem is reduced to the existing results for some
approximate problems for second order elliptic partial
differential equations.

Finally, we treat an application with equal order finite element
spaces.

To the best of our knowledge, the approach in the paper and the
results that we prove are new.

The paper is organized as follows. In Section
\ref{sectiunea_ecuatii_Stokes}, we review a summary of notations,
some results from functional analysis and some reformulations of
Stokes problem. In Sections
\ref{well_posed_sectiunea_general_theorem_on_the_well_posedness_problema_data}
and
\ref{well_posed_sectiunea_general_theorem_on_the_well_posedness},
we develop some formulations on the well-posedness of general
exact and approximate problems of the type of those used in the
paper. Sections \ref{sectiunea_ecuatii_remarks},
\ref{sectiunea_forme_echivalente_ale_ecuatiilor_extended} and
\ref{sectiunea_alt_izomorfism} are devoted to the construction of
the exact extended system. The approximate extended system is
studied in Section
\ref{sectiunea_ecuatii_Stokes_problema_aproximativa}. An analysis
using finite element method is in Section
\ref{sectiunea_ecuatii_Stokes_problema_aproximativa_partea_2}.

\section{Preliminaries and the Stokes problem}
\label{sectiunea_ecuatii_Stokes}

\subsection{Notations}
\label{sectiunea_ecuatii_Stokes_1}

Let $\Omega$ be a bounded and connected open subset of
$\mathbb{R}^{N}$ ($N=2,3$) with a Lipschitz - continuous boundary
$\partial \Omega$.

Let $L^{2}(\Omega)$ be space of all square integrable functions
defined on $\Omega$ with the usual inner product $(\cdot,\cdot)$.
Let $H^{m}(\Omega)$ be the Sobolev space of all functions having
the derivatives up to order $m$ as elements of $L^{2}(\Omega)$,
with the standard norm $\| \cdot \|_{m}$ and seminorm $| \cdot
|_{m}$. Let $\mathcal{D}(\Omega)$ be the linear space of
infinitely differentiable functions with compact support in
$\Omega$, topologized as it is described in
\cite{CLBichir_bib_Sir_AF1982, CLBichir_bib_KYosida1980}.
$H_{0}^{1}(\Omega)$ denotes the closure of $\mathcal{D}(\Omega)$
with respect to the norm $\| \cdot \|_{1}$. $L_{0}^{2}(\Omega)$
$=$ $\{ p \in L^{2}(\Omega) \ | \ \int_{\Omega}p \ dx = 0 \}$.

$H^{-1}(\Omega)$ denotes the dual space of $H_{0}^{1}(\Omega)$ and
$V'$ denotes the dual space of the linear space $V$. $\langle
\cdot,\cdot \rangle$ is the duality pairing.
$\mathcal{D}'(\Omega)$ is the space of distributions on $\Omega$.
$\| \cdot \|_{-1}$ is the norm of $H^{-1}(\Omega)$.

On the product space $Y=\prod_{\jmath \: = 1}^{M}Y_{\jmath}$ , we
use the norm $\| \cdot \|_{(1)}$, $\| \kappa \|_{(1)}$ $=$ $\|
\kappa \|_{Y}=\sum_{\jmath \: = 1}^{M}\| \kappa_{\jmath}
\|_{Y_{\jmath}}$, where $Y_{\jmath}$ is a Banach space with norm
$\| \cdot \|_{Y_{\jmath}}$, for $\jmath=1,\ldots,M$, and
$\kappa=(\kappa_{1},\ldots,\kappa_{M}) \in Y$.

$(v,w)_{V}$ $=$ $\sum_{\imath \: =
1}^{M}(v_{\imath},w_{\imath})_{V_{\imath}}$ is the inner product
on the product space $V$ of the Hilbert spaces $V_{\imath}$ with
the inner product $(\cdot,\cdot)_{V_{\imath}}$,
$v=(v_{1},\ldots,v_{M})$, $v_{\imath}$ $\in$ $V_{\imath}$. The
corresponding norm is $\| v \|_{V}$ $=$ $(\sum_{\imath \: =
1}^{M}\| v_{\imath} \|_{V_{\imath}}^{2})^{1/2}$. $(grad \, v,grad
\, w)$ $=$ $\sum_{i,j=1}^{N}(\frac{\partial \, v_{i}}{\partial \,
x_{j}},\frac{\partial \, w_{i}}{\partial \, x_{j}})$ for $v$, $w$
$\in$ $H^{1}(\Omega)^{N}$, $(grad \, p,grad \, q)$ $=$
$\sum_{j=1}^{N}(\frac{\partial \, p}{\partial \,
x_{j}},\frac{\partial \, q}{\partial \, x_{j}})$ for $p$, $q$
$\in$ $H^{1}(\Omega)$. If $N \geq 2$, the norm and the seminorm on
$H^{m}(\Omega)^{N}$ are simply denoted as in the case $N = 1$.

For vector-valued functions and their spaces, we use the boldface
font notation. For instance, $\textbf{H}_{0}^{1}(\Omega)$ $=$
$H_{0}^{1}(\Omega)^{N}$ and we write $\textbf{u}$ $\in$
$\textbf{H}_{0}^{1}(\Omega)$.

For the inner product on $L^{2}(\Omega)$ and
$\textbf{L}^{2}(\Omega)$, we use the same notation
$(\cdot,\cdot)$.

\subsection{Some results of functional analysis}
\label{sectiunea_ecuatii_Stokes_2}

We specify that, given two normed spaces $E$ and $F$, an
isomorphism of $E$ onto $F$ is a linear, continuous and bijective
mapping $A:E \rightarrow F$ whose inverse $A^{-1}$ is continuous
\cite{CLBichir_bib_Bochev_GunzburgerLSFEM2009,
CLBichir_bib_Braess2007, CLBichir_bib_Cr_Ra1990,
CLBichir_bib_Fabian_Habala_Hajek_Montesinos_Zizler2011,
CLBichir_bib_Gir_Rav1986, CLBichir_bib_Sir_AF1982,
CLBichir_bib_Sir_Spaces_AF1982}. We then say that problem $Au = f$
is well-posed. $L(E,F)$ denotes the space of all the continuous
linear mappings (operators) from $E$ to $F$.

\begin{lem}
\label{lema_prelungire} (\cite{CLBichir_bib_Atkinson_Han2009},
Theorem 2.4.1; \cite{CLBichir_bib_Kreyszig1978}, Theorem 2.7-11;
\cite{CLBichir_bib_Sir_AF1982}, Theorem 5.3.2, Corollary 1;
\cite{CLBichir_bib_E_Zeidler_NFA_IIA}, Proposition 18.29 (a)) Let
$E$ be a normed space, let $E_{0} \subseteq E$ be a linear
subspace dense in $E$, and let $F$ be a Banach space. Assume
$A_{0} \in L(E_{0},F)$. Then, $A_{0}$ has an unique extension $A$
$\in$ $L(E,F)$ with $Au$ $=$ $A_{0}u$, $\forall$ $u$ $\in$
$E_{0}$, and $\| A \|_{L(E,F)}$ $=$ $\| A_{0} \|_{L(E_{0},F)}$.
\end{lem}

\begin{lem}
\label{lema_izom} (\cite{CLBichir_bib_Sir_AF1982}, Theorem 5.3.2,
Corollary 5) Let $E$, $F$ be two Banach spaces. There exists an
isomorphism of $E$ onto $F$ if and only if there exist two linear
subspaces $E_{0}$ and $F_{0}$ of $E$ and $F$ respectively, $E_{0}$
dense in $E$, $F_{0}$ dense in $F$, such that there exists an
isomorphism of $E_{0}$ onto $F_{0}$.

\end{lem}

\begin{cor}
\label{corolarul_lema_izom} In Lemma \ref{lema_izom}, the
isomorphism of $E$ onto $F$ is a unique extension from $E_{0}$ to
$E$ of the isomorphism of $E_{0}$ onto $F_{0}$.
\end{cor}

\begin{proof}
The proof (from \cite{CLBichir_bib_Sir_AF1982}) of Lemma
\ref{lema_izom} uses the result from Lemma \ref{lema_prelungire}.

\qquad
\end{proof}

\begin{cor}
\label{corolarul_lema_izom_restr_prel1} Consider the spaces $E$,
$F$, $E_{0}$ and $F_{0}$ from Lemma \ref{lema_izom}. Let
$\Phi_{1}$ $\in$ $L(E,F)$ be given and let $\Phi_{0}$ be the
restriction of $\Phi_{1}$ to $E_{0}$. If $\Phi_{0}$ is an
isomorphism of $E_{0}$ onto $F_{0}$, then $\Phi_{1}$ is an
isomorphism of $E$ onto $F$.
\end{cor}

\begin{proof} From Lemma \ref{lema_izom}, it follows that there exists an
isomorphism $\Phi$ of $E$ onto $F$. By Corollary
\ref{corolarul_lema_izom}, $\Phi$ is a unique extension of
$\Phi_{0}$ to $E$. So $\Phi$ $=$ $\Phi_{1}$.

\qquad
\end{proof}

\begin{cor}
\label{corolarul_lema_izom_restr} If $\Phi$ is the isomorphism of
$E$ onto $F$ and $\Phi_{0}$ is the isomorphism of $E_{0}$ onto
$F_{0}$ from Lemma \ref{lema_izom}, then $\Phi_{0}$ $=$
$\Phi|_{E_{0}}$ and $\Phi_{0}^{-1}$ $=$ $\Phi^{-1}|_{F_{0}}$.
\end{cor}

\begin{proof}
The proof follows from the proof (from
\cite{CLBichir_bib_Sir_AF1982}) of Lemma \ref{lema_izom}.

\qquad
\end{proof}

\begin{lem}
\label{lema_izom_corolar} Let $E$, $F$ be two normed spaces.
Assume that there exists an isomorphism $A$ of $E$ onto $F$. Then,
for a linear subspace $E_{0}$ dense in $E$, we have that $A_{0}$
$=$ $A|_{E_{0}}$ is an isomorphism of $E_{0}$ onto $F_{0}$ $=$
$A(E_{0})$ and $F_{0}$ is dense in $F$.

\end{lem}

\begin{proof} We mention only that since $A$ is continuous and
surjective, it follows that $A(E_{0})$ is dense in $F$. (We use a
result from \cite{CLBichir_bib_Sir_Topology1983}).

\qquad
\end{proof}

\begin{lem}
\label{lema_izom_corolar_compact} Assume that the hypotheses of
Lemma \ref{lema_izom_corolar} are satisfied. Let $\imath$ be the
injection $E_{0}$ $\subset$ $E$ and let $\jmath$ be the injection
$F_{0}$ $\subset$ $F$. If $\imath$ is compact, then $\jmath$ is
compact.
\end{lem}

\begin{proof} We have $A \circ \imath$ $=$ $\jmath \circ
A_{0}$, so $\jmath$ $=$ $A \circ \imath \circ A_{0}^{-1}$. Since
$\imath$ is compact, it results that $\jmath$ is compact.

\qquad
\end{proof}

\subsection{The Stokes problem}
\label{sectiunea_ecuatii_Stokes_3}

In the sequel, $\textbf{u}$ is the velocity, $p$ is the kinematic
pressure, $\nu$ is the kinematic viscosity and $\textbf{f}$ is a
density of body forces per unit mass.

Consider the following homogeneous Dirichlet problem for the
stationary Stokes equations:
\begin{eqnarray}
   && - \nu \triangle \textbf{u} +
            grad \, p =
            \textbf{f} \; \textrm{in} \; \Omega \, ,
         \label{e1_11_ecStokes} \\
   && div \, \textbf{u} = 0 \; \textrm{in} \; \Omega \, ,
         \label{e1_9_ecStokes} \\
   && \textbf{u} = 0 \; \textrm{on} \; \partial \Omega \, ,
         \label{e1_10_ecStokes}
\end{eqnarray}
under the condition
\begin{equation}
\label{e1_10_ecStokes_ecPoisson_problema_cond2_bis1}
   \int\limits_{\Omega} p \, dx = 0 \, .
\end{equation}

As usual \cite{CLBichir_bib_Bochev_Gunzburger2004,
CLBichir_bib_Gir_Rav1986,
CLBichir_bib_Glowinski_Pironneau1979_Numer_Math}, we take $\nu=1$.
For $\nu \neq 1$, the pressure is $\nu p$ and the free term is
$\nu \textbf{f}$.

In the literature, since $div \, (\triangle \textbf{u})$ $=$
$\triangle (div \, \textbf{u})$ $=$ $0$, the pressure Poisson
equation is obtained,
\begin{equation}
\label{e1_6_ecStokes_ecPoisson0}
   - \triangle p = - div \, \textbf{f} \; \textrm{in} \; \Omega \, .
\end{equation}

Problem (\ref{e1_11_ecStokes}), (\ref{e1_6_ecStokes_ecPoisson0}),
(\ref{e1_10_ecStokes}) is approximated in
\cite{CLBichir_bib_Glowinski_Pironneau1979_Numer_Math}, where it
is introduced the following problem derived from
(\ref{e1_11_ecStokes}) - (\ref{e1_10_ecStokes}): find
$(\textbf{u},\zeta,p)$ $\in$ $\textbf{H}_{0}^{1}(\Omega) \times
H_{0}^{1}(\Omega) \times (H^{1}(\Omega)/\mathbb{R})$ such that
\begin{eqnarray}
   && (grad \, p,grad \, \tilde{\mu})
            = (\textbf{f},grad \, \tilde{\mu}), \
              \forall \, \tilde{\mu}  \, \in \, H_{0}^{1}(\Omega) \, ,
         \label{mathcal_P2_e1_6_ecStokes_ecPoisson_variational_dem_q_compl_p_dem11_cont_alt_izom_G_P_var} \\
   && (grad \, \textbf{u},grad \, \textbf{w})-(p,div \, \textbf{w})=( \textbf{f},\textbf{w} ), \
              \forall \, \textbf{w}  \, \in \, \textbf{H}_{0}^{1}(\Omega) \, ,
         \label{e1_11_ecStokes_ecPoisson_variational_dem_G_P_var} \\
   && (grad \, \zeta,grad \, \bar{\mu})
            - (div \, \textbf{u},\bar{\mu}) = 0, \
              \forall \, \bar{\mu}  \, \in \, H^{1}(\Omega) \, ,
         \label{dem101_Gir_Rav_pag51_G_P_var}
\end{eqnarray}
under the condition that the boundary $\partial \Omega$ is smooth
(such that from $\phi \in H_{0}^{1}(\Omega)$ it can be deduced
that $\phi \in H_{0}^{1}(\Omega) \cap H^{2}(\Omega)$) and
$\textbf{f}$ $\in$ $\textbf{L}^{2}(\Omega)$.
$(\textbf{u},\zeta,p)$ $\in$ $\textbf{H}_{0}^{1}(\Omega) \times
H_{0}^{1}(\Omega) \times (H^{1}(\Omega)/\mathbb{R})$ is also the
solution of the problem
\cite{CLBichir_bib_Glowinski_Pironneau1979_Numer_Math}
\begin{eqnarray}
   && - \triangle p = - div \, \textbf{f} \; \textrm{in} \; \Omega \, .
         \label{e1_6_ecStokes_ecPoisson0_G_P} \\
   && - \triangle \textbf{u} +
            grad \, p =
            \textbf{f} \; \textrm{in} \; \Omega \, ,
         \label{e1_11_ecStokes_G_P} \\
   && \textbf{u} = 0 \; \textrm{on} \; \partial \Omega \, ,
         \label{e1_10_ecStokes_G_P} \\
   && - \triangle \zeta = div \, \textbf{u} \ \textrm{in} \ \Omega \, ,
         \label{dem101_Gir_Rav_pag51_G_P} \\
   && \zeta = 0  \ \textrm{on} \ \partial \Omega \, ,
         \label{dem101_Gir_Rav_pag51_CL_G_P} \\
   && (\textbf{n} \cdot \nabla \zeta)|_{\partial \Omega} = 0  \ \textrm{on} \ \partial \Omega \, .
         \label{dem101_Gir_Rav_pag51_CL_Neumann_G_P}
\end{eqnarray}

Problem (\ref{e1_6_ecStokes_ecPoisson0_G_P}) -
(\ref{dem101_Gir_Rav_pag51_CL_Neumann_G_P}) is decomposed,
equivalently, in some variational problems in
\cite{CLBichir_bib_Gir_Rav1986, CLBichir_bib_Glowinski1984}. We
retain only the equations for pressure: $p$ $\in$ $L^{2}(\Omega)$,
$q$ $\in$ $H_{0}^{1}(\Omega)$, $\hat{p}$ $\in$ $H^{1}(\Omega)$,
\begin{eqnarray}
   && p=q+\hat{p} \, ,
         \label{mathcal_P2_e1_6_ecStokes_ecPoisson_variational_dem_q_compl_hat_p_dem11_cont_alt_izom_CITE} \\
   && (grad \, q,grad \, \bar{\lambda})
            = (\textbf{f},grad \, \bar{\lambda}), \
              \forall \, \bar{\lambda}  \, \in \, H_{0}^{1}(\Omega) \, ,
         \label{mathcal_P2_e1_6_ecStokes_ecPoisson_variational_dem_q_compl_dem11_cont_alt_izom_CITE} \\
   && (grad \, \hat{p},grad \, \tilde{\lambda})
            = 0, \
              \forall \, \tilde{\lambda}  \, \in \, H_{0}^{1}(\Omega) \, ,
         \label{mathcal_P2_e1_6_ecStokes_ecPoisson_variational_dem_q_compl_hat_q_dem11_cont_alt_izom_CITE} \\
   && \hat{p}
            = \lambda_{w} \; \textrm{on} \; \partial \Omega \, ,
         \label{mathcal_P2_e1_6_ecStokes_ecPoisson_variational_dem_q_compl_hat_q_dem11_cont_alt_izom_CITE_CL}
\end{eqnarray}
where $\lambda_{w}$ $\in$ $H^{-1/2}(\partial \Omega)$ is
determined by the methods from \cite{CLBichir_bib_Gir_Rav1986,
CLBichir_bib_Glowinski1984}.

In \cite{CLBichir_bib_Sani_Shen_Pironneau_Gresho2006}, \textit{for
the unsteady case}, for the time dependent incompressible Stokes
problem with Cauchy and Dirichlet data in $\Omega$ over the time
interval $(0,T)$, under the assumption $- \nu \triangle \textbf{u}
+ grad \, p$ $\in$ $\textbf{L}^{2}(\Omega)$, the following
equation is obtained
\begin{eqnarray}
   && (- \nu \triangle \textbf{u} + grad \, p,grad \, \tilde{\mu})
            = (\textbf{f},grad \, \tilde{\mu}), \
              \forall \, \tilde{\mu}  \, \in \, H^{1}(\Omega) \, ,
         \label{mathcal_P2_e1_6_ecStokes_ecPoisson_variational_Sani_Shen_Pironneau_Gresho2006}
\end{eqnarray}
It is the weak form of the equation
(\cite{CLBichir_bib_Sani_Shen_Pironneau_Gresho2006})
\begin{eqnarray}
   && - \nu \, div \, (\triangle \textbf{u}) + \triangle p = div \, \textbf{f} \; \textrm{in} \; \Omega \, ,
         \label{mathcal_P2_e1_6_ecStokes_ecPoisson_clasic_Sani_Shen_Pironneau_Gresho2006}
\end{eqnarray}
which is one of the equations of the consistent pressure Poisson
equation formulation of the mentioned unsteady Stokes problem
\cite{CLBichir_bib_Gresho_Sani1987,
CLBichir_bib_Sani_Shen_Pironneau_Gresho2006}.

According to \cite{CLBichir_bib_De_Fi_Mu2004,
CLBichir_bib_Gresho_Sani1987,
CLBichir_bib_Sani_Shen_Pironneau_Gresho2006}, the problem
(\ref{e1_11_ecStokes}), (\ref{e1_9_ecStokes}),
(\ref{e1_10_ecStokes}), (\ref{e1_9_ecStokes_BC}) is equivalent to
the problem (\ref{e1_11_ecStokes}),
(\ref{e1_6_ecStokes_ecPoisson0}), (\ref{e1_10_ecStokes}),
(\ref{e1_9_ecStokes_BC}), where
\begin{equation}
\label{e1_9_ecStokes_BC}
   div \, \textbf{u} = 0 \; \textrm{on} \; \partial \Omega \, .
\end{equation}
Condition (\ref{dem101_Gir_Rav_pag51_CL_Neumann_G_P}) is
equivalent to (\ref{e1_9_ecStokes_BC})
\cite{CLBichir_bib_Thomasset1981,
CLBichir_bib_Waters_Fix_Cox2004}.

Let us also retain the Neumann boundary condition for the pressure
\begin{equation}
\label{e1_6ppp_ecStokes_ecPoisson0}
   \frac{\partial \, p}{\partial \, n}|_{\partial \Omega} =
      ((\nu \triangle \textbf{u} + \textbf{f}) \cdot \textbf{n})|_{\partial \Omega} \, ,
\end{equation}
where $\textbf{n}$ is the unit outward normal to $\partial \Omega$
\cite{CLBichir_bib_Brezzi_Douglas1988,
CLBichir_bib_Gresho_Sani1987, CLBichir_bib_Gunzburger1989,
CLBichir_bib_Peyret_Taylor1983, CLBichir_bib_Quarteroni_Valli2008,
CLBichir_bib_Sani_Shen_Pironneau_Gresho2006}.

\begin{rem}
\label{observatia5_omega_domega_cond} We do not use
(\ref{e1_9_ecStokes_BC}) or (\ref{e1_6ppp_ecStokes_ecPoisson0}) or
a Dirichlet boundary condition for the pressure in the paper. We
do not use a boundary unknown $\lambda_{w}$ (as in
(\ref{mathcal_P2_e1_6_ecStokes_ecPoisson_variational_dem_q_compl_hat_q_dem11_cont_alt_izom_CITE_CL}))
or a variational problem related to $\partial \Omega$. We do not
impose condition (\ref{dem101_Gir_Rav_pag51_CL_Neumann_G_P}). This
one arises naturally, in a stage of the proof, related to some
dense subspaces. (See
(\ref{e5_1_cond_ecStokes_ecPoisson_variational_dem_modif3}) and
Remark \ref{observatia5_omega_domega222_DEM}).
\end{rem}

\section{A lemma that reduces the proofs of well-posedness to
dense subspaces}
\label{well_posed_sectiunea_general_theorem_on_the_well_posedness_problema_data}

Consider the equation
\begin{eqnarray}
   && \widehat{\Phi}(\widehat{U}) = \widehat{\mathcal{F}} \, .
         \label{well_posed_e1_11_ecStokes_ecPoisson_variational_dem_q_compl4_ec_mathcal_X_op_E_fi_GEN}
\end{eqnarray}

Using Lemma \ref{lema_izom}, we want to reduce the study of
(\ref{well_posed_e1_11_ecStokes_ecPoisson_variational_dem_q_compl4_ec_mathcal_X_op_E_fi_GEN})
to the study of the equation
\begin{eqnarray}
   && \widehat{\Phi}_{0}(\widehat{U}) = \widehat{\mathcal{F}}_{0} \, ,
         \label{well_posed_e1_11_ecStokes_ecPoisson_variational_dem_q_compl4_ec_mathcal_X_op_fi}
\end{eqnarray}
where $\widehat{\Phi}_{0}$ is a restriction of $\widehat{\Phi}$.

\begin{lem}
\label{lema_deschisa_din_dem_1} (proof of the open mapping theorem
\cite{CLBichir_bib_H_Brezis2010,
CLBichir_bib_Christol_Cot_Marle1997,
CLBichir_bib_Dautray_Lions_vol2_ENGLEZA_1988,
CLBichir_bib_Kreyszig1978, CLBichir_bib_Sir_AF1982,
CLBichir_bib_KYosida1980}) Let $E$ and $F$ be two Banach spaces
and let $f:E \rightarrow F$ be a linear, continuous and surjective
mapping. Then, there exist $\varepsilon > 0$ and $\delta > 0$ such
that $f(B_{E}(0,\varepsilon))$ $\supset$ $B_{F}(0,\delta)$.
\end{lem}

\begin{lem}
\label{lema_deschisa_din_dem} (Lemma  5.4, page 149,
\cite{CLBichir_bib_Christol_Cot_Marle1997}) A continuous linear
mapping $f$ from a normed space $E$ into a normed space $F$ is
open if and only if there exists an open ball $B$ of $E$ such that
$f(B)$ is a subset of $F$ with nonempty interior.
\end{lem}

\begin{lem}
\label{well_posed_lema_deschisa}

Let $\widehat{\Gamma}$, $\widehat{\Sigma}$ be two Banach spaces.
Let $\tau$, $\tau'$ be the (norm) topologies on $\widehat{\Gamma}$
and $\widehat{\Sigma}$ respectively. Consider a continuous linear
mapping $\widehat{\Phi}:\widehat{\Gamma} \rightarrow
\widehat{\Sigma}$.

Let $\widehat{\Gamma}_{0}$ $\subset$ $\widehat{\Gamma}$ be a
linear subspace dense in $\widehat{\Gamma}$ and let
$\widehat{\Sigma}_{0}$ $\subset$ $\widehat{\Sigma}$ be a linear
subspace dense in $\widehat{\Sigma}$ such that
$\widehat{\Phi}(\widehat{\Gamma}_{0})$ $\subseteq$
$\widehat{\Sigma}_{0}$. We take the induced topology $\tau \cap
\widehat{\Gamma}_{0}$ on $\widehat{\Gamma}_{0}$ and the induced
topology $\tau' \cap \widehat{\Sigma}_{0}$ on
$\widehat{\Sigma}_{0}$.

Let $\widehat{\Phi}_{0}:\widehat{\Gamma}_{0} \rightarrow
\widehat{\Sigma}_{0}$, $\widehat{\Phi}_{0}(\widehat{U}_{0})$ $=$
$\widehat{\Phi}(\widehat{U}_{0})$, $\forall$ $\widehat{U}_{0}$
$\in$ $\widehat{\Gamma}_{0}$.

Then, $\widehat{\Phi}_{0}$ is a continuous linear mapping.

Moreover, assume that $\widehat{\Phi}$ is surjective and
$\widehat{\Phi}_{0}$ is an bijection of $\widehat{\Gamma}_{0}$
onto $\widehat{\Sigma}_{0}$. Then, $\widehat{\Phi}_{0}$ is an
isomorphism of $\widehat{\Gamma}_{0}$ onto $\widehat{\Sigma}_{0}$
and $\widehat{\Phi}$ is an isomorphism of $\widehat{\Gamma}$ onto
$\widehat{\Sigma}$.

\end{lem}

\begin{proof}

Let us first prove that $\widehat{\Phi}_{0}$ is a continuous
mapping.

Let $\imath_{0}$ be the injection $\widehat{\Gamma}_{0}$ $\subset$
$\widehat{\Gamma}$ and let $\jmath_{0}$ be the injection
$\widehat{\Sigma}_{0}$ $\subset$ $\widehat{\Sigma}$.

We have the induced topology $\tau \cap \widehat{\Gamma}_{0}$ on
$\widehat{\Gamma}_{0}$ and the induced topology $\tau' \cap
\widehat{\Sigma}_{0}$ on $\widehat{\Sigma}_{0}$, so injections
$\imath_{0}$ and $\jmath_{0}$ are continuous.

Define $\widehat{\Phi}_{1}:\widehat{\Gamma}_{0} \rightarrow
\widehat{\Sigma}$, $\widehat{\Phi}_{1}(\widehat{U}_{0})$ $=$
$\widehat{\Phi}_{0}(\widehat{U}_{0})$ $=$
$\widehat{\Phi}(\widehat{U}_{0})$, $\forall$ $\widehat{U}_{0}$
$\in$ $\widehat{\Gamma}_{0}$.

We obtain: $\widehat{\Phi}_{1}$ $=$ $\widehat{\Phi} \circ
\imath_{\widehat{\Gamma}_{0}}$ , so $\widehat{\Phi}_{1}$ is
continuous. $\widehat{\Phi}_{1}$ $=$ $\jmath_{0} \circ
\widehat{\Phi}_{0}$ , so $\widehat{\Phi}_{0}$ is continuous.

In order to prove that $\widehat{\Phi}_{0}$ is open, we use Lemma
\ref{lema_deschisa_din_dem}.

$\widehat{\Phi}$ is defined on Banach spaces and it is linear,
continuous and surjective, so, by the open mapping theorem
(\cite{CLBichir_bib_H_Brezis2010,
CLBichir_bib_Christol_Cot_Marle1997,
CLBichir_bib_Dautray_Lions_vol2_ENGLEZA_1988,
CLBichir_bib_Kreyszig1978, CLBichir_bib_Sir_AF1982,
CLBichir_bib_KYosida1980}), $\widehat{\Phi}$ is open. From Lemma
\ref{lema_deschisa_din_dem_1}, we have that there exist
$\varepsilon
> 0$ and $\delta > 0$ such that

$\widehat{\Phi}(B_{\widehat{\Gamma}}(0,\varepsilon))$ $\supset$
$B_{\widehat{\Sigma}}(0,\delta)$.

$B_{\widehat{\Sigma}}(0,\delta) \cap \widehat{\Sigma}_{0}$ $=$
$\widehat{\Phi}_{0}(\widehat{\Phi}_{0}^{-1}(B_{\widehat{\Sigma}}(0,\delta)
\cap \widehat{\Sigma}_{0}))$ $=$
$\widehat{\Phi}(\widehat{\Phi}_{0}^{-1}(B_{\widehat{\Sigma}}(0,\delta)
\cap \widehat{\Sigma}_{0}) \cap \widehat{\Gamma}_{0})$ $\subset$
$\widehat{\Phi}(\widehat{\Phi}^{-1}(B_{\widehat{\Sigma}}(0,\delta)
\cap \widehat{\Sigma}_{0}) \cap \widehat{\Gamma}_{0})$ $=$
$\widehat{\Phi}(\widehat{\Phi}^{-1}(B_{\widehat{\Sigma}}(0,\delta))
\cap \widehat{\Phi}^{-1}(\widehat{\Sigma}_{0}) \cap
\widehat{\Gamma}_{0})$ $=$
$\widehat{\Phi}(\widehat{\Phi}^{-1}(B_{\widehat{\Sigma}}(0,\delta))
\cap \widehat{\Gamma}_{0})$, where $\widehat{\Phi}_{0}^{-1}(Y)$ is
the inverse image of the subset $Y$ of $\widehat{\Sigma}_{0}$ with
respect to $\widehat{\Phi}_{0}$.

$\widehat{\Phi}_{0}(\widehat{\Phi}^{-1}(B_{\widehat{\Sigma}}(0,\delta))
\cap \widehat{\Gamma}_{0})$ $=$
$\widehat{\Phi}(\widehat{\Phi}^{-1}(B_{\widehat{\Sigma}}(0,\delta))
\cap \widehat{\Gamma}_{0})$ $\supset$
$B_{\widehat{\Sigma}}(0,\delta) \cap \widehat{\Sigma}_{0}$ $\neq$
$\emptyset$.

$\widehat{\Phi}$ is continuous, so
$\widehat{\Phi}^{-1}(B_{\widehat{\Sigma}}(0,\delta))$ is an open
subset of $\widehat{\Gamma}$.

$\widehat{B}_{0}$ $=$
$\widehat{\Phi}^{-1}(B_{\widehat{\Sigma}}(0,\delta)) \cap
\widehat{\Gamma}_{0}$ is an open subset of $\widehat{\Gamma}_{0}$.
$\widehat{\Phi}_{0}(\widehat{B}_{0})$ $\supset$
$B_{\widehat{\Sigma}}(0,\delta) \cap \widehat{\Sigma}_{0}$ $\neq$
$\emptyset$. $B_{\widehat{\Sigma}}(0,\delta) \cap
\widehat{\Sigma}_{0}$ is a nonempty open subset of
$\widehat{\Sigma}_{0}$.

From Lemma \ref{lema_deschisa_din_dem}, we obtain that
$\widehat{\Phi}_{0}$ is open.

$\widehat{\Phi}_{0}$ is linear, continuous and bijective.
$\widehat{\Phi}_{0}$ is open, so $\widehat{\Phi}_{0}^{-1}$ is
continuous. So $\widehat{\Phi}_{0}$ is an isomorphism of
$\widehat{\Gamma}_{0}$ onto $\widehat{\Sigma}_{0}$.

Applying Lemma \ref{lema_izom}, in the form of Corollary
\ref{corolarul_lema_izom_restr_prel1}, it results that
$\widehat{\Phi}$ is an isomorphism of $\widehat{\Gamma}$ onto
$\widehat{\Sigma}$.

\qquad
\end{proof}

\section{A theorem on the well-posedness of a class of
approximate linear problems}
\label{well_posed_sectiunea_general_theorem_on_the_well_posedness}

An equation equivalent to a form of
(\ref{well_posed_e1_11_ecStokes_ecPoisson_variational_dem_q_compl4_ec_mathcal_X_op_E_fi_GEN})
is introduced. A theorem on the well-posedness of its
approximation is formulated. Some practical conditions are given
in two corollaries.

The general framework from this section is applied in the
following sections. Let $\mathcal{Q}$, $\mathcal{X}$,
$\mathcal{Y}$, $\mathcal{Z}$ be four Banach spaces. Define

$\widehat{\Gamma}$ $=$ $\mathcal{Q}$ $\times$ $\mathcal{X}$,
$\widehat{\Sigma}$ $=$ $\mathcal{Z}$ $\times$ $\mathcal{Y}$,
$\widehat{\Delta}$ $=$ $\mathcal{Z}$ $\times$ $\mathcal{X}$.

Assume that $\mathcal{T}$ is an isomorphism of $\mathcal{Y}$ onto
$\mathcal{X}$. $\widehat{\mathcal{T}}$ $=$
$[\mathcal{J}_{\mathcal{Z}},\mathcal{T}]^{T}$ is an isomorphism of
$\widehat{\Sigma}$ onto $\widehat{\Delta}$ and $\widehat{\Delta}$
$=$ $\mathcal{Z}$ $\times$ $\mathcal{T}(\mathcal{Y})$ $=$
$\widehat{\mathcal{T}}(\widehat{\Sigma})$.

$\mathcal{J}_{\mathcal{Z}}$ is the identity operator on
$\mathcal{Z}$, $\mathcal{I}_{\mathcal{X}}$ is the identity
operator on $\mathcal{X}$. $\mathcal{J}_{\mathcal{Z}}$ and
$\mathcal{I}_{\mathcal{X}}$ are automorphisms. (The domain and the
range are endowed with the same norm in each case).

Assume that $\widehat{\Phi}$ is an isomorphism of
$\widehat{\Gamma}$ onto $\widehat{\Sigma}$ and $\widehat{\Phi}$
has a formulation which is deducible from the definition of the
operator $\widehat{\mathcal{A}}$ below. In the present context, we
denote $\widehat{U}$ $=$ $(x,U)$ $\in$ $\widehat{\Gamma}$ and we
consider equation
(\ref{well_posed_e1_11_ecStokes_ecPoisson_variational_dem_q_compl4_ec_mathcal_X_op_E_fi_GEN})
in the following formulation
\begin{eqnarray}
   && \widehat{\Phi}(x,U) = \widehat{\mathcal{F}} \, ,
         \label{well_posed_e1_11_ecStokes_ecPoisson_variational_dem_q_compl4_ec_mathcal_X_op_E_fi}
\end{eqnarray}

Let us introduce the operators $\mathcal{G}$ $\in$ $L(\mathcal{Q}
\times \mathcal{X},\mathcal{Y})$ and $\mathcal{B}$ $\in$
$L(\mathcal{Q} \times \mathcal{X},\mathcal{Z})$ and let
$\widehat{\mathcal{A}}$ $\in$
$L(\widehat{\Gamma},\widehat{\Delta})$ be defined by
$\widehat{\mathcal{A}}$ $=$
$\widehat{\mathcal{T}}\widehat{\Phi}$ $=$ \\
$[\mathcal{J}_{\mathcal{Z}},\mathcal{T}]^{T}[\mathcal{B},\Phi]^{T}$
$=$ $[\mathcal{J}_{\mathcal{Z}}\mathcal{B},\mathcal{T}\Phi]^{T}$
$=$
$[\mathcal{B},\mathcal{I}_{\mathcal{X}}-\mathcal{T}\mathcal{G}]^{T}$
$=$ $[\mathcal{B},\mathcal{A}]^{T}$ \\
or $\widehat{\mathcal{A}}(x,U)$ $=$
$[\mathcal{B}(x,U),U-\mathcal{T}\mathcal{G}(x,U)]^{T}$ $=$
$[\mathcal{B}(x,U),\mathcal{A}(x,U)]^{T}$.

It results that $\widehat{\mathcal{A}}$ is an isomorphism of
$\widehat{\Gamma}$ onto $\widehat{\Delta}$.

$\widehat{\Phi}(\widehat{\Gamma})$ $=$ $\widehat{\Sigma}$,
$\widehat{\mathcal{A}}(\widehat{\Gamma})$ $=$
$\widehat{\mathcal{T}}\widehat{\Phi}(\widehat{\Gamma})$ $=$
$\widehat{\mathcal{T}}(\widehat{\Sigma})$ $=$ $\widehat{\Delta}$
$=$ $\mathcal{Z}$ $\times$ $\mathcal{T}(\mathcal{Y})$.

$\widehat{\Gamma}$ $=$
$\widehat{\mathcal{A}}^{-1}\widehat{\mathcal{T}}(\widehat{\Sigma})$
$=$ $\widehat{\mathcal{A}}^{-1}\widehat{\Delta}$ $=$
$\widehat{\mathcal{A}}^{-1}(\mathcal{Z} \times
\mathcal{T}(\mathcal{Y}))$ $=$
$\widehat{\mathcal{A}}^{-1}(\mathcal{Z},\mathcal{T}(\mathcal{Y}))$.

Let $\mathcal{X}_{h}$ be a closed subspace of $\mathcal{X}$.
Define

$\widehat{\Gamma}_{h}$ $=$ $\mathcal{Q}$ $\times$
$\mathcal{X}_{h}$ $\subset$ $\mathcal{Q}$ $\times$ $\mathcal{X}$
$=$ $\widehat{\Gamma}$, $\widehat{\Delta}_{h}$ $=$ $\mathcal{Z}$
$\times$ $\mathcal{X}_{h}$ $\subset$ $\mathcal{Z}$ $\times$
$\mathcal{X}$ $=$ $\widehat{\Delta}$.

Let $\mathcal{T}_{h}$ $\in$ $L(\mathcal{Y},\mathcal{X}_{h})$,
$\mathcal{X}_{h}$ $=$ $\mathcal{T}_{h}(\mathcal{Y})$,
$\widehat{\mathcal{T}}_{h}$ $\in$ $L(\widehat{\Sigma},\mathcal{Z}
\times \mathcal{X}_{h})$, $\widehat{\mathcal{T}}_{h}$ $=$
$[\mathcal{J}_{\mathcal{Z}},\mathcal{T}_{h}]^{T}$.

Let $\widehat{\mathcal{K}}$ $\in$
$L(\widehat{\Gamma},\widehat{\Delta})$, $\widehat{\mathcal{K}}$
$=$
$[\mathcal{B},\mathcal{I}_{\mathcal{X}}-\mathcal{T}_{h}\mathcal{G}]^{T}$
$=$ $[\mathcal{B},\mathcal{K}]^{T}$ or
$\widehat{\mathcal{K}}(x,U)$ $=$
$[\mathcal{B}(x,U),U-\mathcal{T}_{h}\mathcal{G}(x,U)]^{T}$ $=$
$[\mathcal{B}(x,U),\mathcal{K}(x,U)]^{T}$.

Observe that we have $Range(\mathcal{K})$ $\subseteq$
$Range(\mathcal{A})$ $=$ $\mathcal{X}$ $=$
$\mathcal{T}(\mathcal{Y})$.

We denote $\widehat{\mathcal{F}}$ $=$
$[\mathcal{F}_{\mathcal{B}},\mathcal{F}]^{T}$ $\in$
$\widehat{\Sigma}$.

Define $\widehat{\mathcal{A}}_{h}:\widehat{\Gamma}_{h} \rightarrow
\widehat{\Delta}_{h}$, $\widehat{\mathcal{A}}_{h}$ $\in$
$L(\widehat{\Gamma}_{h},\widehat{\Delta}_{h})$,
$\widehat{\mathcal{A}}_{h}$ $=$
$[\mathcal{B},\mathcal{A}_{h}]^{T}$ or
$\widehat{\mathcal{A}}_{h}(x_{h},U_{h})$ $=$
$[\mathcal{B}(x_{h},U_{h}),U_{h}-\mathcal{T}_{h}\mathcal{G}(x_{h},U_{h})]^{T}$
$=$ $[\mathcal{B}(x_{h},U_{h}),\mathcal{A}_{h}(x_{h},U_{h})]^{T}$,
$\mathcal{A}_{h}$ $\in$ $L(\widehat{\Gamma}_{h},\mathcal{X}_{h})$,
$\mathcal{A}_{h}(x_{h},U_{h})$ $=$
$U_{h}-\mathcal{T}_{h}\mathcal{G}(x_{h},U_{h})$,

Equation
(\ref{well_posed_e1_11_ecStokes_ecPoisson_variational_dem_q_compl4_ec_mathcal_X_op_E_fi})
is equivalent to the following equation
\begin{eqnarray}
   && \begin{array}{c}
      \mathcal{B}(x,U) = \mathcal{J}_{\mathcal{Z}}\mathcal{F}_{\mathcal{B}} \, , \\
      U-\mathcal{T}\mathcal{G}(x,U) = \mathcal{T}\mathcal{F} \, ,
      \end{array}
   \qquad \textrm{or} \qquad
      \widehat{\mathcal{A}}(x,U) = \widehat{\mathcal{T}}\widehat{\mathcal{F}} \, ,
         \label{well_posed_e1_11_ecStokes_ecPoisson_variational_dem_q_compl4_ec_mathcal_X_op_E}
\end{eqnarray}

Given the conforming space $\mathcal{X}_{h}$, approximate
(\ref{well_posed_e1_11_ecStokes_ecPoisson_variational_dem_q_compl4_ec_mathcal_X_op_E})
by
\begin{eqnarray}
   && \begin{array}{c}
      \mathcal{B}(x_{h},U_{h}) = \mathcal{J}_{\mathcal{Z}}\mathcal{F}_{\mathcal{B}} \, , \\
      U_{h}-\mathcal{T}_{h}\mathcal{G}(x_{h},U_{h}) = \mathcal{T}_{h}\mathcal{F} \, ,
      \end{array}
   \qquad \textrm{or} \qquad
      \widehat{\mathcal{A}}_{h}(x_{h},U_{h}) = \widehat{\mathcal{T}}_{h}\widehat{\mathcal{F}} \, ,
         \label{well_posed_e1_11_ecStokes_ecPoisson_variational_dem_q_compl4_ec_mathcal_X_h_E}
\end{eqnarray}

\begin{lem}
\label{well_posed_lema_Sir_Spaces_AF1982_Cor1}
(\cite{CLBichir_bib_Sir_Spaces_AF1982}) Let $E$, $F$ be two Banach
spaces and let $T,S \in L(E,F)$. If the operator $T$ is bijective
and $\| T^{-1} \|_{L(F,E)} \| T - S \|_{L(E,F)}$ $<$ $1$, then the
operator $S$ is bijective and $\| S^{-1} \|_{L(F,E)}$ $\leq$ $(1 -
q)^{-1}$ $\| T^{-1} \|_{L(F,E)}$, $\forall q \in \mathbb{R}$ that
satisfies $\| T^{-1} \|_{L(F,E)} \| T - S \|_{L(E,F)}$ $\leq$ $q$
$<$ $1$.
\end{lem}

In order to use Theorem XIV.1.1, \cite{CLBichir_bib_Ka_Ak1986},
and the proof of Theorem IV.3.8, \cite{CLBichir_bib_Gir_Rav1986},
to relate equations
(\ref{well_posed_e1_11_ecStokes_ecPoisson_variational_dem_q_compl4_ec_mathcal_X_op_E})
and
(\ref{well_posed_e1_11_ecStokes_ecPoisson_variational_dem_q_compl4_ec_mathcal_X_h_E}),
we must reformulate these theorems into the following Theorem
\ref{well_posed_teorema5_1_Stokes_math_B_L2}. The proof is an
adaptation, to our conditions, of the proofs of these mentioned
theorems from \cite{CLBichir_bib_Ka_Ak1986} and
\cite{CLBichir_bib_Gir_Rav1986}.

\begin{thm}
\label{well_posed_teorema5_1_Stokes_math_B_L2}

Consider the framework introduced above and the isomorphism
$\widehat{\mathcal{A}}$ of $\widehat{\Gamma}$ onto
$\widehat{\Delta}$. If
\begin{equation}
\label{well_posed_e5_45p_conditia_q}
   q_{\ast} = \| (\mathcal{T} - \mathcal{T}_{h})\mathcal{G} \|_{L(\widehat{\Gamma},\mathcal{X})}
      \| \widehat{\mathcal{A}}^{-1} \|_{L(\widehat{\Delta},\widehat{\Gamma})}
      < 1 \, ,
\end{equation}
then there exists $\widehat{\Gamma}_{\ast}$,
$\widehat{\Gamma}_{\ast}$ $\subseteq$ $\widehat{\Gamma}_{h}$,
$\widehat{\Gamma}_{\ast}$ $\neq$ $\emptyset$,
$\widehat{\Gamma}_{\ast}$ closed in $\widehat{\Gamma}_{h}$, such
that the restriction $\widehat{\mathcal{A}}_{\ast}$ of
$\widehat{\mathcal{A}}_{h}$ to $\widehat{\Gamma}_{\ast}$,
$\widehat{\mathcal{A}}_{\ast}$ $=$
$\widehat{\mathcal{A}}_{h}|_{\widehat{\Gamma}_{\ast}}$, is an
isomorphism of $\widehat{\Gamma}_{\ast}$ onto
$\widehat{\Delta}_{h}$. We have $\widehat{\Gamma}_{\ast}$ $=$
$\widehat{\mathcal{K}}^{-1}(\mathcal{Z} \times
\mathcal{T}_{h}(\mathcal{Y}))$ $=$
$\widehat{\mathcal{K}}^{-1}(\mathcal{Z} \times \mathcal{X}_{h})$
$=$ $\widehat{\mathcal{K}}^{-1}(\widehat{\Delta}_{h})$ and
\begin{equation}
\label{well_posed_e5_45p_conditia_q_estimare}
   \| \widehat{\mathcal{A}}_{\ast}^{-1} \|_{L(
      \widehat{\Delta}_{h},\widehat{\Gamma}_{h})}
   \leq
   (1-q_{\ast})^{-1}
   \| \widehat{\mathcal{A}}^{-1} \|_{L(\widehat{\Delta},\widehat{\Gamma})} \, ,
\end{equation}

Moreover, the following results hold. Let $\widehat{\mathcal{F}}$
$\in$ $\widehat{\Sigma}$. Then, there exists an unique solution
$(x_{h},U_{h})$ of
(\ref{well_posed_e1_11_ecStokes_ecPoisson_variational_dem_q_compl4_ec_mathcal_X_h_E})
that approximates the unique solution $(x,U)$ of
(\ref{well_posed_e1_11_ecStokes_ecPoisson_variational_dem_q_compl4_ec_mathcal_X_op_E}).
We have
\begin{equation}
\label{well_posed_e5_45p_U_Uh_eroare}
   \left\| \left[\begin{array}{c}
      x \\
      U
      \end{array}\right]
      -
      \left[\begin{array}{c}
      x_{h} \\
      U_{h}
      \end{array}\right] \right\|_{\widehat{\Gamma}}
   \leq
   (1-q_{\ast})^{-1}
   \| \widehat{\mathcal{A}}^{-1} \|_{L(\widehat{\Delta},\widehat{\Gamma})}
   \| (\mathcal{T} - \mathcal{T}_{h})\mathcal{T}^{-1}U \|_{\mathcal{X}} \, ,
\end{equation}

\end{thm}

\begin{proof}

We have $\| \widehat{\mathcal{A}} - \widehat{\mathcal{K}}
\|_{L(\widehat{\Gamma},\widehat{\Delta})} \|
\widehat{\mathcal{A}}^{-1}
\|_{L(\widehat{\Delta},\widehat{\Gamma})}$ $=$ $q_{\ast}$ $<$ $1$.
Applying Lemma \ref{well_posed_lema_Sir_Spaces_AF1982_Cor1}, since
$\widehat{\mathcal{A}}$ is an isomorphism of $\widehat{\Gamma}$
onto $\widehat{\Delta}$, it results that $\widehat{\mathcal{K}}$
is an isomorphism of $\widehat{\Gamma}$ onto $\widehat{\Delta}$
and
\begin{equation}
\label{well_posed_trei_ast_CONT}
   \| \widehat{\mathcal{K}}^{-1} \|_{L(\widehat{\Delta},\widehat{\Gamma})}
   \leq
   (1-q_{\ast})^{-1}
   \| \widehat{\mathcal{A}}^{-1} \|_{L(\widehat{\Delta},\widehat{\Gamma})} \, ,
\end{equation}

In the space $L(\widehat{\Gamma}_{h},\widehat{\Delta}_{h})$, let
us consider the operator $\widehat{\mathcal{A}}_{h}$.
($\widehat{\mathcal{A}}_{h}$ $\in$
$L(\widehat{\Gamma}_{h},\widehat{\Delta}_{h})$). Observe that we
have: $\forall$ $(x_{h},U_{h})$ $\in$ $\widehat{\Gamma}_{h}$,
$\widehat{\mathcal{A}}_{h}(x_{h},U_{h})$ $=$
$\widehat{\mathcal{K}}(x_{h},U_{h})$.

Let $\widehat{\Gamma}_{\ast}$ $=$
$\widehat{\mathcal{K}}^{-1}(\mathcal{Z} \times
\mathcal{T}_{h}(\mathcal{Y}))$ $=$
$\widehat{\mathcal{K}}^{-1}(\mathcal{Z} \times \mathcal{X}_{h})$
$=$ $\widehat{\mathcal{K}}^{-1}(\widehat{\Delta}_{h})$. If
$\eta_{\mathcal{B}}$ $\in$ $\mathcal{Z}$, $\eta_{h}$ $\in$
$\mathcal{T}_{h}(\mathcal{Y})$ and if $(V',V_{h})$ $=$
$\widehat{\mathcal{K}}^{-1}(\eta_{\mathcal{B}},\eta_{h})$ $\in$
$\widehat{\Gamma}$ $=$ $\mathcal{Q}$ $\times$ $\mathcal{X}$, then
$\widehat{\mathcal{K}}(V',V_{h})$ $=$
$(\eta_{\mathcal{B}},\eta_{h})$ or $\mathcal{B}(V',V_{h})$ $=$
$\eta_{\mathcal{B}}$, $V_{h}-\mathcal{T}_{h}\mathcal{G}V_{h}$ $=$
$\eta_{h}$ or $\mathcal{B}(V',V_{h})$ $=$ $\eta_{\mathcal{B}}$,
$V_{h}$ $=$ $\mathcal{T}_{h}\mathcal{G}V_{h}$ $+$ $\eta_{h}$, so
$V_{h}$ $\in$ $\mathcal{X}_{h}$. It follows that $V_{h}$ $\in$
$\mathcal{X}_{h}$ and, since $V'$ $\in$ $\mathcal{Q}$, we obtain
$(V',V_{h})$ $\in$ $\widehat{\Gamma}_{h}$ $=$ $\mathcal{Q}$
$\times$ $\mathcal{X}_{h}$ or
$\widehat{\mathcal{K}}^{-1}(\eta_{\mathcal{B}},\eta_{h})$ $\in$
$\widehat{\Gamma}_{h}$ $=$ $\mathcal{Q}$ $\times$
$\mathcal{X}_{h}$. It follows that $\widehat{\Gamma}_{\ast}$
$\subseteq$ $\widehat{\Gamma}_{h}$ $=$ $\mathcal{Q}$ $\times$
$\mathcal{X}_{h}$.

$\widehat{\mathcal{K}}$ is an isomorphism, so
$\widehat{\Gamma}_{\ast}$ is closed in $\widehat{\Gamma}_{h}$ $=$
$\mathcal{Q}$ $\times$ $\mathcal{X}_{h}$.

Since $\forall$ $(x_{h},U_{h})$ $\in$ $\widehat{\Gamma}_{h}$, we
have $\widehat{\mathcal{A}}_{h}(x_{h},U_{h})$ $=$
$\widehat{\mathcal{K}}(x_{h},U_{h})$, it follows that the
restriction $\widehat{\mathcal{A}}_{\ast}$ of
$\widehat{\mathcal{A}}_{h}$ to $\widehat{\Gamma}_{\ast}$,
$\widehat{\mathcal{A}}_{\ast}$ $=$
$\widehat{\mathcal{A}}_{h}|_{\widehat{\Gamma}_{\ast}}$, is an
isomorphism of $\widehat{\Gamma}_{\ast}$ onto
$\widehat{\Delta}_{h}$.

Hence it follows that the operator $\widehat{\mathcal{A}}_{\ast}$
has a continuous inverse that coincides with
$\widehat{\mathcal{K}}^{-1}$ on $\widehat{\Delta}_{h}$ and
\begin{equation}
\label{well_posed_trei_ast_RESTRICTIE}
   \| \widehat{\mathcal{A}}_{\ast}^{-1} \|_{L(
      \widehat{\Delta}_{h},\widehat{\Gamma}_{h})}
   \leq
   \| \widehat{\mathcal{K}}^{-1} \|_{L(\widehat{\Delta},\widehat{\Gamma})} \, .
\end{equation}

Relation (\ref{well_posed_e5_45p_conditia_q_estimare}) results
from (\ref{well_posed_trei_ast_RESTRICTIE}) and
(\ref{well_posed_trei_ast_CONT}).

Solutions $(x,U)$ and $(x_{h},U_{h})$ are given by the
isomorphisms $\widehat{\mathcal{A}}$ and
$\widehat{\mathcal{A}}_{\ast}$ respectively. Given
$\widehat{\mathcal{F}}$, $(x_{h},U_{h})$ approximates $(x,U)$.

Viewing the identity operator and the projection operator from the
proof of an identity from Theorem 12.1.2, page 479,
\cite{CLBichir_bib_Atkinson_Han2009}, as $\mathcal{T}$ and
$\mathcal{T}_{h}$ from our case, we easy obtain
(\ref{well_posed_e5_45p_U_Uh_eroare}): We have

$\mathcal{T}^{-1}U-\mathcal{G}(x,U) = \mathcal{F}$,

$\mathcal{T}_{h}\mathcal{T}^{-1}U-\mathcal{T}_{h}\mathcal{G}(x,U)
= \mathcal{T}_{h}\mathcal{F}$,

$U_{h}$ $=$ $\mathcal{T}_{h}\mathcal{G}(x_{h},U_{h}) +
\mathcal{T}_{h}\mathcal{F}$ $=$
$\mathcal{T}_{h}\mathcal{G}(x_{h},U_{h}) +
\mathcal{T}_{h}\mathcal{T}^{-1}U -
\mathcal{T}_{h}\mathcal{G}(x,U)$

$\mathcal{B}(x,U) - \mathcal{B}(x_{h},U_{h})$ $=$
$\mathcal{B}(x,U) - \mathcal{F}_{\mathcal{B}}$ $=$ $0$

$\mathcal{K}((x,U) - (x_{h},U_{h}))$ $=$ $U -
\mathcal{T}_{h}\mathcal{G}(x,U) - (U_{h} -
\mathcal{T}_{h}\mathcal{G}(x_{h},U_{h}))$ $=$ $U -
\mathcal{T}_{h}\mathcal{T}^{-1}U $ $=$ $(\mathcal{T} -
\mathcal{T}_{h})\mathcal{T}^{-1}U$

$\left[\begin{array}{c}
      \mathcal{B}(x,U) - \mathcal{B}(x_{h},U_{h}) \\
      \mathcal{K}(x,U) - \mathcal{K}(x_{h},U_{h})
      \end{array}\right]$
$=$ $\left[\begin{array}{c}
      0 \\
      (\mathcal{T} - \mathcal{T}_{h})\mathcal{T}^{-1}U
      \end{array}\right] \, .$

Hence

$\left[\begin{array}{c}
      x \\
      U
      \end{array}\right]
      -
      \left[\begin{array}{c}
      x_{h} \\
      U_{h}
      \end{array}\right]$
$=$ $\widehat{\mathcal{K}}^{-1}\left[\begin{array}{c}
      0 \\
      (\mathcal{T} - \mathcal{T}_{h})\mathcal{T}^{-1}U
      \end{array}\right] \, .$

From this identity, using (\ref{well_posed_trei_ast_CONT}), we
deduce (\ref{well_posed_e5_45p_U_Uh_eroare}).

\qquad
\end{proof}

\begin{cor}
\label{well_posed_teorema5_1_Stokes_math_B_L2_cor_dense}

Let $\widehat{\Gamma}_{r}$ be a linear subspace dense in
$\widehat{\Gamma}$. Then, (\ref{well_posed_e5_45p_conditia_q}) has
the form
\begin{equation}
\label{well_posed_e5_45p_conditia_q_cor}
   q_{\ast} = \| (\mathcal{T} - \mathcal{T}_{h})\mathcal{G} \|_{L(\widehat{\Gamma}_{r},\mathcal{X})}
      \| \widehat{\mathcal{A}}^{-1} \|_{L(\widehat{\Delta},\widehat{\Gamma})}
      < 1 \, ,
\end{equation}
\end{cor}

Thinking to $\mathcal{T}$ as related to the solution operators of
an elliptic problems, let us introduce some additional hypotheses
that result from the approximation of elliptic problems
(conditions for the Galerkin method and the conclusion of C\'ea's
lemma \cite{CLBichir_bib_Atkinson_Han2009,
CLBichir_bib_Ciarlet2002, CLBichir_bib_Quarteroni_Valli2008}). In
this way, the convergence of Galerkin method is sufficient.

\begin{cor}
\label{well_posed_corolarul_teorema5_1_Stokes_math_B_L2_Tdense}
Assume that the hypotheses of Theorem
\ref{well_posed_teorema5_1_Stokes_math_B_L2} are satisfied. Let
$\mathcal{X}_{d}$ be a linear subspace dense in $\mathcal{X}$.
Assume that the embedding $\mathcal{X}_{d} \subseteq \mathcal{X}$
is continuous and the embedding operator $\imath_{0}$ is compact.
Assume that
\begin{equation}
\label{well_posed_e5_45p_conditia_q_restrictie_dense_lim_T}
   \lim_{h \rightarrow 0} \| (\mathcal{T}-\mathcal{T}_{h}) \Psi \|_{\mathcal{X}}=0 \, ,
      \ \forall \, \Psi \in \mathcal{Y} \, ,
\end{equation}
Then, $\exists$ $h_{0}$ such that $\forall$ $h < h_{0}$, $q_{\ast}
< 1$, that is, (\ref{well_posed_e5_45p_conditia_q}). Furthermore,
\begin{equation}
\label{well_posed_e5_45p_U_Uh_eroare_U_Uh_lim_y_yh_lim}
   \lim_{h \rightarrow 0} \| U - U_{h} \|_{\mathcal{X}}=0 \, , \
   \lim_{h \rightarrow 0} \| x - x_{h} \|_{\mathcal{Q}}=0 \, ,
\end{equation}
\end{cor}

\begin{proof}

$\mathcal{T}^{-1}$ is an isomorphism of $\mathcal{X}$ onto
$\mathcal{Y}$.

We apply \ref{lema_izom_corolar} with $E$ $=$ $\mathcal{X}$, $F$
$=$ $\mathcal{Y}$, $A$ $=$ $\mathcal{T}^{-1}$, $E_{0}$ $=$
$\mathcal{X}_{d}$. It results $A_{0}$ $=$
$\mathcal{T}_{d,isom}^{-1}$ $=$
$(\mathcal{T}^{-1})|_{\mathcal{X}_{d}}$ is an isomorphism of
$E_{0}$ onto $F_{0}$ $=$ $\mathcal{Y}_{d}$ $=$
$\mathcal{T}^{-1}(\mathcal{X}_{d})$ and $\mathcal{Y}_{d}$ is dense
in $\mathcal{Y}$. $\mathcal{T}_{d,isom}$ $\in$
$L(\mathcal{Y}_{d},\mathcal{X}_{d})$,
$\mathcal{T}_{d,isom}^{-1}(U)$ $=$
$(\mathcal{T}^{-1})|_{\mathcal{X}_{d}}(U)$, $\forall$ $U$ $\in$
$\mathcal{X}_{d}$.

Let $\jmath_{0}$ be the injection $\mathcal{Y}_{d}$ $\subset$
$\mathcal{Y}$. We now apply Lemma \ref{lema_izom_corolar_compact}
with $E$ $=$ $\mathcal{X}$, $F$ $=$ $\mathcal{Y}$, $A$ $=$
$\mathcal{T}^{-1}$, $E_{0}$ $=$ $\mathcal{X}_{d}$, $F_{0}$ $=$
$\mathcal{Y}_{d}$, $A_{0}$ $=$ $\mathcal{T}_{d,isom}^{-1}$,
$\imath$ $=$ $\imath_{0}$, $\jmath$ $=$ $\jmath_{0}$. So
$\jmath_{0}$ is compact. $\jmath_{0}$ is also continuous. (This
follows from the proof of Lemma \ref{lema_izom_corolar_compact}).

Let $\mathcal{T}_{d}$ be the restriction of $\mathcal{T}$ to
$\mathcal{Y}_{d}$, $\mathcal{T}_{d}$ $\in$
$L(\mathcal{Y}_{d},\mathcal{X})$. We have
$\mathcal{T}_{d,isom}(\mathcal{F})$ $=$
$\mathcal{T}_{d}(\mathcal{F})$, $\forall$ $\mathcal{F}$ $\in$
$\mathcal{Y}_{d}$. Let $\mathcal{T}_{d,h}$ be the restriction of
$\mathcal{T}_{h}$ to $\mathcal{Y}_{d}$. $\mathcal{T}_{d,h}$ $\in$
$L(\mathcal{Y}_{d},\mathcal{X}_{h})$.

From Lemma \ref{lema_prelungire}, it follows that
\begin{equation}
\label{well_posed_e5_45p_conditia_q_restrictie_dense_EG_T}
   \| \mathcal{T} - \mathcal{T}_{h} \|_{L(\mathcal{Y},\mathcal{X})}
      = \| \mathcal{T}_{d} - \mathcal{T}_{d,h} \|_{L(\mathcal{Y}_{d},\mathcal{X})} \, ,
\end{equation}
where $\mathcal{T} - \mathcal{T}_{h}$ is the unique extension of
$\mathcal{T}_{d} - \mathcal{T}_{d,h}$. We have
\begin{equation}
\label{well_posed_e5_45p_conditia_q_restrictie_dense_EG_T_compunere_NOU2}
   \mathcal{T}_{d} - \mathcal{T}_{d,h}
      = ( \mathcal{T} - \mathcal{T}_{h} )|_{\mathcal{Y}_{d}}
      = ( \mathcal{T} - \mathcal{T}_{h} ) \circ \jmath_{0} \, ,
\end{equation}

We can now apply the end of Remark IV.3.4, page 306,
\cite{CLBichir_bib_Gir_Rav1986}, and also page 33,
\cite{CLBichir_bib_Nair2009}. This and
(\ref{well_posed_e5_45p_conditia_q_restrictie_dense_lim_T}) give
(\ref{well_posed_e5_45p_conditia_q_restrictie_dense_lim_T_UNIF}),
where
\begin{equation}
\label{well_posed_e5_45p_conditia_q_restrictie_dense_lim_T_UNIF}
   \lim_{h \rightarrow 0} \| \mathcal{T}_{d} - \mathcal{T}_{d,h} \|_{L(\mathcal{Y}_{d},\mathcal{X})}=0 \, ,
\end{equation}

Taking into account
(\ref{well_posed_e5_45p_conditia_q_restrictie_dense_EG_T}), we
deduce where
\begin{equation}
\label{well_posed_e5_45p_conditia_q_restrictie_dense_lim_T_UNIF_intreg}
   \lim_{h \rightarrow 0} \| \mathcal{T} - \mathcal{T}_{h} \|_{L(\mathcal{Y},\mathcal{X})}=0 \, ,
\end{equation}
and $\| (\mathcal{T} - \mathcal{T}_{h})\mathcal{G}
\|_{L(\widehat{\Gamma},\mathcal{X})}$ $\leq$ $\| \mathcal{T} -
\mathcal{T}_{h} \|_{L(\mathcal{Y},\mathcal{X})} \| \mathcal{G}
\|_{L(\widehat{\Gamma},\mathcal{Y})} \, .$

Using
(\ref{well_posed_e5_45p_conditia_q_restrictie_dense_lim_T_UNIF_intreg}),
we obtain that $\exists$ $h_{0}$ such that $\forall$ $h < h_{0}$,
$q_{\ast} < 1$, that is, (\ref{well_posed_e5_45p_conditia_q}). We
have $\| (\mathcal{T} - \mathcal{T}_{h})\mathcal{T}^{-1}U
\|_{\mathcal{X}}$ $\leq$ $\| \mathcal{T} - \mathcal{T}_{h}
\|_{L(\mathcal{Y},\mathcal{X})} \| \mathcal{T}^{-1}U
\|_{\mathcal{Y}} \, .$

We use (\ref{well_posed_e5_45p_U_Uh_eroare}). Finally,
(\ref{well_posed_e5_45p_conditia_q_restrictie_dense_lim_T_UNIF_intreg})
leads to (\ref{well_posed_e5_45p_U_Uh_eroare_U_Uh_lim_y_yh_lim}).

\qquad
\end{proof}

\begin{cor}
\label{well_posed_corolarul_teorema5_1_restrictie_dense} Assume
that $\mathcal{X}$, $\mathcal{Q}$ and $\mathcal{Z}$ are Hilbert
spaces. Assume that the hypotheses of Theorem
\ref{well_posed_teorema5_1_Stokes_math_B_L2} are satisfied. Let
$h$ be a parameter. Let $\{\mathcal{X}_{h} ; h > 0 \}$ be a family
of finite dimensional subspaces of $\mathcal{X}$. Let
$\mathcal{S}$ be a set dense in $\mathcal{X}$ and assume that
\begin{equation}
\label{well_posed_e5_45p_conditia_q_dense_general_GAMA0_elliptic_restrictie}
   \lim_{h \rightarrow 0} \inf_{V_{h} \in \mathcal{X}_{h}} \| V-V_{h} \|_{\mathcal{X}}=0 \, ,
      \ \forall \, V \in \mathcal{S} \, ,
\end{equation}
Assume that there exists $\gamma_{0} > 0$ such that for all $\Psi$
$\in$ $\mathcal{Y}$ and $W=\mathcal{T}\Psi$,
$W_{h}=\mathcal{T}_{h}\Psi$, we have
\begin{equation}
\label{well_posed_e5_45p_conditia_Cea_GAMA0_ini_restrictie}
   \| W-W_{h} \|_{\mathcal{X}} \leq \gamma_{0} \inf_{V_{h} \in \mathcal{X}_{h}}\| W-V_{h} \|_{\mathcal{X}} \, ,
\end{equation}

Then, (\ref{well_posed_e5_45p_conditia_q_restrictie_dense_lim_T})
holds.

\end{cor}

\begin{proof}

Let us fix $\Psi$, $\Psi$ $\in$ $\mathcal{Y}$.

\begin{equation}
\label{well_posed_e5_45p_conditia_Cea_GAMA0_ini_restrictie_CONT_0}
   \| (\mathcal{T}-\mathcal{T}_{h}) \Psi \|_{\mathcal{X}}
      = \| W-W_{h} \|_{\mathcal{X}} \, ,
\end{equation}

The arguments from the proof of the convergence of Galerkin method
for the elliptic problems (\cite{CLBichir_bib_Atkinson_Han2009,
CLBichir_bib_Ciarlet2002, CLBichir_bib_Quarteroni_Valli2008}) lead
to (\ref{well_posed_e5_45p_conditia_q_restrictie_dense_lim_T}).
$W$ corresponds to the solution of exact variational equation,
$W_{h}$ corresponds to the solution of approximate variational
equation and
(\ref{well_posed_e5_45p_conditia_Cea_GAMA0_ini_restrictie})
corresponds to the C\'ea's lemma.

\qquad
\end{proof}

\section{Preliminary remarks for the formulation of the problem we study}
\label{sectiunea_ecuatii_remarks}

\subsection{Some remarks on some operators and subspaces}
\label{sectiunea_ecuatii_remarks_1}

Related to the Stokes problem, we introduce an extended system in
Section \ref{sectiunea_forme_echivalente_ale_ecuatiilor_extended}.
This extended system has the forms
(\ref{well_posed_e1_11_ecStokes_ecPoisson_variational_dem_q_compl4_ec_mathcal_X_op_E_fi_GEN})
and
(\ref{well_posed_e1_11_ecStokes_ecPoisson_variational_dem_q_compl4_ec_mathcal_X_op_E_fi}).
In order to apply Lemma \ref{well_posed_lema_deschisa} to the
study, we define now some adequate spaces $\widehat{\Gamma}$,
$\widehat{\Sigma}$, $\widehat{\Gamma}_{0}$ and
$\widehat{\Sigma}_{0}$. The results are obtained for $(u,p)$ $\in$
$\textbf{H}_{0}^{1}(\Omega)$ $\times$ $(H^{1}(\Omega) \cap
L_{0}^{2}(\Omega))$ and these spaces are components of some spaces
$\widehat{\Gamma}_{1}$ and $\widehat{\Sigma}_{1}$ which are also
defined below.

\textbf{Some remarks on the Stokes problem}

Let $\Omega$ be a bounded and connected open subset of
$\mathbb{R}^{N}$ ($N=2,3$) with a Lipschitz - continuous boundary
$\partial \Omega$.

\begin{rem}
\label{observatia5_omega_domega} The conditions for $\Omega$ and
$\partial \Omega$ are the ones from Theorem I.5.1, page 80,
\cite{CLBichir_bib_Gir_Rav1986}. This theorem establishes the
existence and the uniqueness of the solution of the nonhomogeneous
Dirichlet problem for the stationary Stokes equations.
\end{rem}

Let us consider the variational formulation associated to
(\ref{e1_11_ecStokes}) - (\ref{e1_10_ecStokes})
\cite{CLBichir_bib_Dautray_Lions_vol8_1988,
CLBichir_bib_Gir_Rav1986}: given $\textbf{f}$ $\in$
$\textbf{H}^{-1}(\Omega)$, $g$ $\in$ $L_{0}^{2}(\Omega)'$ $\simeq$
$L_{0}^{2}(\Omega)$, find $(\textbf{u},p)$ $\in$
$\textbf{H}_{0}^{1}(\Omega) \times L_{0}^{2}(\Omega)$ such that
\begin{eqnarray}
   && (grad \, \textbf{u},grad \, \textbf{w})-(p,div \, \textbf{w})=\langle \textbf{f},\textbf{w} \rangle, \
              \forall \, \textbf{w}  \, \in \, \textbf{H}_{0}^{1}(\Omega) \, ,
         \label{e1_11_ecStokes_ecPoisson_variational_dem_th_I_5_1} \\
   && (div \, \textbf{u},\mu)=(g,\mu), \
              \forall \, \mu  \, \in \, L_{0}^{2}(\Omega) \, .
         \label{e5_1_cond_ecStokes_ecPoisson_variational_dem_th_I_5_1}
\end{eqnarray}

The compatibility condition (page 828,
\cite{CLBichir_bib_Dautray_Lions_vol8_1988}) is satisfied in this
case,
\begin{equation}
\label{e1_conditie_compatibility_Stokes}
   \int\limits_{\Omega} g \, dx
      = \int\limits_{\partial \Omega}(0 \cdot \textbf{n})|_{\partial \Omega} \, ds \, .
\end{equation}

Let $\Psi_{S}$ be the operator defined by the left hand sides of
problem (\ref{e1_11_ecStokes_ecPoisson_variational_dem_th_I_5_1})
- (\ref{e5_1_cond_ecStokes_ecPoisson_variational_dem_th_I_5_1}),
$\Psi_{S}$ $\in$ $L(\textbf{H}_{0}^{1}(\Omega) \times
L_{0}^{2}(\Omega);\textbf{H}^{-1}(\Omega) \times
L_{0}^{2}(\Omega))$. Then, problem
(\ref{e1_11_ecStokes_ecPoisson_variational_dem_th_I_5_1}) -
(\ref{e5_1_cond_ecStokes_ecPoisson_variational_dem_th_I_5_1}) has
the form
\begin{eqnarray}
   && \Psi_{S}(\textbf{u},p) = (\textbf{f},g) \, ,
         \label{e1_11_ecStokes_ecPoisson_ec_FI_widetilde_dem}
\end{eqnarray}

Theorem I.5.1, page 80, \cite{CLBichir_bib_Gir_Rav1986}, together
with Theorem I.4.1, page 59, \cite{CLBichir_bib_Gir_Rav1986}, and
Theorem XIX.9, page 828,
\cite{CLBichir_bib_Dautray_Lions_vol8_1988}, together with Remark
XIX.6, page 828, \cite{CLBichir_bib_Dautray_Lions_vol8_1988},
provide that problem
(\ref{e1_11_ecStokes_ecPoisson_variational_dem_th_I_5_1}) -
(\ref{e5_1_cond_ecStokes_ecPoisson_variational_dem_th_I_5_1}) is
well - posed. Hence $\Psi_{S}$ is an isomorphism of
$\textbf{H}_{0}^{1}(\Omega) \times L_{0}^{2}(\Omega)$ onto
$\textbf{H}^{-1}(\Omega) \times L_{0}^{2}(\Omega)$.

Since we need $p$ $\in$ $H^{1}(\Omega)$, we take $(u,p)$ $\in$
$\textbf{H}_{0}^{1}(\Omega)$ $\times$ $(H^{1}(\Omega) \cap
L_{0}^{2}(\Omega))$ and $\Upsilon_{1} \times M$ $=$
$\Psi_{S}(\textbf{H}_{0}^{1}(\Omega) \times (H^{1}(\Omega) \cap
L_{0}^{2}(\Omega)))$, $\Upsilon_{1}$ is a dense subspace of
$\textbf{H}^{-1}(\Omega)$ and $M$ $=$ $L_{0}^{2}(\Omega)$. This
follows from Lemma \ref{lema_izom_corolar} and from the fact that
$H^{1}(\Omega) \cap L_{0}^{2}(\Omega)$ is dense in
$L_{0}^{2}(\Omega)$. These data correspond to the spaces
$\widehat{\Gamma}_{1}$ si $\widehat{\Sigma}_{1}$ defined below.

To obtain
(\ref{e1_11_ecStokes_ecPoisson_variational_dem_ec_Poisson_pt_p_CONT})
below for the steady case, we take $p$ $\in$ $H^{1}(\Omega) \cap
L_{0}^{2}(\Omega)$ and we adapt the technique used in
\cite{CLBichir_bib_Sani_Shen_Pironneau_Gresho2006} to deduce
(\ref{mathcal_P2_e1_6_ecStokes_ecPoisson_variational_Sani_Shen_Pironneau_Gresho2006})
in the unsteady case.

From (\ref{e1_11_ecStokes_ecPoisson_variational_dem_th_I_5_1}), we
obtain
\begin{eqnarray}
   && -\triangle \textbf{u} + grad \, p - \textbf{f} = 0,
              \ \textrm{in} \ (\mathcal{D}'(\Omega))^{N} \, ,
         \label{e1_11_ecStokes_ecPoisson_variational_dem_distr_CONT}
\end{eqnarray}
that is, (\ref{e1_11_ecStokes_ecPoisson_variational_dem_distr})
below. In other words,
\begin{eqnarray}
   && grad \, p = \triangle \textbf{u} + \textbf{f},
              \ \textrm{in} \ (\mathcal{D}'(\Omega))^{N} \, ,
         \label{e1_11_ecStokes_ecPoisson_variational_dem_distr_CONT2}
\end{eqnarray}
$grad \, p$ $\in$ $\textbf{L}^{2}(\Omega)$, hence $\triangle
\textbf{u} + \textbf{f}$ $\in$ $\textbf{L}^{2}(\Omega)$, so we
have
\begin{eqnarray}
   && (-\triangle \textbf{u} - \textbf{f},grad \, \bar{\mu})
            + (grad \, p,grad \, \bar{\mu})=0, \
              \forall \, \bar{\mu}  \, \in \, H^{1}(\Omega) \, ,
         \label{e1_11_ecStokes_ecPoisson_variational_dem_ec_Poisson_pt_p_CONT}
\end{eqnarray}

Let $\textbf{z}$ $=$ $-\triangle \textbf{u}$. Then,
(\ref{e1_11_ecStokes_ecPoisson_variational_dem_ec_Poisson_pt_p_CONT})
is written
\begin{eqnarray}
   && (\textbf{z} - \textbf{f},grad \, \bar{\mu})
            + (grad \, p,grad \, \bar{\mu})=0, \
              \forall \, \bar{\mu}  \, \in \, H^{1}(\Omega) \, ,
         \label{e1_11_ecStokes_ecPoisson_variational_dem_ec_Poisson_pt_p_CONT2}
\end{eqnarray}

We observe that, instead (\ref{e1_6_ecStokes_ecPoisson0}), we can
use
\begin{equation}
\label{e1_6_ecStokes_ecPoisson0_completata}
   \alpha \, \nu \, div \, (\triangle \textbf{u}) - \triangle p = - div \, \textbf{f} \; \textrm{in} \; \Omega \, ,
\end{equation}
with a fixed real $\alpha \neq 1$ ($\nu=1$ as we established). We
quickly deduce that $div \, \textbf{u} = 0$. Indeed, from
(\ref{e1_11_ecStokes}) and
(\ref{e1_6_ecStokes_ecPoisson0_completata}), we obtain $(1-\alpha)
\, \triangle (div \, \textbf{u})$  $=$ $0$. Using
(\ref{dem101_Gir_Rav_pag51_G_P}), we deduce $\triangle \triangle
\zeta = 0$, Using (\ref{dem101_Gir_Rav_pag51_CL_G_P}),
(\ref{dem101_Gir_Rav_pag51_CL_Neumann_G_P}), it results $\zeta =
0$, so $div \, \textbf{u} = 0$, that is, (\ref{e1_9_ecStokes}).
For each $\alpha \neq 1$, we can obtain an approximate solution of
the given Stokes problem.

\textbf{Definition of the subspace $\mathcal{W}$}

Let us retain, from \cite{CLBichir_bib_Ciarlet2002,
CLBichir_bib_Glowinski1984,
CLBichir_bib_Glowinski_Pironneau1979_Numer_Math,
CLBichir_bib_Glowinski_Pironneau1979}, the formulations of the
Dirichlet problem for the biharmonic operator and the space
$\mathcal{V}$,
\begin{eqnarray}
   && \mathcal{V} = \{ (v,\psi) \in
      H_{0}^{1}(\Omega) \times L^{2}(\Omega);
      \forall \mu  \, \in \, H^{1}(\Omega),
      b((v,\psi),\mu) = 0 \} \, ,
         \label{e5_1_cond_ecStokes_ecPoisson_variational_dem_H02_var1cont_spVmathcal}
\end{eqnarray}
where $b((v,\psi),\mu)$ $=$ $(grad \, v,grad \, \mu) -
(\psi,\mu)$. ($\psi$ $=$ $div \, \textbf{u}$, with $u$ $\in$
$H^{1}(\Omega)^{N}$, $u = u_{w}$ on $\partial \Omega$, in
\cite{CLBichir_bib_Glowinski1984,
CLBichir_bib_Glowinski_Pironneau1979_Numer_Math}). $\mathcal{V}$
is defined for $\Omega$ convex or with $\partial \Omega$
sufficiently smooth and, in this case, $v$ $\in$ $H^{2}(\Omega)$.
Theorem 7.1.1, page 384, \cite{CLBichir_bib_Ciarlet2002}, tells us
that $\mathcal{V}$ $=$ $\mathcal{V}_{1}$, where
\begin{eqnarray}
   && \mathcal{V}_{1} = \{ (v,\psi) \in
      H_{0}^{2}(\Omega) \times L^{2}(\Omega);
      -\triangle v = \psi \} \, .
         \label{e5_1_cond_ecStokes_ecPoisson_variational_dem_H02_var1cont_spVmathcal2}
\end{eqnarray}

Since we consider a Lipschitz - continuous boundary $\partial
\Omega$, we take directly $v$ $\in$ $H_{0}^{1}(\Omega) \cap
H^{2}(\Omega)$ and consider the space
\begin{eqnarray}
   && \mathcal{W} = \{ (v,\psi) \in
      (H_{0}^{1}(\Omega) \cap H^{2}(\Omega)) \times L^{2}(\Omega);
      (v,\psi) \in \mathcal{V} \} \, .
         \label{e5_1_cond_ecStokes_ecPoisson_variational_dem_H02_var1cont_spWmathcal}
\end{eqnarray}
The proof of Theorem 7.1.1, page 384,
\cite{CLBichir_bib_Ciarlet2002}, remains valid if $\mathcal{V}$ is
replaced by $\mathcal{W}$, so $\mathcal{W}$ $=$ $\mathcal{V}_{1}$
and we can use the formulations of the Dirichlet problem for the
biharmonic operator, from \cite{CLBichir_bib_Ciarlet2002,
CLBichir_bib_Glowinski1984,
CLBichir_bib_Glowinski_Pironneau1979_Numer_Math,
CLBichir_bib_Glowinski_Pironneau1979}, using the space
$\mathcal{W}$ instead of the space $\mathcal{V}$.

\textbf{A remark on the density of a subspace}

Let $T_{PD}$ $\in$ $L(H^{-1}(\Omega),H_{0}^{1}(\Omega))$ be the
solution operator for the (scalar) Poisson equation with
homogeneous Dirichlet boundary conditions. Details are given in
Section \ref{sectiunea_alt_izomorfism}. Following page 370,
\cite{CLBichir_bib_Dautray_Lions_vol2_ENGLEZA_1988}, where we take
$V$ $=$ $H_{0}^{1}(\Omega)$, $H$ $=$ $L^{2}(\Omega)$ and $A$ $=$
$T_{PD}^{-1}$, we have
\begin{eqnarray}
   && D = D(T_{PD}^{-1}) = T_{PD}(L^{2}(\Omega)) = \{
      v \in L^{2}(\Omega),
      T_{PD}^{-1} v  \in L^{2}(\Omega) \} \, ,
         \label{e5_1_cond_ecStokes_ecPoisson_variational_dem_H02_var1cont_spWmathcal_1_D}
\end{eqnarray}
and $T_{PD}^{-1}|_{T_{PD}(L^{2}(\Omega))}$ is an isomorphism of
$D(T_{PD}^{-1})$, considered with the norm of the graph, onto
$L^{2}(\Omega)$
\cite{CLBichir_bib_Dautray_Lions_vol2_ENGLEZA_1988}.
$D(T_{PD}^{-1})$ is dense in $L^{2}(\Omega)$
\cite{CLBichir_bib_Dautray_Lions_vol2_ENGLEZA_1988}.
$D(T_{PD}^{-1})$ is dense in $H_{0}^{1}(\Omega)$
\cite{CLBichir_bib_Dautray_Lions_vol2_ENGLEZA_1988}.

Let
\begin{eqnarray}
   && D_{0} = T_{PD}(L_{0}^{2}(\Omega)) = \{
      v \in L^{2}(\Omega),
      T_{PD}^{-1} v  \in L_{0}^{2}(\Omega) \} \, ,
         \label{e5_1_cond_ecStokes_ecPoisson_variational_dem_H02_var1cont_spWmathcal_1_D_0}
\end{eqnarray}

$L_{0}^{2}(\Omega)$ is closed in $L^{2}(\Omega)$, so $D_{0}$ $=$
$T_{PD}(L_{0}^{2}(\Omega))$ is closed  in $D(T_{PD}^{-1}) =
T_{PD}(L^{2}(\Omega))$ (in the topology of norm of the graph).
$T_{PD}^{-1}(D_{0})$ $=$ $L_{0}^{2}(\Omega)$.

\begin{lem}
\label{lema_dense_norm_of_the_graph} $D_{0}$ $\cap$
$H_{0}^{2}(\Omega)$ is dense in $D_{0}$ in the topology of the
norm of the graph.
\end{lem}

\begin{proof} If $v$ $\in$ $H^{2}(\Omega)$, then
$-\triangle v$ $\in$ $L^{2}(\Omega)$.  So
$T_{PD}^{-1}(H_{0}^{2}(\Omega))$ $\subset$ $L^{2}(\Omega)$ and
$H_{0}^{2}(\Omega)$ $\subset$ $D$.

Let us prove that $H_{0}^{2}(\Omega)$ is dense $D$.

We denote $\mathbb{B}_{\varepsilon}(v)$ $=$ $\{ \tilde{v} \in
L^{2}(\Omega), \textrm{such that
(\ref{e5_1_norma_dem_H02_var1cont_spWmathcal_1_D_0}) holds.} \}$,
where
\begin{eqnarray}
   && \| v - \tilde{v} \|_{L^{2}(\Omega)}
      + \| T_{PD}^{-1}|_{T_{PD}(L^{2}(\Omega))}(v - \tilde{v}) \|_{L^{2}(\Omega)}
      < \varepsilon \, .
         \label{e5_1_norma_dem_H02_var1cont_spWmathcal_1_D_0}
\end{eqnarray}

Let $v$ $\in$ $D$.

Assume that $\exists$ $\varepsilon > 0$ such that $\forall$
$\tilde{v}$ $\in$ $H_{0}^{2}(\Omega)$ we have
\begin{eqnarray}
   && \| v - \tilde{v} \|_{L^{2}(\Omega)}
      + \| T_{PD}^{-1}|_{T_{PD}(L^{2}(\Omega))}(v - \tilde{v}) \|_{L^{2}(\Omega)}
      \geq \varepsilon \, .
         \label{e5_1_norma_dem_H02_var1cont_spWmathcal_1_D_0_not}
\end{eqnarray}
or $\mathbb{B}_{\varepsilon}(v)$ $\cap$ $H_{0}^{2}(\Omega)$ $=$
$\emptyset$ or $\forall$ $\tilde{v}$ $\in$
$\mathbb{B}_{\varepsilon}(v)$, we have $\tilde{v}$ $\notin$
$H_{0}^{2}(\Omega)$ and
(\ref{e5_1_norma_dem_H02_var1cont_spWmathcal_1_D_0}).

Hence $\forall$ $\tilde{v}$ $\in$ $\mathbb{B}_{\varepsilon}(v)$,
we have $\tilde{v}$ $\notin$ $H_{0}^{2}(\Omega)$ and $\| v -
\tilde{v} \|_{L^{2}(\Omega)}$ $<$ $\varepsilon$. This contradicts
the density of $H_{0}^{2}(\Omega)$ in $L^{2}(\Omega)$.

Hence, for all $v$ $\in$ $D$ and for all $\varepsilon > 0$, there
exists $\tilde{v}$ $\in$ $H_{0}^{2}(\Omega)$ such that
(\ref{e5_1_norma_dem_H02_var1cont_spWmathcal_1_D_0}) holds, so
$H_{0}^{2}(\Omega)$ is dense $D$.

Let $v$ $\in$ $int \ D_{0}$.

There exists $\varepsilon > 0$ such that
$\mathbb{B}_{\varepsilon}(v)$ $\subset$ $int \ D_{0}$

Assume that for a such $\varepsilon$ we have
$\mathbb{B}_{\varepsilon}(v)$ $\cap$ $H_{0}^{2}(\Omega)$ $=$
$\emptyset$.

Proceeding as before, we obtain that for $v$ $\in$ $int \ D_{0}$,
for all $\varepsilon > 0$, there exists $\tilde{v}$ $\in$ $D_{0}$
$\cap$ $H_{0}^{2}(\Omega)$ such that
(\ref{e5_1_norma_dem_H02_var1cont_spWmathcal_1_D_0}) holds.

Hence, $int \ D_{0}$ $\cap$ $H_{0}^{2}(\Omega)$ $\neq$ $\emptyset$
and, as a consequence, $int \ L_{0}^{2}(\Omega)$ $\cap$
$T_{PD}^{-1}(H_{0}^{2}(\Omega))$ $\neq$ $\emptyset$. $int \ D_{0}$
is in the induced topology of $D$ on $D_{0}$, $int \
L_{0}^{2}(\Omega)$ is in the induced topology of $L^{2}(\Omega)$
on $L_{0}^{2}(\Omega)$.

The space $\mathcal{V}_{1}$ is a closed subspace of
$H_{0}^{1}(\Omega)$ $\times$ $L^{2}(\Omega)$
\cite{CLBichir_bib_Ciarlet2002}. Instead of $\mathcal{V}_{1}$, we
work with
\begin{eqnarray}
   && \widetilde{\mathcal{V}}_{1} = \{ (\psi,v) \in
      L^{2}(\Omega) \times H_{0}^{2}(\Omega);
      T_{PD}\psi = v \} \, ,
         \label{e5_1_cond_ecStokes_ecPoisson_variational_dem_H02_var1cont_spVmathcal2_widetilde}
\end{eqnarray}

The space $\widetilde{\mathcal{V}}_{1}$ is a closed subspace of
$L^{2}(\Omega) \times H_{0}^{2}(\Omega)$.

It results that the restriction $T_{2}$ $\in$
$L(D(T_{2}),H_{0}^{2}(\Omega))$ of $T_{PD}$ to $D(T_{2})$
$\subset$ $L^{2}(\Omega)$ is a closed operator. $D(T_{2})$ is
complete for the norm of the graph.

If $v$ $\in$ $H^{2}(\Omega)$, then $-\triangle v$ $\in$
$L^{2}(\Omega)$.  So $T_{PD}^{-1}(H_{0}^{2}(\Omega))$ $\subset$
$L^{2}(\Omega)$.

$T_{PD}^{-1}(H_{0}^{2}(\Omega))$ $=$ $D(T_{2})$ as sets (or as
subspaces of $L^{2}(\Omega)$ (in the topology of
$L^{2}(\Omega)$)).

Hence, $D(T_{2})$ $\cap$ $int \ L_{0}^{2}(\Omega)$ $\neq$
$\emptyset$ (in the topology of $L^{2}(\Omega)$).

Let us prove that the restriction $\widetilde{L}$ of
$L_{0}^{2}(\Omega)$ to $D(T_{2})$ is closed in the topology of the
norm of the graph.

Consider a sequence $\{ l_{n} \}$ $\subset$ $\widetilde{L}$ such
that $\| l_{n} - l \|_{D(T_{2})}$ $\rightarrow$ $0$ in $D(T_{2})$
as $n$ $\rightarrow$ $\infty$. $l$ $\in$ $D(T_{2})$ since
$D(T_{2})$ is closed. $\| l_{n} - l \|_{L^{2}(\Omega)}$ $\leq$ $\|
l_{n} - l \|_{D(T_{2})}$, so $\| l_{n} - l \|_{L^{2}(\Omega)}$
$\rightarrow$ $0$ in $L^{2}(\Omega)$ with $\{ l_{n} \}$ $\subset$
$L_{0}^{2}(\Omega)$. $L_{0}^{2}(\Omega)$ is closed in
$L^{2}(\Omega)$, hence $l$ $\in$ $L_{0}^{2}(\Omega)$. We obtain
$l$ $\in$ $\widetilde{L}$ and $\widetilde{L}$ is closed in
$D(T_{2})$.

Inspired by some density proofs from
\cite{CLBichir_bib_Schechter2002}, where some projection operators
are used, let us work with the operators $P_{2}$ and $P$, where

$P_{2}$ $=$ $I_{2}$ on $D(T_{2})$ $\ominus$ $(D(T_{2})$ $\cap$
$L_{0}^{2}(\Omega))$ $=$ $D(T_{2})$ $\ominus$ $\widetilde{L}$,

$P_{2}$ $=$ $0$ on $D(T_{2})$ $\cap$ $L_{0}^{2}(\Omega)$ $=$
$\widetilde{L}$,

$P$ $=$ $I$ on $L^{2}(\Omega)$ $\ominus$ $L_{0}^{2}(\Omega)$,

$P$ $=$ $0$ on $L_{0}^{2}(\Omega)$.

Here, $I_{2}$ is the identity operator on $D(T_{2})$, $I$ is the
identity operator on $L^{2}(\Omega)$, $P_{2}$ is the projection
operator of $D(T_{2})$ onto $D(T_{2})$ $\ominus$ $\widetilde{L}$,
$P_{2}$ $\in$ $L(D(T_{2}),D(T_{2}))$, $P$ is the projection
operator of $L^{2}(\Omega)$ onto $L^{2}(\Omega)$ $\ominus$
$L_{0}^{2}(\Omega)$, $P$ $\in$ $L(L^{2}(\Omega),L^{2}(\Omega))$.

We have $P_{2}(y)$ $=$ $P(y)$ and $I_{2}(y)$ $=$ $I(y)$ for $y$
$\in$ $D(T_{2})$.

We write $T_{PD}$, but the operator is the restriction
$T_{PD}|_{L^{2}(\Omega)}$ of $T_{PD}$ to $L^{2}(\Omega)$.
$T_{PD}|_{L^{2}(\Omega)}$ is the inverse of
$T_{PD}^{-1}|_{T_{PD}(L^{2}(\Omega))}$ from the definition of $D$.

Let $u$ $\in$ $D_{0}$. There exists $y$ $\in$ $L_{0}^{2}(\Omega)$
such that $u$ $=$ $T_{PD}y$.

$u$ $\in$ $D$.

$H_{0}^{2}(\Omega)$ is dense in $D$, so there exists a sequence
$\{ u_{n} \}$ $\subset$ $H_{0}^{2}(\Omega)$ such that $\| u -
u_{n} \|_{D}$ $\rightarrow$ $0$ as $n$ $\rightarrow$ $\infty$ .

There exists $y_{n}$ $\in$ $D(T_{2})$ $\subset$ $L^{2}(\Omega)$
such that $u_{n}$ $=$ $T_{2}y_{n}$ $=$ $T_{PD}y_{n}$.

It results $\| y - y_{n} \|_{L^{2}(\Omega)}$ $\rightarrow$ $0$ as
$n$ $\rightarrow$ $\infty$.

Let $v_{n}$ $=$ $T_{2}(I_{2}-P_{2})y_{n}$ $=$ $T_{PD}(I-P)y_{n}$.

$y_{n}$ $\in$ $D(T_{2})$, so $(I_{2}-P_{2})y_{n}$ $\in$
$\widetilde{L}$. Hence, $v_{n}$ $\in$ $D_{0}$ and $v_{n}$ $\in$
$H_{0}^{2}(\Omega)$, so $v_{n}$ $\in$ $D_{0}$ $\cap$
$H_{0}^{2}(\Omega)$.

We have $\| y - y_{n} \|_{L^{2}(\Omega)}$ $\rightarrow$ $0$ as $n$
$\rightarrow$ $\infty$.

It results $v_{n}$ $=$ $T_{PD}(I-P)y_{n}$ $\rightarrow$
$T_{PD}(I-P)y$ $=$ $u$ in the topology of $D$ (since $(I-P)y_{n}$
$\rightarrow$ $(I-P)y$ in the topology of $L^{2}(\Omega)$).

Hence, $D_{0}$ $\cap$ $H_{0}^{2}(\Omega)$ is dense in $D_{0}$ in
the topology of the norm of the graph.

\qquad
\end{proof}

\textbf{Definition of the subspaces $\widetilde{V^{\bot}}$ and
$\Upsilon_{0}$}

It results that $T_{PD}^{-1}(D_{0}$ $\cap$ $H_{0}^{2}(\Omega))$ is
dense in $T_{PD}^{-1}(D_{0})$ $=$ $L_{0}^{2}(\Omega)$ in the
topology of $L^{2}(\Omega)$.

$div$ is an isomorphism of $V^{\bot}$ onto $L_{0}^{2}(\Omega)$,
where $V$ $=$ $\{ \textbf{v} \in \textbf{H}_{0}^{1}(\Omega)| div
\, \textbf{v}=0 \}$ and $\textbf{H}_{0}^{1}(\Omega)$ $=$ $V \oplus
V^{\bot}$ \cite{CLBichir_bib_Gir_Rav1986}.

Let us define $\widetilde{V^{\bot}}$ $=$ $(div \,
|_{V^{\bot}})^{-1}(T_{PD}^{-1}(D_{0}$ $\cap$
$H_{0}^{2}(\Omega)))$. $T_{PD}^{-1}(D_{0}$ $\cap$
$H_{0}^{2}(\Omega))$ is dense in $T_{PD}^{-1}(D_{0})$ $=$
$L_{0}^{2}(\Omega)$, in the topology of $L_{0}^{2}(\Omega)$ (or
$L^{2}(\Omega)$), so $\widetilde{V^{\bot}}$ is dense in $V^{\bot}$
$=$ $(div \, |_{V^{\bot}})^{-1}(L_{0}^{2}(\Omega))$.

It results that $V \oplus \widetilde{V^{\bot}}$ is dense in
$\textbf{H}_{0}^{1}(\Omega)$ ($=$ $V \oplus V^{\bot}$) (since $(V
\oplus \widetilde{V^{\bot}})^{\bot} = \{ 0 \}$).

Let us define $\Upsilon_{0}$ by $\Upsilon_{0}$ $\times$
$T_{PD}^{-1}(D_{0}$ $\cap$ $H_{0}^{2}(\Omega))$ $=$ $\Psi_{S}(V
\oplus \widetilde{V^{\bot}};H^{1}(\Omega) \cap
L_{0}^{2}(\Omega))$. If we take $(\textbf{u},p)$ $\in$ $V \oplus
\widetilde{V^{\bot}}$ $\times$ $(H^{1}(\Omega) \cap
L_{0}^{2}(\Omega))$, we have $(\textbf{f},g)$ $\in$ $\Upsilon_{0}$
$\times$ $T_{PD}^{-1}(D_{0}$ $\cap$ $H_{0}^{2}(\Omega))$.

It results that $\Upsilon_{0}$ $\times$ $T_{PD}^{-1}(D_{0}$ $\cap$
$H_{0}^{2}(\Omega))$ is dense in $\textbf{H}^{-1}(\Omega)$
$\times$ $L_{0}^{2}(\Omega)$, so $\Upsilon_{0}$ is a dense
subspace of $\textbf{H}^{-1}(\Omega)$.

Let $\textbf{F}$ $\in$ $\textbf{H}^{-1}(\Omega)$. Let
$S_{\textbf{F}}$ $\in$ $\mathcal{D}'(\Omega)$ be defined by
\begin{eqnarray}
   && \langle S_{\textbf{F}},\varphi \rangle = - \langle \textbf{F} , grad \, \varphi \rangle, \
              \forall \, \varphi  \, \in \, \mathcal{D}(\Omega) \, ,
         \label{e1_6_ecStokes_ecPoisson_variational_dem_f}
\end{eqnarray}
Let $\widetilde{S}_{\textbf{F}}$ be the unique continuous
extension of $S_{\textbf{F}}$ to $H_{0}^{1}(\Omega)$ (using page
49, \cite{CLBichir_bib_Adams1978}, and page 291,
\cite{CLBichir_bib_Atkinson_Han2009}).
$\widetilde{S}_{\textbf{F}}$ $\in$ $H^{-1}(\Omega)$.

\subsection{Some definitions introduced in order to apply Lemma
\ref{well_posed_lema_deschisa} and Theorem
\ref{well_posed_teorema5_1_Stokes_math_B_L2}}
\label{sectiunea_ecuatii_remarks_2}

In order to gather some mathematical entities that we use in the
sequel, we consider some definitions from Sections
\ref{sectiunea_alt_izomorfism} and
\ref{sectiunea_ecuatii_Stokes_problema_aproximativa} below. Let us
now define:

$\widehat{\Gamma}$ $=$ $\mathcal{Q}$ $\times$ $\mathcal{X}$,
$\widehat{\Sigma}$ $=$ $\mathcal{Z}$ $\times$ $\mathcal{Y}$,
$\widehat{\Gamma}_{h}$ $=$ $\mathcal{Q}$ $\times$
$\mathcal{X}_{h}$,

$\mathcal{Q}$ $=$ $H^{1}(\Omega)$ $\times$
$L_{0}^{2}(\Omega)^{2}$,

$\mathcal{X}$ $=$ $\textbf{H}_{0}^{1}(\Omega)$ $\times$
$H_{0}^{1}(\Omega)^{2}$ $\times$ $\textbf{L}^{2}(\Omega)$ $\times$
$H^{1}(\Omega) \cap L_{0}^{2}(\Omega)$ $\times$ $H^{1}(\Omega)$
$\times$ $\textbf{L}^{2}(\Omega)$,

$\mathcal{Z}$ $=$ $H^{1}(\Omega)$ $\times$
$L_{0}^{2}(\Omega)^{2}$,

$\mathcal{Y}$ $=$ $\textbf{H}^{-1}(\Omega)$ $\times$
$H^{-1}(\Omega)^{2}$ $\times$ $\textbf{L}^{2}(\Omega)$ $\times$
$B_{1}^{-1}(H^{1}(\Omega) \cap L_{0}^{2}(\Omega))$ $\times$
$(H^{1}(\Omega))'$ $\times$ $\textbf{L}^{2}(\Omega)$,

$\mathcal{X}_{h}$ $=$ $\textbf{X}_{0h}$ $\times$ $Y_{0h}^{2}$
$\times$ $\textbf{X}_{h}$ $\times$ $M_{h}$ $\times$ $Y_{h}$
$\times$ $\textbf{X}_{h}$,

$\widehat{\Gamma}_{0}$ $=$ $\mathcal{Q}_{0}$ $\times$
$\mathcal{X}_{0}$, $\widehat{\Sigma}_{0}$ $=$ $\mathcal{Z}_{0}$
$\times$ $\mathcal{Y}_{0}$,

$\mathcal{Q}_{0}$ $=$ $H^{1}(\Omega)$ $\times$ $H^{1}(\Omega) \cap
L_{0}^{2}(\Omega)$ $\times$ $T_{PD}^{-1}(D_{0}$ $\cap$
$H_{0}^{2}(\Omega))$,

$\mathcal{X}_{0}$ $=$ $V \oplus \widetilde{V^{\bot}}$ $\times$
$H_{0}^{1}(\Omega)^{2}$ $\times$ $\textbf{L}^{2}(\Omega)$ $\times$
$H^{1}(\Omega) \cap L_{0}^{2}(\Omega)$ $\times$ $H^{1}(\Omega)$
$\times$ $\textbf{L}^{2}(\Omega)$,

$\mathcal{Z}_{0}$ $=$ $H^{1}(\Omega)$ $\times$ $H^{1}(\Omega) \cap
L_{0}^{2}(\Omega)$ $\times$ $(T_{PD}^{-1}(D_{0}$ $\cap$
$H_{0}^{2}(\Omega)))$,

$\mathcal{Y}_{0}$ $=$ $\Upsilon_{0}$ $\times$ $H^{-1}(\Omega)^{2}$
$\times$ $\textbf{L}^{2}(\Omega)$ $\times$
$B_{1}^{-1}(H^{1}(\Omega) \cap L_{0}^{2}(\Omega))$ $\times$
$(H^{1}(\Omega))'$ $\times$ $\textbf{L}^{2}(\Omega)$,

$\widehat{\Gamma}_{1}$ $=$ $\mathcal{Q}_{1}$ $\times$
$\mathcal{X}_{1}$, $\widehat{\Sigma}_{1}$ $=$ $\mathcal{Z}_{1}$
$\times$ $\mathcal{Y}_{1}$,

$\mathcal{Q}_{1}$ $=$ $H^{1}(\Omega)$ $\times$ $H^{1}(\Omega) \cap
L_{0}^{2}(\Omega)$ $\times$ $L_{0}^{2}(\Omega)$,

$\mathcal{X}_{1}$ $=$ $\mathcal{X}$,

$\mathcal{Z}_{1}$ $=$ $H^{1}(\Omega)$ $\times$ $H^{1}(\Omega) \cap
L_{0}^{2}(\Omega)$ $\times$ $L_{0}^{2}(\Omega)$,

$\mathcal{Y}_{1}$ $=$ $\Upsilon_{1}$ $\times$ $H^{-1}(\Omega)^{2}$
$\times$ $\textbf{L}^{2}(\Omega)$ $\times$
$B_{1}^{-1}(H^{1}(\Omega) \cap L_{0}^{2}(\Omega))$ $\times$
$(H^{1}(\Omega))'$ $\times$ $\textbf{L}^{2}(\Omega)$,

$x$ $=$ $(\hat{p}, p_{S}, y)$,

$U$ $=$ $(\textbf{u}, q, \hat{q}, \textbf{z}, p, r, \textbf{t})$,

$x_{h}$ $=$ $(\hat{p}_{h}, p_{S,h}, y_{h})$,

$U_{h}$ $=$ $(\textbf{u}_{h}, q_{h}, \hat{q}_{h}, \textbf{z}_{h},
p_{h}, r_{h}, \textbf{t}_{h})$,

$\mathcal{T}$ $\in$ $L(\mathcal{Y},\mathcal{X})$, $\mathcal{T}$
$=$ $(T_{VPD}$, $T_{PD}$, $T_{PD}$, $\pi$, $B_{1}$, $B_{1}$,
$\pi)$,

$\mathcal{G}$ $\in$ $L(\mathcal{Q} \times
\mathcal{X},\mathcal{Y})$, $\mathcal{G}(y,U)$ $=$
$(G_{\textbf{u}}(p_{S})$, $G_{q}(\textbf{t},\textbf{z})$,
$G_{\hat{q}}(p)$, $G_{\textbf{z}}(p)$,
$G_{p}(p,\hat{p},\hat{q},\textbf{z},r)$, $G_{r}(q,\hat{q},r)$,
$G_{\textbf{t}}(\textbf{u}))$,

$\mathcal{B}$ $\in$ $L(\mathcal{Q} \times
\mathcal{X},\mathcal{Z})$,

$\widehat{\mathcal{F}}$ $=$
$[\mathcal{F}_{\mathcal{B}},\mathcal{F}]^{T}$,
$\mathcal{F}_{\mathcal{B}}$ $=$ $(\hat{g}$, $\ell_{S}$, $g)$,
$\mathcal{F}$ $=$ $(\textbf{f}$, $\psi$, $\hat{\psi}$,
$\textbf{f}_{\triangle}$, $\varrho$, $\hat{\varrho}$,
$\hat{\textbf{f}})$,

$\widehat{\mathcal{T}}\widehat{\mathcal{F}}$ $=$
$[\mathcal{J}_{\mathcal{Z}}\mathcal{F}_{\mathcal{B}},\mathcal{T}\mathcal{F}]^{T}$,
$\mathcal{J}_{\mathcal{Z}}\mathcal{F}_{\mathcal{B}}$ $=$
$\mathcal{F}_{\mathcal{B}}$, $\mathcal{T}\mathcal{F}$ $=$
$(T_{VPD}\textbf{f}$, $T_{PD}\psi$, $T_{PD}\hat{\psi}$,
$\pi\textbf{f}_{\triangle}$, $B_{1}\varrho$, $B_{1}\hat{\varrho}$,
$\pi\hat{\textbf{f}})$,

$\mathcal{T}_{h}$ $\in$ $L(\mathcal{Y},\mathcal{X}_{h})$,
$\mathcal{T}_{h}$ $=$ $(T_{VPD,h}$, $T_{PD,h}$, $T_{PD,h}$,
$\pi_{h}$, $B_{1,h}$, $B_{1,h}$, $\pi_{h})$,

$\widehat{\mathcal{T}}_{h}\widehat{\mathcal{F}}$ $=$
$[\mathcal{J}_{\mathcal{Z}}\mathcal{F}_{\mathcal{B}},\mathcal{T}_{h}\mathcal{F}]^{T}$,
$\mathcal{J}_{\mathcal{Z}}\mathcal{F}_{\mathcal{B}}$ $=$
$\mathcal{F}_{\mathcal{B}}$, $\mathcal{T}_{h}\mathcal{F}$ $=$
$(T_{VPD,h}\textbf{f}$, $T_{PD,h}\psi$, $T_{PD,h}\hat{\psi}$,
$\pi_{h}\textbf{f}_{\triangle}$, $B_{1,h}\varrho$,
$B_{1,h}\hat{\varrho}$, $\pi_{h}\hat{\textbf{f}})$,

$\mathcal{X}_{d}$ $=$ $\textbf{H}_{0}^{2}(\Omega)$ $\times$
$H_{0}^{2}(\Omega)^{2}$ $\times$ $\textbf{H}_{0}^{1}(\Omega)$
$\times$ $H^{m+1}(\Omega) \cap L_{0}^{2}(\Omega)$ $\times$
$H^{m+1}(\Omega)$ $\times$ $\textbf{H}_{0}^{1}(\Omega)$,

$\mathcal{S}$ $=$ $(C^{\infty}(\overline{\Omega}) \cap
H_{0}^{1}(\Omega))^{N+2}$ $\times$ $C_{0}^{\infty}(\Omega)^{N}$
$\times$ $C^{\infty}(\overline{\Omega}) \cap L_{0}^{2}(\Omega)$
$\times$ $C^{\infty}(\overline{\Omega})$ $\times$
$C_{0}^{\infty}(\Omega)^{N}$.

The correspondences between components are shown in the following
tables, where we skip "$(\Omega)$" in the notations.

\begin{tabular}{|c|c|c|c|c|c|}
  \hline 
  & & & & & \\
  $(x,U)$      & $j$ & $\widehat{\Gamma}$ & $\widehat{\Gamma}_{0}$ & $\widehat{\Gamma}_{1}$ & $\widehat{\Gamma}_{h}$ \\
  \hline 
  $\hat{p}$    &    & $H^{1}$ & $H^{1}$ & $H^{1}$ & $H^{1}$ \\
  $p_{S}$      &    & $L_{0}^{2}$ & $H^{1} \cap L_{0}^{2}$ & $H^{1} \cap L_{0}^{2}$ & $L_{0}^{2}$ \\
  $y$          &    & $T_{PD}^{-1}(D_{0})$ $=$ $L_{0}^{2}$ & $T_{PD}^{-1}(D_{0}$ $\cap$ $H_{0}^{2})$ & $L_{0}^{2}$ & $L_{0}^{2}$ \\
  $\textbf{u}$ &  1 & $\textbf{H}_{0}^{1}$ & $V \oplus \widetilde{V^{\bot}}$ & $\textbf{H}_{0}^{1}$ & $\textbf{X}_{0h}$ \\
  $q$          &  2 & $H_{0}^{1}$ & $H_{0}^{1}$ & $H_{0}^{1}$ & $Y_{0h}$ \\
  $\hat{q}$    &  3 & $H_{0}^{1}$ & $H_{0}^{1}$ & $H_{0}^{1}$ & $Y_{0h}$ \\
  $\textbf{z}$ &  4 & $\textbf{L}^{2}$ & $\textbf{L}^{2}$ & $\textbf{L}^{2}$ & $\textbf{X}_{h}$ \\
  $p$          &  5 & $H^{1} \cap L_{0}^{2}$ & $H^{1} \cap L_{0}^{2}$ & $H^{1} \cap L_{0}^{2}$ & $M_{h}$ \\
  $r$          &  6 & $H^{1}$ & $H^{1}$ & $H^{1}$ & $Y_{h}$ \\
  $\textbf{t}$ &  7 & $\textbf{L}^{2}$ & $\textbf{L}^{2}$ & $\textbf{L}^{2}$ & $\textbf{X}_{h}$ \\
  \hline
\end{tabular}

\begin{tabular}{|c|c|c|c|c|c|}
  \hline 
  & & & & &  \\
  $(y,U)$   & $j$ & $\widehat{\mathcal{F}}$ & $\widehat{\Sigma}$ & $\widehat{\Sigma}_{0}$ & $\widehat{\Sigma}_{1}$  \\
  \hline 
  $\hat{p}$    &    & $\hat{g}$       & $H^{1}$ & $H^{1}$ & $H^{1}$  \\
  $p_{S}$      &    & $\ell_{S}$      & $L_{0}^{2}$ & $H^{1} \cap L_{0}^{2}$ & $H^{1} \cap L_{0}^{2}$  \\
  $y$          &    & $g$             & $T_{PD}^{-1}(D_{0})$ $=$ $L_{0}^{2}$ & $T_{PD}^{-1}(D_{0}$ $\cap$ $H_{0}^{2})$ & $L_{0}^{2}$  \\
  $\textbf{u}$ &  1 & $f$             & $\textbf{H}^{-1}$ & $\Upsilon_{0}$ & $\Upsilon_{1}$  \\
  $q$          &  2 & $\psi$          & $H^{-1}$ & $H^{-1}$ & $H^{-1}$  \\
  $\hat{q}$    &  3 & $\hat{\psi}$    & $H^{-1}$ & $H^{-1}$ & $H^{-1}$  \\
  $\textbf{z}$ &  4 & $\textbf{f}_{\triangle}$          & $\textbf{L}^{2}$ & $\textbf{L}^{2}$ & $\textbf{L}^{2}$  \\
  $p$          &  5 & $\varrho$       & $B_{1}^{-1}(H^{1} \cap L_{0}^{2})$ & $B_{1}^{-1}(H^{1} \cap L_{0}^{2})$ & $B_{1}^{-1}(H^{1} \cap L_{0}^{2})$  \\
  $r$          &  6 & $\hat{\varrho}$ & $(H^{1})'$ & $(H^{1})'$ & $(H^{1})'$  \\
  $\textbf{t}$ &  7 & $\hat{\textbf{f}}$       & $\textbf{L}^{2}$ & $\textbf{L}^{2}$ & $\textbf{L}^{2}$  \\
  \hline
\end{tabular}

\begin{tabular}{|c|c|c|c|c|}
  \hline 
  & & & & \\
  $(y,U)$      & $j$ & For $\widehat{\Gamma}$, topology of & $\mathcal{X}_{d}$ & $\mathcal{S}$  \\
  \hline 
  $\hat{p}$    &    & $H^{1}$ & - & - \\
  $p_{S}$      &    & $L_{0}^{2}$ & - & - \\
  $y$          &    & $L_{0}^{2}$ (or $L^{2}$) & - & - \\
  $\textbf{u}$ &  1 & $\textbf{H}_{0}^{1}$ & $\textbf{H}_{0}^{2}$ & $(C^{\infty}(\overline{\Omega}) \cap H_{0}^{1})^{N}$ \\
  $q$          &  2 & $H_{0}^{1}$ & $H_{0}^{2}$ & $C^{\infty}(\overline{\Omega}) \cap H_{0}^{1}$ \\
  $\hat{q}$    &  3 & $H_{0}^{1}$ & $H_{0}^{2}$ & $C^{\infty}(\overline{\Omega}) \cap H_{0}^{1}$ \\
  $\textbf{z}$ &  4 & $\textbf{L}^{2}$ & $\textbf{H}_{0}^{1}$ & $C_{0}^{\infty}(\Omega)^{N}$ \\
  $p$          &  5 & $H^{1}$ & $H^{m+1} \cap L_{0}^{2}$ & $C^{\infty}(\overline{\Omega}) \cap L_{0}^{2}$ \\
  $r$          &  6 & $H^{1}$ & $H^{m+1}$ & $C^{\infty}(\overline{\Omega})$ \\
  $\textbf{t}$ &  7 & $\textbf{L}^{2}$ & $\textbf{H}_{0}^{1}$ & $C_{0}^{\infty}(\Omega)^{N}$ \\
  \hline
\end{tabular}
\begin{lem}
\label{lema_widehat_GAMA_dense_T_tabel} $\widehat{\Gamma}_{0}$ and
$\widehat{\Gamma}_{1}$ are dense subspaces of $\widehat{\Gamma}$.
$\widehat{\Sigma}_{0}$ and $\widehat{\Sigma}_{1}$ are dense
subspaces of $\widehat{\Sigma}$. The embedding $\mathcal{X}_{d}
\subseteq \mathcal{X}$ is continuous and the embedding operator
$\imath_{0}$ is compact. $\mathcal{S}$ is a set dense in
$\mathcal{X}$.
\end{lem}

\begin{proof}

The elements from the column of $\widehat{\Gamma}_{0}$ are dense
in the corresponding elements, on a line, from the column of
$\widehat{\Gamma}$ and so on. We use the references
\cite{CLBichir_bib_Adams1978, CLBichir_bib_Atkinson_Han2009,
CLBichir_bib_Bochev_GunzburgerLSFEM2009, CLBichir_bib_Braess2007,
CLBichir_bib_Brenner_Scott2008, CLBichir_bib_H_Brezis2010,
CLBichir_bib_Ciarlet2002,
CLBichir_bib_Ciarlet_Lions_Handbook_V2_P1_2003,
CLBichir_bib_Gir_Rav1986, CLBichir_bib_Quarteroni_Valli2008,
CLBichir_bib_E_Zeidler_NFA_IIA}.

\qquad
\end{proof}

\section{An extended system based on the Stokes problem}
\label{sectiunea_forme_echivalente_ale_ecuatiilor_extended}

\subsection{The formulation of the extended system}
\label{sectiunea_forme_echivalente_ale_ecuatiilor_extended_1}

Let us introduce the following exact extended system, related to
the Stokes problem, using the notations from Section
\ref{sectiunea_ecuatii_remarks}: given a fixed $\alpha \neq 1$ and
$\widehat{\mathcal{F}}$ $\in$ $\widehat{\Sigma}$, find $(x,U)$
$\in$ $\widehat{\Gamma}$ such that
\begin{eqnarray}
   && p - (q+\hat{p}) = \hat{g} \, ,
         \label{mathcal_P2_e1_6_ecStokes_ecPoisson_variational_dem_q_compl_p_dem11_cont_alt_izom} \\
   && p_{S} - p = \ell_{S} \, ,
         \label{e5_1_cond_ecStokes_ecPoisson_variational_dem_div_corelatie_exact} \\
   && (div \, \textbf{u}-y,\hat{\mu}) = (g,\hat{\mu}), \
              \forall \, \hat{\mu}  \, \in \, L_{0}^{2}(\Omega) \, ,
         \label{e5_1_cond_ecStokes_ecPoisson_variational_dem_div} \\
   && (grad \, \textbf{u},grad \, \textbf{w})-(p_{S},div \, \textbf{w})=\langle \textbf{f},\textbf{w} \rangle, \
              \forall \, \textbf{w}  \, \in \, \textbf{H}_{0}^{1}(\Omega) \, ,
         \label{e1_11_ecStokes_ecPoisson_variational_dem} \\
   && - \alpha \langle \widetilde{S}_{\textbf{z}},\bar{\lambda} \rangle
            + \langle \widetilde{S}_{\textbf{t}},\bar{\lambda} \rangle
            + (grad \, q,grad \, \bar{\lambda})
            =\langle \psi,\bar{\lambda} \rangle, \
              \forall \, \bar{\lambda}  \, \in \, H_{0}^{1}(\Omega) \, ,
         \label{mathcal_P2_e1_6_ecStokes_ecPoisson_variational_dem_q_compl_dem11_cont_alt_izom} \\
   && (grad \, \hat{q},grad \, \tilde{\lambda})
            +(grad \, p,grad \, \tilde{\lambda})
            =\langle \hat{\psi},\tilde{\lambda} \rangle, \
              \forall \, \tilde{\lambda}  \, \in \, H_{0}^{1}(\Omega) \, ,
          \label{mathcal_P2_e1_6_ecStokes_ecPoisson_variational_dem_q_compl_hat_q_dem11_cont_alt_izom} \\
   && (\textbf{z},\tilde{\textbf{w}})-(p,div \, \tilde{\textbf{w}})=(\textbf{f}_{\triangle},\tilde{\textbf{w}}), \
              \forall \, \tilde{\textbf{w}}  \, \in \, \textbf{H}_{0}^{1}(\Omega) \, ,
         \label{mathcal_P2_e1_6_ecStokes_ecPoisson_variational_dem_q_compl_dem11_cont_alt_izom_instead2_delta_hat_delta_aprox} \\
   && (\textbf{z} - \textbf{t} + grad \, (\hat{p}-\hat{q}),grad \, \varphi)
         \label{e1_11_ecStokes_ecPoisson_variational_dem_grad_miu} \\
   && \qquad +(grad \, r,grad \, \varphi)
            = \langle \varrho,\varphi \rangle, \
              \forall \, \varphi  \, \in \, H^{1}(\Omega) \, ,
         \nonumber \\
   && (r,\tilde{\varphi})
            +(grad \, (q+\hat{q}),grad \, \tilde{\varphi})
            = \langle \hat{\varrho},\tilde{\varphi} \rangle, \
              \forall \, \tilde{\varphi}  \, \in \, H^{1}(\Omega) \, ,
         \label{e1_11_ecStokes_ecPoisson_variational_dem_grad_miu_r} \\
   && (\textbf{t},\hat{\textbf{w}})=(\hat{\textbf{f}},\hat{\textbf{w}}), \
              \forall \, \hat{\textbf{w}}  \, \in \, \textbf{L}^{2}(\Omega) \, ,
          \label{e1_11_ecStokes_ecPoisson_variational_dem_grad_miu_ec_TAU}
\end{eqnarray}

\begin{rem}
\label{observatia_EC_EXACT_ecuatii_r}

Since $\textbf{z}$, $grad \, p$, $\textbf{f}_{\triangle}$ $\in$
$\textbf{L}^{2}(\Omega)$ and $H_{0}^{1}(\Omega)$ is dense in
$L^{2}(\Omega)$,
(\ref{mathcal_P2_e1_6_ecStokes_ecPoisson_variational_dem_q_compl_dem11_cont_alt_izom_instead2_delta_hat_delta_aprox})
is equivalent to the following formulations
\begin{eqnarray}
   && (\textbf{z},\tilde{\textbf{w}})+(grad \, p,\tilde{\textbf{w}})=(\textbf{f}_{\triangle},\tilde{\textbf{w}}), \
              \forall \, \tilde{\textbf{w}}  \, \in \, \textbf{L}^{2}(\Omega) \, ,
         \label{mathcal_P2_e1_6_ecStokes_ecPoisson_variational_dem_q_compl_dem11_cont_alt_izom_instead2_delta_hat_delta_aprox_ECH_var_modif_cont}
\end{eqnarray}
\begin{eqnarray}
   && \textbf{z} + grad \, p = \textbf{f}_{\triangle} \, ,
         \label{VAR_mathcal_P2_e1_6_ecStokes_ecPoisson_variational_dem_q_compl_dem11_cont_alt_izom_instead2_delta_hat_delta_aprox_D_VAR_ec}
\end{eqnarray}
In the sequel, we work with
(\ref{mathcal_P2_e1_6_ecStokes_ecPoisson_variational_dem_q_compl_dem11_cont_alt_izom_instead2_delta_hat_delta_aprox_ECH_var_modif_cont})
instead of
(\ref{mathcal_P2_e1_6_ecStokes_ecPoisson_variational_dem_q_compl_dem11_cont_alt_izom_instead2_delta_hat_delta_aprox}).
We use $\mathcal{M}_{5}$ and $\mathcal{M}_{h,5}$ below.

\end{rem}

\begin{rem}
\label{observatia_EC_EXACT_tau} Assume that problem
(\ref{mathcal_P2_e1_6_ecStokes_ecPoisson_variational_dem_q_compl_p_dem11_cont_alt_izom})
- (\ref{e1_11_ecStokes_ecPoisson_variational_dem_grad_miu_ec_TAU})
has a solution. Let us introduce $p_{S}$ from
(\ref{e5_1_cond_ecStokes_ecPoisson_variational_dem_div_corelatie_exact})
into (\ref{e1_11_ecStokes_ecPoisson_variational_dem}) and then
replace $(p,div \, \textbf{w})$ using
(\ref{mathcal_P2_e1_6_ecStokes_ecPoisson_variational_dem_q_compl_dem11_cont_alt_izom_instead2_delta_hat_delta_aprox}).
We obtain
\begin{eqnarray}
   && (grad \, \textbf{u},grad \, \textbf{w})
      - (\textbf{z},\textbf{w})
      = (\ell_{S},div \, \textbf{w})
         \label{e1_11_ecStokes_ecPoisson_variational_dem_diferenta} \\
   && \qquad +\langle \textbf{f},\textbf{w} \rangle-(\textbf{f}_{\triangle},\textbf{w}), \
              \forall \, \textbf{w}  \, \in \, \textbf{H}_{0}^{1}(\Omega) \, .
         \nonumber
\end{eqnarray}
\end{rem}

\begin{rem}
\label{observatia_EC_EXACT_descriere}

Equation (\ref{e1_11_ecStokes_ecPoisson_variational_dem}) is the
weak formulation
(\ref{e1_11_ecStokes_ecPoisson_variational_dem_th_I_5_1}) of the
fist equation of the  Stokes system (that is, of the momentum
equation). Equation
(\ref{e5_1_cond_ecStokes_ecPoisson_variational_dem_div})
corresponds to
(\ref{e5_1_cond_ecStokes_ecPoisson_variational_dem_th_I_5_1}).

Equation
(\ref{mathcal_P2_e1_6_ecStokes_ecPoisson_variational_dem_q_compl_p_dem11_cont_alt_izom})
corresponds to
(\ref{mathcal_P2_e1_6_ecStokes_ecPoisson_variational_dem_q_compl_hat_p_dem11_cont_alt_izom_CITE}).

We introduce equation
(\ref{e5_1_cond_ecStokes_ecPoisson_variational_dem_div_corelatie_exact})
to relate $p_{S}$ $\in$ $L_{0}^{2}(\Omega)$ and $p$ $\in$
$H^{1}(\Omega) \cap L_{0}^{2}(\Omega)$ in order to use Lemma
\ref{well_posed_lema_deschisa}. $p_{S}$ $=$ $p$ for $\ell_{S} =
0$.

(\ref{mathcal_P2_e1_6_ecStokes_ecPoisson_variational_dem_q_compl_dem11_cont_alt_izom})
corresponds to
(\ref{mathcal_P2_e1_6_ecStokes_ecPoisson_variational_dem_q_compl_dem11_cont_alt_izom_CITE})
and introduces (\ref{e1_6_ecStokes_ecPoisson0_completata}).

(\ref{mathcal_P2_e1_6_ecStokes_ecPoisson_variational_dem_q_compl_hat_q_dem11_cont_alt_izom})
corresponds to
(\ref{mathcal_P2_e1_6_ecStokes_ecPoisson_variational_dem_q_compl_hat_q_dem11_cont_alt_izom_CITE})
rewritten as an homogeneous Dirichlet problem for the Laplace
operator ($p$ and $\hat{p}$ have the same non-homogeneous
Dirichlet boundary condition). The connection between $p$,
$\hat{p}$ and $\hat{q}$ is in
(\ref{e1_11_ecStokes_ecPoisson_variational_dem_grad_miu}).

Equations
(\ref{mathcal_P2_e1_6_ecStokes_ecPoisson_variational_dem_q_compl_dem11_cont_alt_izom_instead2_delta_hat_delta_aprox})
and (\ref{e1_11_ecStokes_ecPoisson_variational_dem_grad_miu}) are
two forms of the momentum equation. Equation
(\ref{mathcal_P2_e1_6_ecStokes_ecPoisson_variational_dem_q_compl_dem11_cont_alt_izom_instead2_delta_hat_delta_aprox})
defines $\textbf{z}$ to be $-\triangle \textbf{u}$.

Equation (\ref{e1_11_ecStokes_ecPoisson_variational_dem_grad_miu})
is
(\ref{e1_11_ecStokes_ecPoisson_variational_dem_ec_Poisson_pt_p_CONT})
and it is like
(\ref{mathcal_P2_e1_6_ecStokes_ecPoisson_variational_Sani_Shen_Pironneau_Gresho2006}).

Equation (\ref{e1_11_ecStokes_ecPoisson_variational_dem_grad_miu})
gives an equation, on the boundary, for the pressure $p$ in the
approximate case. We regain the equation, containing the pressure,
that we lose when the velocity boundary condition is implemented,
for instance, in the case of nodal finite element method.

The unknown $r$ and equation
(\ref{e1_11_ecStokes_ecPoisson_variational_dem_grad_miu_r}) are
introduced in order to avoid a compatibility condition between
some components of the right hand term.

Equation
(\ref{e1_11_ecStokes_ecPoisson_variational_dem_grad_miu_ec_TAU})
is intended to introduce the free term $\textbf{f}$ between the
unknowns.

We do not use a boundary unknown $\lambda_{w}$, as in
(\ref{mathcal_P2_e1_6_ecStokes_ecPoisson_variational_dem_q_compl_hat_q_dem11_cont_alt_izom_CITE_CL}),
that is, as in \cite{CLBichir_bib_Gir_Rav1986,
CLBichir_bib_Glowinski1984}. We do not use a variational problem
related to $\partial \Omega$.

\end{rem}

\begin{rem}
\label{observatia_EC_EXACT_ecuatii_ech_T_aaa} The following
equation, which corresponds to
(\ref{e1_6_ecStokes_ecPoisson0_completata}),
\begin{eqnarray}
   && - \alpha \langle \widetilde{S}_{\textbf{z}},\bar{\lambda} \rangle
            + \langle \widetilde{S}_{\textbf{t}},\bar{\lambda} \rangle
            + (grad \, p,grad \, \bar{\lambda})
            =\langle \psi,\bar{\lambda} \rangle, \
              \forall \, \bar{\lambda}  \, \in \, H_{0}^{1}(\Omega) \, ,
         \label{VAR_mathcal_P2_e1_6_ecStokes_ecPoisson_variational_dem_q_compl_dem11_cont_alt_izom_PRESIUNE}
\end{eqnarray}
is obtained from
(\ref{VAR_mathcal_P2_e1_6_ecStokes_ecPoisson_variational_dem_q_compl_hat_q_dem11_cont_alt_izom_CALC_inlocuire_INTERMEDIAR_NOU}).
\end{rem}

\begin{rem}
\label{observatia_EC_EXACT_ecuatii_ech_T}

Replacing
(\ref{mathcal_P2_e1_6_ecStokes_ecPoisson_variational_dem_q_compl_p_dem11_cont_alt_izom}),
(\ref{e1_11_ecStokes_ecPoisson_variational_dem_grad_miu_ec_TAU})
and
(\ref{VAR_mathcal_P2_e1_6_ecStokes_ecPoisson_variational_dem_q_compl_dem11_cont_alt_izom_instead2_delta_hat_delta_aprox_D_VAR_ec})
in (\ref{e1_11_ecStokes_ecPoisson_variational_dem_grad_miu}), it
results
\begin{eqnarray}
   && (\textbf{f}_{\triangle} - \hat{\textbf{f}} - grad \, \hat{g},grad \, \varphi)
         \label{e1_11_ecStokes_ecPoisson_variational_dem_grad_miu_inlocuiri}
\end{eqnarray}
$$ - (grad \, (q+\hat{q}),grad \, \varphi)
      + (grad \, r,grad \, \varphi)
            = \langle \varrho,\varphi \rangle, \
              \forall \, \varphi  \, \in \, H^{1}(\Omega) \, . $$
Let $\tilde{\psi}_{r}$ $\in$ $(H^{1}(\Omega))'$ be defined by
\begin{eqnarray}
   && \langle \tilde{\psi}_{r},\varphi \rangle
      = (- \textbf{f}_{\triangle} + \hat{\textbf{f}} + grad \, \hat{g},grad \, \varphi)
         \label{e1_11_ecStokes_ecPoisson_variational_dem_grad_miu_inlocuiri_def} \\
   && \qquad + \langle \varrho,\varphi \rangle
            + \langle \hat{\varrho},\varphi \rangle, \
              \forall \, \varphi  \, \in \, H^{1}(\Omega) \, ,
         \nonumber
\end{eqnarray}
Summing
(\ref{e1_11_ecStokes_ecPoisson_variational_dem_grad_miu_inlocuiri})
and (\ref{e1_11_ecStokes_ecPoisson_variational_dem_grad_miu_r}),
we obtain that $r$ is the solution of the equation
\begin{eqnarray}
   && ((r,\varphi))_{1}
      = \langle \tilde{\psi}_{r},\varphi \rangle, \
              \forall \, \varphi  \, \in \, H^{1}(\Omega) \, ,
         \label{e1_11_ecStokes_ecPoisson_variational_dem_grad_miu_inlocuiri_ec}
\end{eqnarray}
so $r$ $=$ $B_{1}\tilde{\psi}_{r}$. Then, in
(\ref{e1_11_ecStokes_ecPoisson_variational_dem_grad_miu}), $(grad
\, r,grad \, \varphi)$ is $(grad \, B_{1}\tilde{\psi}_{r},grad \,
\varphi)$.

Equation
(\ref{e1_11_ecStokes_ecPoisson_variational_dem_grad_miu_r}) is
written
\begin{eqnarray}
   && ((r,\tilde{\varphi}))_{1}
      + (grad \, (q+\hat{q}),grad \, \tilde{\varphi})
         \label{e1_11_ecStokes_ecPoisson_variational_dem_grad_miu_inlocuiri_ec_op} \\
   && \qquad - (grad \, r,grad \, \tilde{\varphi})
      = \langle \hat{\varrho},\tilde{\varphi} \rangle, \
              \forall \, \tilde{\varphi}  \, \in \, H^{1}(\Omega) \, ,
         \nonumber
\end{eqnarray}
so
\begin{eqnarray}
   && r - B_{1}G_{r}(q,\hat{q},r)
         = B_{1}\hat{\varrho} \, ,
         \label{e1_11_ecStokes_ecPoisson_variational_dem_grad_miu_inlocuiri_ec_op_cont}
\end{eqnarray}

From
(\ref{e1_11_ecStokes_ecPoisson_variational_dem_grad_miu_ec_TAU})
and
(\ref{VAR_mathcal_P2_e1_6_ecStokes_ecPoisson_variational_dem_q_compl_dem11_cont_alt_izom_instead2_delta_hat_delta_aprox_D_VAR_ec}),
we deduce
\begin{eqnarray}
   && (\textbf{z} - \textbf{t} + grad \, p,grad \, \varphi)
            = (\textbf{f}_{\triangle} - \hat{\textbf{f}},grad \, \varphi), \
              \forall \, \varphi  \, \in \, H^{1}(\Omega) \, ,
         \label{e1_11_ecStokes_ecPoisson_variational_dem_grad_miu_inlocuiri_sum_2}
\end{eqnarray}
The same is obtained if we replace
(\ref{mathcal_P2_e1_6_ecStokes_ecPoisson_variational_dem_q_compl_p_dem11_cont_alt_izom}),
(\ref{e1_11_ecStokes_ecPoisson_variational_dem_grad_miu_ec_TAU})
in (\ref{e1_11_ecStokes_ecPoisson_variational_dem_grad_miu}) and
then sum
(\ref{e1_11_ecStokes_ecPoisson_variational_dem_grad_miu_r}) to
(\ref{e1_11_ecStokes_ecPoisson_variational_dem_grad_miu}), where
we use
(\ref{e1_11_ecStokes_ecPoisson_variational_dem_grad_miu_inlocuiri_ec}).
In both cases, we have a form of equation
(\ref{e1_11_ecStokes_ecPoisson_variational_dem_ec_Poisson_pt_p_CONT2}).

\end{rem}

\subsection{The well-posedness of the exact extended system}
\label{sectiunea_forme_echivalente_ale_ecuatiilor_extended_2}

Let $\widehat{\Phi}$ be the operator defined by the left hand side
of problem
(\ref{mathcal_P2_e1_6_ecStokes_ecPoisson_variational_dem_q_compl_p_dem11_cont_alt_izom})
-
(\ref{e1_11_ecStokes_ecPoisson_variational_dem_grad_miu_ec_TAU}).
For this problem, the spaces $\widehat{\Gamma}$,
$\widehat{\Sigma}$, $\widehat{\Gamma}_{0}$, $\widehat{\Sigma}_{0}$
are defined in Section \ref{sectiunea_ecuatii_remarks}.
$\widehat{\Phi}$ $\in$ $L(\widehat{\Gamma},\widehat{\Sigma})$.
Then, problem
(\ref{mathcal_P2_e1_6_ecStokes_ecPoisson_variational_dem_q_compl_p_dem11_cont_alt_izom})
- (\ref{e1_11_ecStokes_ecPoisson_variational_dem_grad_miu_ec_TAU})
has the forms
(\ref{well_posed_e1_11_ecStokes_ecPoisson_variational_dem_q_compl4_ec_mathcal_X_op_E_fi_GEN})
and
(\ref{well_posed_e1_11_ecStokes_ecPoisson_variational_dem_q_compl4_ec_mathcal_X_op_E_fi}).
Let us use Lemma \ref{well_posed_lema_deschisa} in order to prove
that $\widehat{\Phi}$ is an isomorphism of $\widehat{\Gamma}$ onto
$\widehat{\Sigma}$. Let $\widehat{\Phi}_{0}$ be the restriction of
$\widehat{\Phi}$ to $\widehat{\Gamma}_{0}$. Consider equation
(\ref{well_posed_e1_11_ecStokes_ecPoisson_variational_dem_q_compl4_ec_mathcal_X_op_fi})
in this particular case.

\begin{lem}
\label{lema_0_teorema_izom_1} $\widehat{\Phi}_{0}$ is a bijection
of $\widehat{\Gamma}_{0}$ onto $\widehat{\Sigma}_{0}$.
\end{lem}

\begin{proof} We constructed $\widehat{\Gamma}_{0}$ and
$\widehat{\Sigma}_{0}$ such that
$\widehat{\Phi}(\widehat{\Gamma}_{0})$ $\subseteq$
$\widehat{\Sigma}_{0}$. We have $\widehat{\Phi}_{0}(x,U)$ $=$
$\widehat{\Phi}(x,U)$, $\forall$ $(x,U)$ $\in$
$\widehat{\Gamma}_{0}$.

Fix $\widehat{\mathcal{F}}_{0}$ $=$ $(\hat{g}$, $\ell_{S}$, $g$,
$\textbf{f}$, $\psi$, $\hat{\psi}$, $\textbf{f}_{\triangle}$,
$\varrho$, $\hat{\varrho}$, $\hat{\textbf{f}})$ $\in$
$\widehat{\Sigma}_{0}$.

By the density of $\mathcal{D}(\Omega)^{N}$ in
$\textbf{H}_{0}^{1}(\Omega)$, from
(\ref{e1_11_ecStokes_ecPoisson_variational_dem}), we obtain
equation (\ref{e1_11_ecStokes}) in the sense of distributions,
\begin{eqnarray}
   && \langle -\triangle \textbf{u} + grad \, p_{S} - \textbf{f}, \hat{\varphi} \rangle = 0,
              \forall \, \hat{\varphi}  \, \in \, \mathcal{D}(\Omega)^{N} \, ,
         \label{e1_11_ecStokes_ecPoisson_variational_dem_distr}
\end{eqnarray}

Let $\varphi$ $\in$ $\mathcal{D}(\Omega)$. Then, $\langle
\triangle \textbf{u} , grad \, \varphi \rangle$ $=$ $\langle
\textbf{u} , \triangle (grad \, \varphi) \rangle$ $=$ $\langle
\textbf{u} , grad \, (\triangle \varphi) \rangle$ $=$ $\langle div
\, \textbf{u} , -\triangle \varphi \rangle$. Using
(\ref{e1_11_ecStokes_ecPoisson_variational_dem_distr}) with
$\hat{\varphi}$ $=$ $grad \, \varphi$, we deduce
\begin{eqnarray}
   && \langle div \, \textbf{u} ,-\triangle \varphi \rangle
      = \langle grad \, p_{S},grad \, \varphi \rangle
         - \langle \textbf{f} , grad \, \varphi \rangle, \
              \forall \, \varphi  \, \in \, \mathcal{D}(\Omega) \, .
         \label{e5_1_cond_ecStokes_ecPoisson_variational_dem_grad_y_distr_pStokes}
\end{eqnarray}

Using
(\ref{e5_1_cond_ecStokes_ecPoisson_variational_dem_div_corelatie_exact}),
we obtain
\begin{eqnarray}
   && \langle div \, \textbf{u} ,-\triangle \varphi \rangle
      = \langle grad \, (p+\ell_{S}),grad \, \varphi \rangle
         - \langle \textbf{f} , grad \, \varphi \rangle, \
              \forall \, \varphi  \, \in \, \mathcal{D}(\Omega) \, .
         \label{e5_1_cond_ecStokes_ecPoisson_variational_dem_grad_y_distr}
\end{eqnarray}

Taking the test functions $\varphi$ $\in$ $H_{0}^{1}(\Omega)$ in
(\ref{e1_11_ecStokes_ecPoisson_variational_dem_grad_miu}),
replacing $\hat{p}$ from
(\ref{mathcal_P2_e1_6_ecStokes_ecPoisson_variational_dem_q_compl_p_dem11_cont_alt_izom})
in (\ref{e1_11_ecStokes_ecPoisson_variational_dem_grad_miu}),
summing
(\ref{mathcal_P2_e1_6_ecStokes_ecPoisson_variational_dem_q_compl_dem11_cont_alt_izom}),
(\ref{mathcal_P2_e1_6_ecStokes_ecPoisson_variational_dem_q_compl_hat_q_dem11_cont_alt_izom})
to (\ref{e1_11_ecStokes_ecPoisson_variational_dem_grad_miu}),
using
(\ref{VAR_mathcal_P2_e1_6_ecStokes_ecPoisson_variational_dem_q_compl_dem11_cont_alt_izom_instead2_delta_hat_delta_aprox_D_VAR_ec})
and $r$ $=$ $B_{1}\tilde{\psi}_{r}$ from
(\ref{e1_11_ecStokes_ecPoisson_variational_dem_grad_miu_inlocuiri_ec}),
we obtain
\begin{eqnarray}
   && (grad \, p,grad \, \tilde{\lambda})
            =(grad \, \hat{g},grad \, \tilde{\lambda})
            +\alpha \langle \widetilde{S}_{\textbf{z}},\tilde{\lambda} \rangle
            -\langle \widetilde{S}_{\textbf{t}},\bar{\lambda} \rangle
          \label{VAR_mathcal_P2_e1_6_ecStokes_ecPoisson_variational_dem_q_compl_hat_q_dem11_cont_alt_izom_CALC_inlocuire_INTERMEDIAR_NOU} \\
   && +\langle \psi,\tilde{\lambda} \rangle
            +\langle \hat{\psi},\tilde{\lambda} \rangle
            -\langle \textbf{f}_{\triangle},grad \, \tilde{\lambda} \rangle
            +\langle \varrho,\tilde{\lambda} \rangle
            -(grad \, (B_{1}\tilde{\psi}_{r}),grad \, \tilde{\lambda})
         \nonumber \\
   && +\langle \hat{\textbf{f}},grad \, \tilde{\lambda} \rangle, \
              \forall \, \tilde{\lambda}  \, \in \, \mathcal{D}(\Omega) \, ,
         \nonumber
\end{eqnarray}

Using
(\ref{VAR_mathcal_P2_e1_6_ecStokes_ecPoisson_variational_dem_q_compl_dem11_cont_alt_izom_instead2_delta_hat_delta_aprox_D_VAR_ec})
and
(\ref{e1_11_ecStokes_ecPoisson_variational_dem_grad_miu_ec_TAU}),
relation
(\ref{VAR_mathcal_P2_e1_6_ecStokes_ecPoisson_variational_dem_q_compl_hat_q_dem11_cont_alt_izom_CALC_inlocuire_INTERMEDIAR_NOU})
becomes:
\begin{eqnarray}
   && (1-\alpha) (grad \, p,grad \, \varphi)
            =(grad \, \hat{g},grad \, \varphi)
            +\langle \psi,\varphi \rangle
            +\langle \hat{\psi},\varphi \rangle
          \label{VAR_mathcal_P2_e1_6_ecStokes_ecPoisson_variational_dem_q_compl_hat_q_dem11_cont_alt_izom_CALC_inlocuire_D_NOU} \\
   && - (1+\alpha) \langle \textbf{f}_{\triangle},grad \, \varphi \rangle
            +\langle \varrho,\varphi \rangle
            -(grad \, (B_{1}\tilde{\psi}_{r}),grad \, \varphi)
         \nonumber \\
   && +(1+\gamma) \langle \hat{\textbf{f}},grad \, \varphi \rangle, \
              \forall \, \varphi  \, \in \, \mathcal{D}(\Omega) \, ,
         \nonumber
\end{eqnarray}

Relation
(\ref{e5_1_cond_ecStokes_ecPoisson_variational_dem_grad_y_distr})
and
(\ref{VAR_mathcal_P2_e1_6_ecStokes_ecPoisson_variational_dem_q_compl_hat_q_dem11_cont_alt_izom_CALC_inlocuire_D_NOU})
give
\begin{eqnarray}
   && (div \, \textbf{u} , -\triangle \varphi)
      = \langle \psi_{div},\varphi \rangle, \
              \forall \, \varphi  \, \in \, \mathcal{D}(\Omega) \, .
         \label{e5_1_cond_ecStokes_ecPoisson_variational_dem_grad_y_distr_cont}
\end{eqnarray}
where $\psi_{div}$ $\in$ $H^{-1}(\Omega)$ is
\begin{eqnarray}
   && \langle \psi_{div},\varphi \rangle
      = (1-\alpha)^{-1}
            (grad \, \hat{g},grad \, \varphi)
      + \langle grad \, \ell_{S},grad \, \varphi \rangle
         \label{VAR_e5_1_cond_ecStokes_ecPoisson_variational_dem_grad_y_distr_NOU} \\
   && + \langle \widetilde{S}_{\textbf{f}},\varphi \rangle
      +(1-\alpha)^{-1}(\langle \psi,\varphi \rangle
            +\langle \hat{\psi},\varphi \rangle
            + (1+\alpha) \langle \widetilde{S}_{\textbf{f}_{\triangle}},\varphi \rangle
         \nonumber \\
   && +\langle \varrho,\varphi \rangle
            - (grad \, (B_{1}\tilde{\psi}_{r}),grad \, \varphi)
      - (1+\gamma) \langle \widetilde{S}_{\hat{\textbf{f}}},\varphi \rangle), \
              \forall \, \varphi  \, \in \, H_{0}^{1}(\Omega) \, ,
         \nonumber
\end{eqnarray}

Let $\psi_{g}$ be given by the sum of some distributional
derivatives of $g$ by
\begin{eqnarray}
   && (g , \triangle \varphi)
      = \langle \psi_{g},\varphi \rangle, \
              \forall \, \varphi  \, \in \, \mathcal{D}(\Omega) \, .
         \label{e5_1_cond_ecStokes_ecPoisson_variational_dem_grad_y_distr_cont_g_dem}
\end{eqnarray}

Equation (\ref{e5_1_cond_ecStokes_ecPoisson_variational_dem_div})
gives $div \, \textbf{u} = y + g$ and with
(\ref{e5_1_cond_ecStokes_ecPoisson_variational_dem_grad_y_distr_cont})
and
(\ref{e5_1_cond_ecStokes_ecPoisson_variational_dem_grad_y_distr_cont_g_dem}),
we have
\begin{eqnarray}
   && ( y , -\triangle \varphi)
      = \langle \psi_{g} + \psi_{div}, \varphi \rangle, \
              \forall \, \varphi  \, \in \, \mathcal{D}(\Omega) \, .
         \label{e5_1_cond_ecStokes_ecPoisson_variational_dem_grad_y2_H2_cont_SUMA_Drond}
\end{eqnarray}

The density of $\mathcal{D}(\Omega)$ in $H_{0}^{2}(\Omega)$ leads
to
\begin{eqnarray}
   && ( y , -\triangle \varphi)
      = \langle \psi_{g} + \psi_{div}, \varphi \rangle, \
              \forall \, \varphi  \, \in \, H_{0}^{2}(\Omega) \, .
         \label{e5_1_cond_ecStokes_ecPoisson_variational_dem_grad_y2_H2}
\end{eqnarray}
or, using the space $\mathcal{W}$ as the space $\mathcal{V}$ is
used in Theorem 7.1.1, page 384, \cite{CLBichir_bib_Ciarlet2002}
($\triangle \varphi \, \in \, L^{2}(\Omega)$ since $\varphi \, \in
\, H_{0}^{2}(\Omega)$),
\begin{eqnarray}
   && ( y , -\triangle \varphi)
      = \langle \psi_{g} + \psi_{div}, \varphi \rangle, \
              \forall \, (\varphi,-\triangle \varphi)  \, \in \, \mathcal{W} \, .
         \label{e5_1_cond_ecStokes_ecPoisson_variational_dem_grad_y2_H2_cont_SUMA}
\end{eqnarray}

Since $y$ $\in$ $T_{PD}^{-1}(D_{0}$ $\cap$ $H_{0}^{2})$, there
exists $\phi$ $\in$ $H_{0}^{2}(\Omega)$ such that
\begin{eqnarray}
   && (\phi,y)  \, \in \, \mathcal{W} \, ,
         \label{e5_1_cond_ecStokes_ecPoisson_variational_dem_modif3}
\end{eqnarray}

Following the technique from the proof of Theorem 7.1.2, page 385,
\cite{CLBichir_bib_Ciarlet2002}, we infer that
(\ref{e5_1_cond_ecStokes_ecPoisson_variational_dem_modif3}) and
(\ref{e5_1_cond_ecStokes_ecPoisson_variational_dem_grad_y2_H2_cont_SUMA})
are equivalent to
(\ref{e5_1_cond_ecStokes_ecPoisson_variational_dem_modif3}) and
(\ref{e5_1_cond_ecStokes_ecPoisson_variational_dem_grad_y2_H2_biarmonica_ini}),
where
\begin{eqnarray}
   && (\triangle \phi , \triangle \varphi)
      = \langle \psi_{g} + \psi_{div}, \varphi \rangle, \
              \forall \, \varphi  \, \in \, H_{0}^{2}(\Omega) \, .
         \label{e5_1_cond_ecStokes_ecPoisson_variational_dem_grad_y2_H2_biarmonica_ini}
\end{eqnarray}
Equation
(\ref{e5_1_cond_ecStokes_ecPoisson_variational_dem_grad_y2_H2_biarmonica_ini})
satisfies the conditions of Proposition I.1.3, page 16,
\cite{CLBichir_bib_Gir_Rav1986} (where we take $u_{0} = 0$), for a
bounded and connected open subset of $\mathbb{R}^{N}$ ($N=2,3$)
with a Lipschitz - continuous boundary $\partial \Omega$, so this
equation has an unique solution $\phi$ $\in$ $H_{0}^{2}(\Omega)$.

$y$ is uniquely determined by
(\ref{e5_1_cond_ecStokes_ecPoisson_variational_dem_modif3}).

$(\textbf{u},p_{S})$ is the solution of the Stokes problem given
by (\ref{e1_11_ecStokes_ecPoisson_variational_dem}),
(\ref{e5_1_cond_ecStokes_ecPoisson_variational_dem_div}).

$p$, $\textbf{z}$, $\textbf{t}$, $q$, $\hat{p}$, $\hat{q}$, $r$
are obtained from
(\ref{e5_1_cond_ecStokes_ecPoisson_variational_dem_div_corelatie_exact}),
(\ref{mathcal_P2_e1_6_ecStokes_ecPoisson_variational_dem_q_compl_dem11_cont_alt_izom_instead2_delta_hat_delta_aprox}),
(\ref{e1_11_ecStokes_ecPoisson_variational_dem_grad_miu_ec_TAU}),
(\ref{mathcal_P2_e1_6_ecStokes_ecPoisson_variational_dem_q_compl_dem11_cont_alt_izom}),
(\ref{mathcal_P2_e1_6_ecStokes_ecPoisson_variational_dem_q_compl_p_dem11_cont_alt_izom}),
(\ref{mathcal_P2_e1_6_ecStokes_ecPoisson_variational_dem_q_compl_hat_q_dem11_cont_alt_izom}),
(\ref{e1_11_ecStokes_ecPoisson_variational_dem_grad_miu_inlocuiri_ec})
respectively. So $\widehat{\Phi}_{0}$ is surjective.

If $\widehat{\mathcal{F}}_{0}$ $=$ $(\hat{g}$, $\ell_{S}$, $g$,
$\textbf{f}$, $\psi$, $\hat{\psi}$, $\textbf{f}_{\triangle}$,
$\varrho$, $\hat{\varrho}$, $\hat{\textbf{f}})$ $=$ $0$, we deduce
$\widehat{U}_{0}$ $=$ $(\hat{p}, p_{S}, y, \textbf{u}, q, \hat{q},
\textbf{z}, p, r, \textbf{t})$ $=$ $0$ uniquely determined.
$Ker(\widehat{\Phi}_{0})$ $=$ $\{ 0 \}$ and $\widehat{\Phi}_{0}$
is injective. So $\widehat{\Phi}_{0}$ is a bijection of
$\widehat{\Gamma}_{0}$ onto $\widehat{\Sigma}_{0}$.

\qquad
\end{proof}

\begin{lem}
\label{well_posed_lema_surjectiva} $\widehat{\Phi}$ is a linear,
continuous and surjective mapping.
\end{lem}

\begin{proof}

Fix $\widehat{\mathcal{F}}$ and $\widehat{\mathcal{F}}_{0}$ $\in$
$\widehat{\Sigma}$.

$\widehat{\mathcal{F}}$ $=$ $(\hat{g}$, $\ell_{S}$, $g$,
$\textbf{f}$, $\psi$, $\hat{\psi}$, $\textbf{f}_{\triangle}$,
$\varrho$, $\hat{\varrho}$, $\hat{\textbf{f}})$.

$\widehat{\mathcal{F}}_{0}$ $=$ $(\hat{g}$, $(\ell_{S})_{0}$,
$g_{0}$, $\textbf{f}_{0}$, $\psi$, $\hat{\psi}$,
$\textbf{f}_{\triangle}$, $\varrho$, $\hat{\varrho}$,
$\hat{\textbf{f}})$, where the index "$0$" is written only for the
elements that do not coincide with the corresponding element from
$\widehat{\mathcal{F}}$.

$\widehat{\Phi}_{0}:\widehat{\Gamma}_{0}$ $\rightarrow$
$\widehat{\Sigma}_{0}$ is bijective. Let $\widehat{U}_{0}$ $=$
$(x_{0},U_{0})$ $=$ $(\hat{p}_{0}, (p_{S})_{0}, y_{0},
\textbf{u}_{0}, q_{0}, \hat{q}_{0}, \textbf{z}_{0}, p_{0}, r_{0},
\textbf{t}_{0})$ be the solution of
\begin{eqnarray}
   && \widehat{\Phi}_{0}(\widehat{U}) = \widehat{\mathcal{F}}_{0} \, ,
         \label{well_posed_e1_11_ecStokes_ecPoisson_variational_dem_q_compl4_ec_mathcal_X_op_fi_surj}
\end{eqnarray}

We have $\widehat{\mathcal{F}} - \widehat{\mathcal{F}}_{0}$ $=$
$(0$, $\ell_{S}-(\ell_{S})_{0}$, $g-g_{0}$,
$\textbf{f}-\textbf{f}_{0}$, $0$, $0$, $0$, $0$, $0$, $0)$.
Consider the equation
\begin{eqnarray}
   && \widehat{\Phi}(\widehat{V}) = \widehat{\mathcal{F}} - \widehat{\mathcal{F}}_{0} \, ,
         \label{well_posed_e1_11_ecStokes_ecPoisson_variational_dem_q_compl4_ec_mathcal_X_op_E_fi_surj_dif}
\end{eqnarray}

Let us determine a solution $\widehat{U}'$ $=$ $(x',U')$ $=$
$(\hat{p}', p_{S}', y', \textbf{u}', q', \hat{q}', \textbf{z}',
p', r', \textbf{t}')$ of
(\ref{well_posed_e1_11_ecStokes_ecPoisson_variational_dem_q_compl4_ec_mathcal_X_op_E_fi_surj_dif}).

If we take $\hat{p}'$ $=$ $0$, $q'$ $=$ $0$, $\hat{q}'$ $=$ $0$,
$\textbf{z}'$ $=$ $0$, $p'$ $=$ $0$, $r'$ $=$ $0$, $\textbf{t}'$
$=$ $0$, then, from the equations
(\ref{mathcal_P2_e1_6_ecStokes_ecPoisson_variational_dem_q_compl_p_dem11_cont_alt_izom})
-
(\ref{e1_11_ecStokes_ecPoisson_variational_dem_grad_miu_ec_TAU}),
it remains
\begin{eqnarray}
   && p_{S}' = \ell_{S}-(\ell_{S})_{0} \, ,
         \label{e5_1_cond_ecStokes_ecPoisson_variational_dem_div_corelatie_exact_prim} \\
   && (div \, \textbf{u}'-y',\hat{\mu}) = (g-g_{0},\hat{\mu}), \
              \forall \, \hat{\mu}  \, \in \, L_{0}^{2}(\Omega) \, ,
         \label{e5_1_cond_ecStokes_ecPoisson_variational_dem_div_prim} \\
   && (grad \, \textbf{u}',grad \, \textbf{w})-(p_{S}',div \, \textbf{w})=\langle \textbf{f}-\textbf{f}_{0},w \rangle, \
              \forall \, \textbf{w}  \, \in \, \textbf{H}_{0}^{1}(\Omega) \, ,
         \label{e1_11_ecStokes_ecPoisson_variational_dem_prim}
\end{eqnarray}

We obtain $p_{S}'$ $=$ $\ell_{S} - (\ell_{S})_{0}$ from
(\ref{e5_1_cond_ecStokes_ecPoisson_variational_dem_div_corelatie_exact_prim}),
$\textbf{u}'$ from
(\ref{e1_11_ecStokes_ecPoisson_variational_dem_prim}) and $y'$
from
(\ref{e5_1_cond_ecStokes_ecPoisson_variational_dem_div_prim}), so
$\widehat{U}'$ is a solution of the equation
(\ref{well_posed_e1_11_ecStokes_ecPoisson_variational_dem_q_compl4_ec_mathcal_X_op_E_fi_surj_dif}).
We have
\begin{eqnarray}
   && \widehat{\Phi}(\widehat{U}_{0} + \widehat{U}') = \widehat{\mathcal{F}} \, ,
         \label{well_posed_e1_11_ecStokes_ecPoisson_variational_dem_q_compl4_ec_mathcal_X_op_E_fi_surj}
\end{eqnarray}
Hence $\widehat{\Phi}$ is a surjective mapping.

\qquad
\end{proof}

\begin{rem}
\label{observatia_EC_EXACT_tau_prim} $\widehat{U}'$ and
$\widehat{\mathcal{F}} - \widehat{\mathcal{F}}_{0}$ also satisfy
(\ref{e1_11_ecStokes_ecPoisson_variational_dem_diferenta}).
$\forall \, \textbf{w}  \, \in \, \textbf{H}_{0}^{1}(\Omega)$, we
have
\begin{eqnarray*}
   && (grad \, \textbf{u}',grad \, \textbf{w})
      - (\ell_{S} - (\ell_{S})_{0},div \, \textbf{w})-\langle \textbf{f}-\textbf{f}_{0},w \rangle
      = (\textbf{z}',\textbf{w}) \, .
         \label{e1_11_ecStokes_ecPoisson_variational_dem_diferenta_prim}
\end{eqnarray*}
\end{rem}

\begin{thm}
\label{teorema_izom_1} $\widehat{\Phi}$ is an isomorphism of
$\widehat{\Gamma}$ onto $\widehat{\Sigma}$.
\end{thm}

\begin{proof} $\widehat{\Gamma}$, $\widehat{\Sigma}$, $\widehat{\Gamma}_{0}$, $\widehat{\Sigma}_{0}$,
$\widehat{\Phi}$, $\widehat{\Phi}_{0}$ satisfy the hypotheses of
Lemma \ref{well_posed_lema_deschisa}.

\qquad
\end{proof}

\begin{rem}
\label{observatia5_omega_domega222_DEM} We do not impose condition
(\ref{dem101_Gir_Rav_pag51_CL_Neumann_G_P}). This one arises for
$y$ $\in$ $T_{PD}^{-1}(D_{0}$ $\cap$ $H_{0}^{2}(\Omega))$ in the
form $\phi$ $\in$ $D_{0}$ $\cap$ $H_{0}^{2}(\Omega)$, related to
some dense subspaces. See
(\ref{e5_1_cond_ecStokes_ecPoisson_variational_dem_modif3}). In
this situation, we also have $\textbf{u}$ $\in$ $V \oplus
\widetilde{V^{\bot}}$.
\end{rem}

\begin{cor}
\label{corolarul5_teorema_izom_1_cor10_div0} Let $\textbf{f}$
$\in$ $\Upsilon_{1}$ $\cap$ $\textbf{L}^{2}(\Omega)$ be given. If
$\hat{g}=0$, $\ell_{S}=0$, $g=0$, $\psi=0$, $\hat{\psi}=0$,
$\textbf{f}_{\triangle}=\textbf{f}$, $\varrho=0$,
$\hat{\varrho}=0$, $\hat{\textbf{f}}=\textbf{f}$ in problem
(\ref{mathcal_P2_e1_6_ecStokes_ecPoisson_variational_dem_q_compl_p_dem11_cont_alt_izom})
-
(\ref{e1_11_ecStokes_ecPoisson_variational_dem_grad_miu_ec_TAU}),
then $(\textbf{u},p)$ $\in$ $\textbf{H}_{0}^{1}(\Omega)$ $\times$
$(H^{1}(\Omega) \cap L_{0}^{2}(\Omega))$ is the unique solution of
problem (\ref{e1_11_ecStokes_ecPoisson_variational_dem_th_I_5_1})
- (\ref{e5_1_cond_ecStokes_ecPoisson_variational_dem_th_I_5_1})
with $g=0$.
\end{cor}

\begin{proof} We obtain $\psi_{div} = 0$, $\psi_{g} = 0$, so $\phi =0$, and so
on.

\qquad
\end{proof}

\begin{cor}
\label{corolarul5_teorema_izom_1_cor10} Let $\textbf{f}$ $\in$
$\Upsilon_{1}$ $\cap$ $\textbf{L}^{2}(\Omega)$ be given. To solve
the problem
(\ref{e1_11_ecStokes_ecPoisson_variational_dem_th_I_5_1}) -
(\ref{e5_1_cond_ecStokes_ecPoisson_variational_dem_th_I_5_1}) with
$g=0$ is equivalent to solve the problem
(\ref{mathcal_P2_e1_6_ecStokes_ecPoisson_variational_dem_q_compl_p_dem11_cont_alt_izom})
- (\ref{e1_11_ecStokes_ecPoisson_variational_dem_grad_miu_ec_TAU})
with $\hat{g}=0$, $\ell_{S}=0$, $g=0$, $\psi=0$, $\hat{\psi}=0$,
$\textbf{f}_{\triangle}=\textbf{f}$, $\varrho=0$,
$\hat{\varrho}=0$, $\hat{\textbf{f}}=\textbf{f}$.
\end{cor}

\begin{cor}
\label{corolarul5_teorema_izom_1_beta_exact} Under the hypotheses
of Corollaries \ref{corolarul5_teorema_izom_1_cor10_div0} and
\ref{corolarul5_teorema_izom_1_cor10}, for $\alpha_{1}$ $\neq$
$\alpha_{2}$, we have $(x(\alpha_{1}),U(\alpha_{1}))$ $=$
$(x(\alpha_{2}),U(\alpha_{2}))$, where
$(x(\alpha_{i}),U(\alpha_{i}))$ is the solution of
(\ref{mathcal_P2_e1_6_ecStokes_ecPoisson_variational_dem_q_compl_p_dem11_cont_alt_izom})
- (\ref{e1_11_ecStokes_ecPoisson_variational_dem_grad_miu_ec_TAU})
for $\alpha_{i}$.
\end{cor}

\begin{proof} Under the hypotheses
of the Corollary, $\psi_{div}$ $=$ $0$, $\forall$ $\alpha$ $\neq$
$1$.

\qquad
\end{proof}

\begin{cor}
\label{observatia5_omega_domega_DEM_beta} Under the hypotheses of
Corollaries \ref{corolarul5_teorema_izom_1_cor10_div0} and
\ref{corolarul5_teorema_izom_1_cor10}, $(\textbf{z},grad \,
\varphi)$ $=$ $0$, $\forall$ $\varphi$ $\in$ $H_{0}^{1}(\Omega)$.
\end{cor}

\begin{proof} Using (\ref{mathcal_P2_e1_6_ecStokes_ecPoisson_variational_dem_q_compl_dem11_cont_alt_izom_instead2_delta_hat_delta_aprox}), we deduce
\begin{eqnarray}
   && (\textbf{z},grad \, \varphi) + (grad \, p,grad \, \varphi) = (\textbf{f}_{\triangle},grad \, \varphi),
              \forall \, \varphi  \, \in \, \mathcal{D}(\Omega) \, ,
         \label{e1_11_ecStokes_ecPoisson_variational_dem_diferenta_laplacian}
\end{eqnarray}

Subtracting
(\ref{e1_11_ecStokes_ecPoisson_variational_dem_diferenta_laplacian})
from
(\ref{VAR_mathcal_P2_e1_6_ecStokes_ecPoisson_variational_dem_q_compl_hat_q_dem11_cont_alt_izom_CALC_inlocuire_INTERMEDIAR_NOU}),
we obtain
\begin{eqnarray}
   && (1-\alpha) (\textbf{z},grad \, \varphi) = 0,
              \forall \, \varphi  \, \in \, \mathcal{D}(\Omega) \, ,
         \label{e1_11_ecStokes_ecPoisson_variational_dem_distr_grad_ddd_beta}
\end{eqnarray}
Hence, $(\textbf{z},grad \, \varphi)$ $=$ $0$, $\forall$ $\varphi$
$\in$ $H_{0}^{1}(\Omega)$.

\qquad
\end{proof}

\begin{cor}
\label{observatia_ec_Poisson_grad_p} We have
\begin{eqnarray}
   && ( \alpha \textbf{z} - \textbf{t}, grad \, \bar{\lambda} )
            + (grad \, p,grad \, \bar{\lambda})
            =(r,\bar{\lambda})
          \label{VAR_mathcal_P2_e1_6_ecStokes_ecPoisson_variational_dem_q_compl_dem11_cont_alt_izom_PRESIUNE_dir} \\
   && +\langle \psi,\bar{\lambda} \rangle
            +\langle \hat{\psi},\bar{\lambda} \rangle
            -\langle \hat{\varrho},\bar{\lambda} \rangle , \
              \forall \, \bar{\lambda}  \, \in \, H_{0}^{1}(\Omega) \, ,
         \nonumber
\end{eqnarray}
where we use
(\ref{e1_11_ecStokes_ecPoisson_variational_dem_grad_miu_inlocuiri_ec}),
and
\begin{eqnarray}
   && (\textbf{z} - \textbf{t} + grad \, p,grad \, \varphi)
            = (\textbf{f}_{\triangle} - \hat{\textbf{f}},grad \, \varphi), \
              \forall \, \varphi  \, \in \, H^{1}(\Omega) \ominus H_{0}^{1}(\Omega) \, ,
         \label{e1_11_ecStokes_ecPoisson_variational_dem_grad_miu_inlocuiri_sum_2_L2_cor}
\end{eqnarray}
\end{cor}

\begin{proof}
If we add
(\ref{mathcal_P2_e1_6_ecStokes_ecPoisson_variational_dem_q_compl_dem11_cont_alt_izom})
and
(\ref{mathcal_P2_e1_6_ecStokes_ecPoisson_variational_dem_q_compl_hat_q_dem11_cont_alt_izom})
and subtract
(\ref{e1_11_ecStokes_ecPoisson_variational_dem_grad_miu_r}), where
the test functions $\varphi$ $\in$ $H_{0}^{1}(\Omega)$, we have
(\ref{VAR_mathcal_P2_e1_6_ecStokes_ecPoisson_variational_dem_q_compl_dem11_cont_alt_izom_PRESIUNE_dir})
where we use
(\ref{e1_11_ecStokes_ecPoisson_variational_dem_grad_miu_inlocuiri_ec}).

(\ref{e1_11_ecStokes_ecPoisson_variational_dem_grad_miu_inlocuiri_sum_2})
gives
(\ref{e1_11_ecStokes_ecPoisson_variational_dem_grad_miu_inlocuiri_sum_2_L2_cor}).

\qquad
\end{proof}

The parameterized perturbed pressure Poisson equation is given by
(\ref{VAR_mathcal_P2_e1_6_ecStokes_ecPoisson_variational_dem_q_compl_dem11_cont_alt_izom_PRESIUNE_dir})
and
(\ref{e1_11_ecStokes_ecPoisson_variational_dem_grad_miu_inlocuiri_sum_2_L2_cor}).
In the following corollary, we formulate the announced result that
to solve the stationary Stokes problem is equivalent to solve a
problem for the momentum equation, the parameterized perturbed
pressure Poisson equation and the equation that defines the
Laplace operator acting on velocity.

\begin{cor}
\label{corolarul5_teorema_izom_1_cor10_echiv} Let $\textbf{f}$
$\in$ $\Upsilon_{1}$ $\cap$ $\textbf{L}^{2}(\Omega)$ be given. To
solve the problem
(\ref{e1_11_ecStokes_ecPoisson_variational_dem_th_I_5_1}) -
(\ref{e5_1_cond_ecStokes_ecPoisson_variational_dem_th_I_5_1}) with
$g=0$ is equivalent to solve the problem
(\ref{e1_11_ecStokes_ecPoisson_variational_dem}),
(\ref{e1_11_ecStokes_ecPoisson_variational_dem_diferenta})
(\ref{VAR_mathcal_P2_e1_6_ecStokes_ecPoisson_variational_dem_q_compl_dem11_cont_alt_izom_PRESIUNE_dir}),
(\ref{e1_11_ecStokes_ecPoisson_variational_dem_grad_miu_inlocuiri_sum_2_L2_cor})
with $\hat{g}=0$, $\ell_{S}=0$, $g=0$, $\psi=0$, $\hat{\psi}=0$,
$\textbf{f}_{\triangle}=\textbf{f}$, $\varrho=0$,
$\hat{\varrho}=0$, $\hat{\textbf{f}}=\textbf{f}$.
\end{cor}

\begin{proof}
The proof follows from Corollary
\ref{corolarul5_teorema_izom_1_cor10_div0} and the proof of Lemma
\ref{lema_0_teorema_izom_1} using the density.

\qquad
\end{proof}

\section{Construction of another formulation of the isomorphism for the exact problem}
\label{sectiunea_alt_izomorfism}

In the sequel, we write the problem
(\ref{mathcal_P2_e1_6_ecStokes_ecPoisson_variational_dem_q_compl_p_dem11_cont_alt_izom})
- (\ref{e1_11_ecStokes_ecPoisson_variational_dem_grad_miu_ec_TAU})
in the form of equation
(\ref{well_posed_e1_11_ecStokes_ecPoisson_variational_dem_q_compl4_ec_mathcal_X_op_E}).

Let us consider the generic form of a variational equation: find
$v$ $\in$ $E$ such that
\begin{eqnarray}
         a(v,\lambda)=\langle \psi,\lambda \rangle , \
              \forall \, \lambda  \, \in \, E \,
         \label{e1_ecStokes_ecPoisson_variational_scriere_COMPACTA_exact_generic}
\end{eqnarray}
\cite{CLBichir_bib_Atkinson_Han2009,
CLBichir_bib_Bochev_GunzburgerLSFEM2009,
CLBichir_bib_Brenner_Scott2008, CLBichir_bib_Ciarlet2002,
CLBichir_bib_Dautray_Lions_vol6_1988, CLBichir_bib_A_Ern2005,
CLBichir_bib_Ili1980, CLBichir_bib_Quarteroni_Valli2008,
CLBichir_bib_Temam1979, CLBichir_bib_E_Zeidler_NFA_IIA}. $E$ is a
Hilbert space, $\mathcal{E}$ is a dense subset in $E$.
$a(\cdot,\cdot)$ is a continuous, (strongly) coercive bilinear
form on $E \times E$ and $\langle \psi,\cdot \rangle$ is a linear
continuous functional on $E$. $T \in L(E',E)$, $T$ is the solution
operator defined by: given $\psi$ $\in$ $E'$, $v=T\psi$ if and
only if $v$ $\in$ $E$ is the unique solution of
(\ref{e1_ecStokes_ecPoisson_variational_scriere_COMPACTA_exact_generic})
(it exists and it is unique according to Lax - Milgram lemma). $T$
is an isomorphism from $E'$ onto $E$.

Let us associate a set $\mathcal{M}_{i}$ $=$ $\{ i$,
$(\ref{e1_ecStokes_ecPoisson_variational_scriere_COMPACTA_exact_generic})$,
$E$, $E'$, $\mathcal{E}_{i}$, $v$, $\psi$, $T \}$ to equation
(\ref{e1_ecStokes_ecPoisson_variational_scriere_COMPACTA_exact_generic}),
where the first element is an index $i$ to be fixed for an
particular equation, the second element is the number of the
equation. Let us retain the following particular cases
$\mathcal{M}_{1},\ldots,\mathcal{M}_{5}$ of $\mathcal{M}_{i}$ (In
all the cases, the bilinear form $a(\cdot,\cdot)$ is an inner
product):

$\mathcal{M}_{1}$ $=$ $\{ 1$,
$(\ref{e1_ecStokes_ecPoisson_variational_a_BIS})$,
$\textbf{H}_{0}^{1}(\Omega)$, $\textbf{H}^{-1}(\Omega)$,
$\mathcal{E}_{1}$, $\textbf{u}$, $\textbf{f}$, $T_{VPD} \}$ with
\begin{equation}
\label{e1_ecStokes_ecPoisson_variational_a_BIS}
   (grad \, \textbf{u},grad \, \textbf{w})
      =\langle \textbf{f},\textbf{w} \rangle, \
              \forall \, \textbf{w}  \, \in \, \textbf{H}_{0}^{1}(\Omega) \, .
\end{equation}

$\mathcal{M}_{2}$ $=$ $\{ 2$,
$(\ref{e1_ecStokes_ecPoisson_variational_a_BIS_scalar})$,
$H_{0}^{1}(\Omega)$, $H^{-1}(\Omega)$, $\mathcal{E}_{2}$, $q$,
$\psi$, $T_{PD} \}$ with
\begin{equation}
\label{e1_ecStokes_ecPoisson_variational_a_BIS_scalar}
   (grad \, q,grad \, \lambda)
      =\langle \psi,\lambda \rangle, \
              \forall \, \lambda  \, \in \, H_{0}^{1}(\Omega) \, .
\end{equation}

$\mathcal{M}_{3}$ $=$ $\{ 3$,
$(\ref{e1_ecStokes_ecHelmholtz_variational_d_dual})$,
$H^{1}(\Omega)$, $H^{1}(\Omega)'$, $\mathcal{E}_{3}$, $\phi$,
$\psi_{1}$, $B_{1} \}$ with
\begin{equation}
\label{e1_ecStokes_ecHelmholtz_variational_d_dual}
   (grad \, \phi,grad \, \mu) + (\phi,\mu)
      =\langle \psi_{1},\mu \rangle, \
              \forall \, \mu  \, \in \, H^{1}(\Omega) \, .
\end{equation}

$\mathcal{M}_{4}$ $=$ $\{ 4$,
$(\ref{e1_ecStokes_ecHelmholtz_variational_d_dual})$,
$H^{1}(\Omega) \cap L_{0}^{2}(\Omega)$, $B_{1}^{-1}(H^{1}(\Omega)
\cap L_{0}^{2}(\Omega))$, $\mathcal{E}_{4}$, $\phi$, $\psi_{1}$,
$B_{1} \}$.

Let $\pi$ be the identity operator on $\textbf{L}^{2}(\Omega)$. We
have: $\textbf{z}$, $\textbf{f}_{\ast}$ $\in$
$\textbf{L}^{2}(\Omega)$, $\pi \textbf{f}_{\ast}=\textbf{z}$ if
and only if
\begin{equation}
\label{prel_op_hat_pi_h_DEF2H_1_EC_M9}
   (\textbf{z}, \textbf{w})
      =(\textbf{f}_{\ast}, \textbf{w}), \
              \forall \, \textbf{w}  \, \in \, \textbf{L}^{2}(\Omega) \, ,
\end{equation}
$\mathcal{M}_{5}$ $=$ $\{ 5$,
$(\ref{prel_op_hat_pi_h_DEF2H_1_EC_M9})$,
$\textbf{L}^{2}(\Omega)$, $\textbf{L}^{2}(\Omega)$,
$\mathcal{E}_{5}$, $\textbf{z}$, $\textbf{f}_{\ast}$, $\pi \}$.

Let us define the following linear and continuous mappings:

\begin{eqnarray}
   && \mathcal{B}(y,U) = \left[\begin{array}{l}
        p - (q+\hat{p}) \\
        p_{S} - p \\
        div \, \textbf{u}-y
   \end{array}\right] \, ,
         \nonumber \\ 
   && G_{\textbf{u}}:L_{0}^{2}(\Omega)
      \rightarrow \textbf{H}^{-1}(\Omega) \, , \
         \nonumber \\ 
   && \langle G_{\textbf{u}}(p_{S}),\textbf{w} \rangle = (p_{S},div \, \textbf{w}), \
              \forall \, \textbf{w}  \, \in \, \textbf{H}_{0}^{1}(\Omega) \, ,
         \nonumber \\ 
   && G_{q}:\textbf{L}^{2}(\Omega)^{2}
      \rightarrow H^{-1}(\Omega) \, , \
      \langle G_{q}(\textbf{t},\textbf{z}),\bar{\lambda} \rangle = \alpha \langle \widetilde{S}_{\textbf{z}},\bar{\lambda} \rangle - \langle \widetilde{S}_{\textbf{t}},\bar{\lambda} \rangle, \
              \forall \, \bar{\lambda}  \, \in \, H_{0}^{1}(\Omega) \, ,
         \nonumber \\ 
   && G_{\hat{q}}:H^{1}(\Omega) \cap L_{0}^{2}(\Omega)
      \rightarrow H^{-1}(\Omega) \, , \
         \nonumber \\ 
   && \qquad \langle G_{\hat{q}}(p),\hat{\lambda} \rangle
         = - (grad \, p,grad \, \hat{\lambda}), \
            \forall \, \hat{\lambda}  \, \in \, H_{0}^{1}(\Omega) \, ,
         \nonumber \\ 
   && G_{\textbf{z}}:H^{1}(\Omega) \cap L_{0}^{2}(\Omega)
      \rightarrow \textbf{L}^{2}(\Omega) \, , \
      G_{\textbf{z}}(p) = - grad \, p  \, ,
         \nonumber \\ 
   && G_{p}: (H^{1}(\Omega) \cap L_{0}^{2}(\Omega)) \times H^{1}(\Omega) \times H_{0}^{1}(\Omega) \times
         \textbf{L}^{2}(\Omega)^{2}
         \nonumber \\ 
   && \qquad \qquad \times H^{1}(\Omega)
      \rightarrow B_{1}^{-1}(H^{1}(\Omega) \cap L_{0}^{2}(\Omega)) \, , \
         \nonumber \\
   && \qquad \langle G_{p}(p,\hat{p},\hat{q},\textbf{z},\textbf{t},r),\varphi \rangle
         = -(\textbf{z} - \textbf{t} + grad \, (\hat{p}-\hat{q}+r),grad \, \varphi)
         \nonumber \\ 
   && \qquad \qquad +((p,\varphi))_{1}, \
              \forall \, \varphi  \, \in \, H^{1}(\Omega) \, ,
         \nonumber \\
   && G_{r}:H_{0}^{1}(\Omega)^{2} \times
         H^{1}(\Omega)
      \rightarrow H^{1}(\Omega)' \, , \
         \nonumber \\ 
   && \qquad \langle G_{r}(q,\hat{q},r),\varphi \rangle
         =  - (grad \, (q+\hat{q}-r),grad \, \tilde{\varphi}), \
              \forall \, \varphi  \, \in \, H^{1}(\Omega) \, ,
         \nonumber \\
   && G_{\textbf{t}}:\textbf{H}_{0}^{1}(\Omega)
      \rightarrow \textbf{L}^{2}(\Omega) \, , \
      G_{\textbf{t}}(\textbf{u}) = 0 \, ,
         \nonumber  
\end{eqnarray}
where the index $u$ indicates, in $G_{\textbf{u}}$, that this
function is used to construct the equation obtained by the
perturbation of $\textbf{u}$ and so on.

Problem
(\ref{mathcal_P2_e1_6_ecStokes_ecPoisson_variational_dem_q_compl_p_dem11_cont_alt_izom})
- (\ref{e1_11_ecStokes_ecPoisson_variational_dem_grad_miu_ec_TAU})
becomes: find $(x,U)$ $\in$ $\widehat{\Gamma}$ such that
\begin{eqnarray}
   && p - (q+\hat{p}) = \hat{g} \, ,
         \label{mathcal_P2_e1_6_ecStokes_ecPoisson_variational_dem_q_compl_p_dem11_cont_alt_izom_sist} \\
   && p_{S} - p = \ell_{S} \, ,
         \label{e5_1_cond_ecStokes_ecPoisson_variational_dem_div_corelatie_exact_sist} \\
   && div \, \textbf{u}-y = g \, ,
         \label{e5_1_cond_ecStokes_ecPoisson_variational_dem_div_sist} \\
   && \textbf{u} - T_{VPD}G_{\textbf{u}}(p_{S})
         = T_{VPD}\textbf{f} \, ,
         \label{e1_11_ecStokes_ecPoisson_variational_dem_q_compl4} \\
   && q - T_{PD}G_{q}(\textbf{t},\textbf{z})
         = T_{PD}\psi \, ,
         \label{e1_6_ecStokes_ecPoisson_variational_dem_q_compl4} \\
   && \hat{q} - T_{PD}G_{\hat{q}}(p)
         = T_{PD}\hat{\psi} \, ,
         \label{e1_6_ecStokes_ecPoisson_variational_dem_q_compl_hat_q4} \\
   && \textbf{z} - \pi G_{\textbf{z}}(p)
         = \pi \textbf{f}_{\triangle} \, ,
         \label{dem10_cont_alt_izom_instead2_delta_cont} \\
   && p - B_{1}G_{p}(p,\hat{p},\hat{q},\textbf{z},\textbf{t},r)
         = B_{1}\varrho \, ,
         \label{e1_11_ecStokes_ecPoisson_variational_dem_grad_miu_B0} \\
   && r - B_{1}G_{r}(q,\hat{q},r)
         = B_{1}\hat{\varrho} \, ,
         \label{e1_11_ecStokes_ecPoisson_variational_dem_grad_miu_inlocuiri_ec_op_sist} \\
   && \textbf{t} - \pi G_{\textbf{t}}(\textbf{u})
            =\pi \hat{\textbf{f}} \, ,
          \label{e1_11_ecStokes_ecPoisson_variational_dem_grad_miu_ec_TAU_sist}
\end{eqnarray}

\begin{lem}
\label{lema_teorema_izom_2} Problem
(\ref{mathcal_P2_e1_6_ecStokes_ecPoisson_variational_dem_q_compl_p_dem11_cont_alt_izom_sist})
-
(\ref{e1_11_ecStokes_ecPoisson_variational_dem_grad_miu_ec_TAU_sist})
has the form of equation
(\ref{well_posed_e1_11_ecStokes_ecPoisson_variational_dem_q_compl4_ec_mathcal_X_op_E}).
\end{lem}

\begin{proof} Consider the spaces
$\mathcal{Q}$, $\mathcal{X}$, $\mathcal{Y}$, $\mathcal{Z}$ defined
in Section \ref{sectiunea_ecuatii_remarks}. The operator
$\widehat{\mathcal{A}}$ defined by the left hand side of
(\ref{mathcal_P2_e1_6_ecStokes_ecPoisson_variational_dem_q_compl_p_dem11_cont_alt_izom_sist})
-
(\ref{e1_11_ecStokes_ecPoisson_variational_dem_grad_miu_ec_TAU_sist})
has the form of the operator $\widehat{\mathcal{A}}$ defined in
Section
\ref{well_posed_sectiunea_general_theorem_on_the_well_posedness}.

\qquad
\end{proof}

\begin{thm}
\label{teorema_izom_2} Assume the hypotheses of Theorem
\ref{teorema_izom_1}. Then, to solve
(\ref{mathcal_P2_e1_6_ecStokes_ecPoisson_variational_dem_q_compl_p_dem11_cont_alt_izom})
- (\ref{e1_11_ecStokes_ecPoisson_variational_dem_grad_miu_ec_TAU})
is equivalent to solve
(\ref{mathcal_P2_e1_6_ecStokes_ecPoisson_variational_dem_q_compl_p_dem11_cont_alt_izom_sist})
-
(\ref{e1_11_ecStokes_ecPoisson_variational_dem_grad_miu_ec_TAU_sist}).
In other words,
(\ref{well_posed_e1_11_ecStokes_ecPoisson_variational_dem_q_compl4_ec_mathcal_X_op_E_fi})
is equivalent to
(\ref{well_posed_e1_11_ecStokes_ecPoisson_variational_dem_q_compl4_ec_mathcal_X_op_E})
in the particular case considered.
\end{thm}

\begin{thm}
\label{teorema_izom_1_T0} $\mathcal{T}$ is an isomorphism of
$\mathcal{Y}$ onto $\mathcal{X}$. $\mathcal{T}^{-1}$ $=$
$(T_{VPD}^{-1}$, $T_{PD}^{-1}$, $T_{PD}^{-1}$, $\pi^{-1}$,
$B_{1}^{-1}$, $B_{1}^{-1}$, $\pi^{-1})$, $\mathcal{T}^{-1}U$ $=$
$(T_{VPD}^{-1}\textbf{u}$, $T_{PD}^{-1}q$, $T_{PD}^{-1}\hat{q}$,
$\pi^{-1}\textbf{z}$, $B_{1}^{-1}p$, $B_{1}^{-1}r$,
$\pi^{-1}\textbf{t})$.

\end{thm}

\section{An abstract approximation. Construction of the isomorphism for the approximate problem}
\label{sectiunea_ecuatii_Stokes_problema_aproximativa}

\subsection{The approximate problem}
\label{sectiunea_ecuatii_Stokes_problema_aproximativa_general}

In the sequel, we write the approximate problem of the problem
(\ref{mathcal_P2_e1_6_ecStokes_ecPoisson_variational_dem_q_compl_p_dem11_cont_alt_izom})
- (\ref{e1_11_ecStokes_ecPoisson_variational_dem_grad_miu_ec_TAU})
in a general framework. We then formulate this approximate problem
as equation
(\ref{well_posed_e1_11_ecStokes_ecPoisson_variational_dem_q_compl4_ec_mathcal_X_h_E}).
Particular approximations can be obtained by finite element
method, finite differences method, finite volume method, spectral
methods, wavelets and so on. In Section
\ref{sectiunea_ecuatii_Stokes_problema_aproximativa_partea_2}, we
use finite element method.

We use the general approximation context formulated in
\cite{CLBichir_bib_Bochev_GunzburgerLSFEM2009,
CLBichir_bib_Gir_Rav1986, CLBichir_bib_Ka_Ak1986}. Let $X_{h}$ and
$Y_{h}$ be two closed subspaces of $H^{1}(\Omega)$. Following
\cite{CLBichir_bib_Ciarlet2002, CLBichir_bib_Ciarlet_Raviart1974,
CLBichir_bib_Gir_Rav1986}, we approximate $H^{1}(\Omega)$ and
$L^{2}(\Omega)$ with the same space.

In order to approximate
(\ref{mathcal_P2_e1_6_ecStokes_ecPoisson_variational_dem_q_compl_p_dem11_cont_alt_izom})
-
(\ref{e1_11_ecStokes_ecPoisson_variational_dem_grad_miu_ec_TAU}),
we introduce the spaces

$X_{0h}$ $=$ $X_{h} \cap H_{0}^{1}(\Omega)$ $=$ $\{ x_{h} \in
X_{h}; \ x_{h} = 0 \ \textrm{on} \ \partial \Omega \}$,

$M_{h}$ $=$ $Y_{h} \cap L_{0}^{2}(\Omega)$ $=$ $\{ y_{h} \in
Y_{h}; \, \int_{\Omega} y_{h} dx = 0 \}$,

$Y_{0h}$ $=$ $Y_{h} \cap H_{0}^{1}(\Omega)$ $=$ $\{ y_{h} \in
Y_{h}; \ y_{h} = 0 \ \textrm{on} \ \partial \Omega \}$,

$\textbf{X}_{0h}$ $=$ $\textbf{X}_{h} \cap
\textbf{H}_{0}^{1}(\Omega)$.

Let us consider the approximate equation: find $v_{h}$ $\in$
$E_{h}$ such that
\begin{eqnarray}
         a(v_{h},\lambda_{h})=\langle \psi,\lambda_{h} \rangle , \
              \forall \, \lambda_{h}  \, \in \, E_{h} \, ,
         \label{e1_ecStokes_ecPoisson_variational_scriere_COMPACTA_h_generic}
\end{eqnarray}
corresponding to
(\ref{e1_ecStokes_ecPoisson_variational_scriere_COMPACTA_exact_generic})
\cite{CLBichir_bib_Atkinson_Han2009,
CLBichir_bib_Bochev_GunzburgerLSFEM2009,
CLBichir_bib_Brenner_Scott2008, CLBichir_bib_Ciarlet2002,
CLBichir_bib_Dautray_Lions_vol6_1988, CLBichir_bib_A_Ern2005,
CLBichir_bib_Ili1980, CLBichir_bib_Quarteroni_Valli2008,
CLBichir_bib_Temam1979, CLBichir_bib_E_Zeidler_NFA_IIB}. $E_{h}$
is a closed subspace of $E$. The hypotheses of C\'ea's lemma are
satisfied for this problem. So
(\ref{e1_ecStokes_ecPoisson_variational_scriere_COMPACTA_h_generic})
has an unique solution $v_{h}$ $\in$ $E_{h}$ and, for the
solutions $v$, $v_{h}$ of
(\ref{e1_ecStokes_ecPoisson_variational_scriere_COMPACTA_exact_generic}),
(\ref{e1_ecStokes_ecPoisson_variational_scriere_COMPACTA_h_generic}),
there exists $\gamma_{E,a} > 0$ (that depends on the continuity
and ellipticity of $a(\cdot,\cdot)$ on $E$) such that
\begin{equation}
\label{e5_45p_conditia_q_dense_general_Cea}
   \| v-v_{h} \|_{E} \leq \gamma_{E,a}\inf_{e_{h} \in E_{h}}\| v-e_{h} \|_{E} \, ,
\end{equation}
Define the operator $T_{h} : E' \rightarrow E_{h}$ by: given
$\psi$ $\in$ $E'$, $v_{h}=T_{h}\psi$ if and only if $v_{h}$ $\in$
$E_{h}$ is the unique solution of problem
(\ref{e1_ecStokes_ecPoisson_variational_scriere_COMPACTA_h_generic}).

Let us associate the set $\mathcal{M}_{h,i}$ $=$ $\{ i$,
$\mathcal{M}_{i}$,
$(\ref{e1_ecStokes_ecPoisson_variational_scriere_COMPACTA_exact_generic})$,
$(\ref{e1_ecStokes_ecPoisson_variational_scriere_COMPACTA_h_generic})$,
$E_{h}$, $v_{h}$, $\gamma_{E,a}$, $T_{h} \}$ to the set
$\mathcal{M}_{i}$. Let us retain the following particular cases
$\mathcal{M}_{h,1}$, ..., $\mathcal{M}_{h,5}$ of
$\mathcal{M}_{h,i}$:

$\mathcal{M}_{h,1}$ $=$ $\{ 1$, $\mathcal{M}_{1}$,
$(\ref{e1_ecStokes_ecPoisson_variational_a_BIS})$,
$(\ref{e1_ecStokes_ecPoisson_variational_a_BIS_h})$,
$\textbf{X}_{0h}$, $\textbf{u}_{h}$, $\gamma_{1}$, $T_{VPD,h} \}$
with
\begin{equation}
\label{e1_ecStokes_ecPoisson_variational_a_BIS_h}
   (grad \, \textbf{u}_{h},grad \, \textbf{w}_{h})
      =\langle \textbf{f},\textbf{w}_{h} \rangle, \
              \forall \, w_{h}  \, \in \, \textbf{X}_{0h} \, .
\end{equation}

$\mathcal{M}_{h,2}$ $=$ $\{ 2$, $\mathcal{M}_{2}$,
$(\ref{e1_ecStokes_ecPoisson_variational_a_BIS_scalar})$,
$(\ref{e1_ecStokes_ecPoisson_variational_a_BIS_scalar_h})$,
$Y_{0h}$, $q_{h}$, $\gamma_{2}$, $T_{PD,h} \}$ with
\begin{equation}
\label{e1_ecStokes_ecPoisson_variational_a_BIS_scalar_h}
   (grad \, q_{h},grad \, \lambda_{h})=\langle \psi,\lambda_{h} \rangle, \
              \forall \, \lambda_{h}  \, \in \, Y_{0h} \, .
\end{equation}

$\mathcal{M}_{h,3}$ $=$ $\{ 3$, $\mathcal{M}_{3}$,
$(\ref{e1_ecStokes_ecHelmholtz_variational_d_dual})$,
$(\ref{e1_ecStokes_ecHelmholtz_variational_d_h})$, $Y_{h}$,
$\phi_{h}$, $\gamma_{3}$, $B_{1,h} \}$ with
\begin{equation}
\label{e1_ecStokes_ecHelmholtz_variational_d_h}
   (grad \, \phi_{h},grad \, \mu_{h}) + (\phi_{h},\mu_{h})
      =\langle \psi_{1},\mu_{h} \rangle, \
              \forall \, \mu_{h}  \, \in \, Y_{h} \, .
\end{equation}

$\mathcal{M}_{h,4}$ $=$ $\{ 4$, $\mathcal{M}_{4}$,
$(\ref{e1_ecStokes_ecHelmholtz_variational_d_dual})$,
$(\ref{e1_ecStokes_ecHelmholtz_variational_d_h})$, $M_{h}$,
$\phi_{h}$, $\gamma_{4}$, $B_{1,h} \}$.

Let $\pi_{h}$ be the $L^{2}$ projection operator onto
$\textbf{X}_{h}$ defined  by $\pi_{h}$ $\in$
$L(\textbf{L}^{2}(\Omega),\textbf{X}_{h})$ such that
$\textbf{z}_{h}$ $\in$ $\textbf{X}_{h}$, $\textbf{f}_{\ast}$ $\in$
$\textbf{L}^{2}(\Omega)$,
$\pi_{h}\textbf{f}_{\ast}=\textbf{z}_{h}$ if and only if
\begin{equation}
\label{SOL_prel_op_hat_pi_h_DEF1}
   (\textbf{z}_{h}, \textbf{w}_{h})_{0}
      =(\textbf{f}_{\ast}, \textbf{w}_{h})_{0}, \
              \forall \, \textbf{w}_{h}  \, \in \, \textbf{X}_{h} \, ,
\end{equation}
$\mathcal{M}_{h,5}$ $=$ $\{ 5$, $\mathcal{M}_{5}$,
$(\ref{prel_op_hat_pi_h_DEF2H_1_EC_M9})$,
$(\ref{SOL_prel_op_hat_pi_h_DEF1})$, $\textbf{X}_{h}$,
$\textbf{z}_{h}$, $\gamma_{5}$, $\pi_{h} \}$.

Approximate
(\ref{mathcal_P2_e1_6_ecStokes_ecPoisson_variational_dem_q_compl_p_dem11_cont_alt_izom})
-
(\ref{e1_11_ecStokes_ecPoisson_variational_dem_grad_miu_ec_TAU}),
where
(\ref{mathcal_P2_e1_6_ecStokes_ecPoisson_variational_dem_q_compl_dem11_cont_alt_izom_instead2_delta_hat_delta_aprox})
is considered with the formulation
(\ref{mathcal_P2_e1_6_ecStokes_ecPoisson_variational_dem_q_compl_dem11_cont_alt_izom_instead2_delta_hat_delta_aprox_ECH_var_modif_cont}),
by the approximate extended system: find $(x_{h},U_{h})$ $\in$
$\widehat{\Gamma}_{h}$ such that
\begin{eqnarray}
   && p_{h} - (q_{h}+\hat{p}_{h}) = \hat{g} \, ,
         \label{mathcal_P2_e1_6_ecStokes_ecPoisson_variational_dem_q_compl_p_dem11_cont_alt_izom_h} \\
   && p_{S,h} - p_{h} = \ell_{S} \, ,
         \label{e5_1_cond_ecStokes_ecPoisson_variational_dem_div_corelatie_aprox} \\
   && (div \, \textbf{u}_{h}-y_{h},\hat{\mu}) = (g,\hat{\mu}), \
              \forall \, \hat{\mu}  \, \in \, L_{0}^{2}(\Omega) \, ,
         \label{e5_1_cond_ecStokes_ecPoisson_variational_dem_div_h} \\
   && (grad \, \textbf{u}_{h},grad \, \textbf{w}_{h})-(p_{S,h},div \, \textbf{w}_{h})=\langle \textbf{f},\textbf{w}_{h} \rangle, \
              \forall \, \textbf{w}_{h}  \, \in \, \textbf{X}_{0h} \, ,
         \label{e1_11_ecStokes_ecPoisson_variational_dem_h} \\
   && - \alpha \langle \widetilde{S}_{\textbf{z}_{h}},\bar{\lambda}_{h} \rangle
            + \langle \widetilde{S}_{\textbf{t}_{h}},\bar{\lambda}_{h} \rangle
            + (grad \, q_{h},grad \, \bar{\lambda}_{h})
            =\langle \psi,\bar{\lambda}_{h} \rangle, \
              \forall \, \bar{\lambda}_{h}  \, \in \, Y_{0h} \, ,
         \label{mathcal_P2_e1_6_ecStokes_ecPoisson_variational_dem_q_compl_dem11_cont_alt_izom_h} \\
   && (grad \, \hat{q}_{h},grad \, \tilde{\lambda}_{h})
            +(grad \, p_{h},grad \, \tilde{\lambda}_{h})
         \label{mathcal_P2_e1_6_ecStokes_ecPoisson_variational_dem_q_compl_hat_q_dem11_cont_alt_izom_h} \\
   && \qquad  = \langle \hat{\psi},\tilde{\lambda}_{h} \rangle, \
              \forall \, \tilde{\lambda}_{h}  \, \in \, Y_{0h} \, ,
         \nonumber \\
   && (\textbf{z}_{h},\tilde{\textbf{w}}_{h}) + (grad \, p_{h},\tilde{\textbf{w}}_{h})=(\textbf{f}_{\triangle},\tilde{\textbf{w}}_{h}), \
              \forall \, \tilde{\textbf{w}}_{h}  \, \in \, \textbf{X}_{h} \, ,
         \label{mathcal_P2_e1_6_ecStokes_ecPoisson_variational_dem_q_compl_dem11_cont_alt_izom_instead2_delta_h_hat_delta_aprox} \\
   && (\textbf{z}_{h} - \textbf{t}_{h} + grad \, (\hat{p}_{h}-\hat{q}_{h}),grad \, \varphi_{h})
         \label{e1_11_ecStokes_ecPoisson_variational_dem_grad_miu_h} \\
   && \qquad  +(grad \, r_{h},grad \, \varphi_{h})
              = \langle \varrho,\varphi_{h} \rangle, \
              \forall \, \varphi_{h}  \, \in \, Y_{h} \, ,
         \nonumber \\
   && (r_{h},\tilde{\varphi}_{h})
            +(grad \, (q_{h}+\hat{q}_{h}),grad \, \tilde{\varphi}_{h})
            = \langle \hat{\varrho},\tilde{\varphi}_{h} \rangle, \
              \forall \, \tilde{\varphi}_{h}  \, \in \, Y_{h} \, ,
         \label{e1_11_ecStokes_ecPoisson_variational_dem_grad_miu_r_h} \\
   && (\textbf{t}_{h},\hat{\textbf{w}}_{h})=(\hat{\textbf{f}},\hat{\textbf{w}}_{h}), \
              \forall \, \hat{\textbf{w}}_{h}  \, \in \, \textbf{X}_{h} \, ,
         \label{e1_11_ecStokes_ecPoisson_variational_dem_grad_miu_ec_TAU_h}
\end{eqnarray}

\begin{rem}
\label{observatia5_echiv_dense} Equation
(\ref{mathcal_P2_e1_6_ecStokes_ecPoisson_variational_dem_q_compl_dem11_cont_alt_izom_instead2_delta_h_hat_delta_aprox})
approximates equation
(\ref{mathcal_P2_e1_6_ecStokes_ecPoisson_variational_dem_q_compl_dem11_cont_alt_izom_instead2_delta_hat_delta_aprox_ECH_var_modif_cont}).
\end{rem}

Let $\textbf{f}_{\triangle,h}$ and $\hat{\textbf{f}}_{h}$ $\in$
$\textbf{X}_{h}$ be defined by
\begin{eqnarray}
   && (\textbf{f}_{\triangle,h},\tilde{\textbf{w}}_{h})=(\textbf{f}_{\triangle},\tilde{\textbf{w}}_{h}), \
              \forall \, \tilde{\textbf{w}}_{h}  \, \in \, \textbf{X}_{h} \, , \
      (\hat{\textbf{f}}_{h},\tilde{\textbf{w}}_{h})=(\hat{\textbf{f}},\hat{\textbf{w}}_{h}), \
              \forall \, \hat{\textbf{w}}_{h}  \, \in \, \textbf{X}_{h} \, .
         \label{def_chi_h}
\end{eqnarray}
Using $\textbf{f}_{\triangle,h}$, we work with
(\ref{mathcal_P2_e1_6_ecStokes_ecPoisson_variational_dem_q_compl_dem11_cont_alt_izom_instead2_delta_h_hat_delta_aprox})
as we work with
(\ref{VAR_mathcal_P2_e1_6_ecStokes_ecPoisson_variational_dem_q_compl_dem11_cont_alt_izom_instead2_delta_hat_delta_aprox_D_VAR_ec}).

Using (\ref{def_chi_h}), equations
(\ref{mathcal_P2_e1_6_ecStokes_ecPoisson_variational_dem_q_compl_dem11_cont_alt_izom_instead2_delta_h_hat_delta_aprox})
and
(\ref{e1_11_ecStokes_ecPoisson_variational_dem_grad_miu_ec_TAU_h})
are equivalent, respectively, to
\begin{eqnarray}
   && \textbf{z}_{h} + grad \, p_{h} = \textbf{f}_{\triangle,h} \, , \
      \textbf{t}_{h} = \hat{\textbf{f}}_{h}
         \label{VAR_mathcal_P2_e1_6_ecStokes_ecPoisson_variational_dem_q_compl_dem11_cont_alt_izom_instead2_delta_hat_delta_aprox_D_VAR_ec_h}
\end{eqnarray}

Replacing
(\ref{mathcal_P2_e1_6_ecStokes_ecPoisson_variational_dem_q_compl_p_dem11_cont_alt_izom_h})
and
(\ref{VAR_mathcal_P2_e1_6_ecStokes_ecPoisson_variational_dem_q_compl_dem11_cont_alt_izom_instead2_delta_hat_delta_aprox_D_VAR_ec_h})
in (\ref{e1_11_ecStokes_ecPoisson_variational_dem_grad_miu_h}), it
results
\begin{eqnarray}
   && (\textbf{f}_{\triangle,h} - \hat{\textbf{f}}_{h} - grad \, \hat{g},grad \, \varphi_{h})
      - (grad \, (q_{h}+\hat{q}_{h}),grad \, \varphi_{h})
         \label{e1_11_ecStokes_ecPoisson_variational_dem_grad_miu_inlocuiri_h} \\
   && \qquad + (grad \, r_{h},grad \, \varphi_{h})
            = \langle \varrho,\varphi_{h} \rangle, \
              \forall \, \varphi_{h}  \, \in \, Y_{h} \, ,
         \nonumber
\end{eqnarray}

Let $\tilde{\psi}_{r,h}$ $\in$ $Y_{h}'$ be defined by
\begin{eqnarray}
   && \langle \tilde{\psi}_{r,h},\varphi_{h} \rangle
      = (-\textbf{f}_{\triangle,h} + \hat{\textbf{f}}_{h} + grad \, \hat{g},grad \, \varphi_{h})
         \label{e1_11_ecStokes_ecPoisson_variational_dem_grad_miu_inlocuiri_def_h} \\
   && \qquad + \langle \varrho,\varphi_{h} \rangle
            + \langle \hat{\varrho},\varphi_{h} \rangle, \
              \forall \, \varphi_{h}  \, \in \, Y_{h} \, ,
         \nonumber
\end{eqnarray}
Summing
(\ref{e1_11_ecStokes_ecPoisson_variational_dem_grad_miu_inlocuiri_h})
and (\ref{e1_11_ecStokes_ecPoisson_variational_dem_grad_miu_r_h}),
we obtain that $r_{h}$ is the solution of the equation
\begin{eqnarray}
   && ((r_{h},\varphi_{h}))_{1}
      = \langle \tilde{\psi}_{r,h},\varphi_{h} \rangle, \
              \forall \, \varphi_{h}  \, \in \, Y_{h} \, ,
         \label{e1_11_ecStokes_ecPoisson_variational_dem_grad_miu_inlocuiri_ec_h}
\end{eqnarray}

From
(\ref{VAR_mathcal_P2_e1_6_ecStokes_ecPoisson_variational_dem_q_compl_dem11_cont_alt_izom_instead2_delta_hat_delta_aprox_D_VAR_ec_h}),
we deduce
\begin{eqnarray}
   && (\textbf{z}_{h} - \textbf{t}_{h} + grad \, p_{h},grad \, \varphi_{h})
            = (\textbf{f}_{\triangle,h} - \hat{\textbf{f}}_{h},grad \, \varphi_{h}), \
              \forall \, \varphi_{h}  \, \in \, Y_{h} \, ,
         \label{e1_11_ecStokes_ecPoisson_variational_dem_grad_miu_inlocuiri_sum_2_L2_h}
\end{eqnarray}
Equation
(\ref{e1_11_ecStokes_ecPoisson_variational_dem_grad_miu_inlocuiri_sum_2_L2_h})
approximates
(\ref{e1_11_ecStokes_ecPoisson_variational_dem_grad_miu_inlocuiri_sum_2}).
The same is obtained if we replace
(\ref{mathcal_P2_e1_6_ecStokes_ecPoisson_variational_dem_q_compl_p_dem11_cont_alt_izom_h}),
(\ref{e1_11_ecStokes_ecPoisson_variational_dem_grad_miu_ec_TAU_h})
in (\ref{e1_11_ecStokes_ecPoisson_variational_dem_grad_miu_h}) and
then sum
(\ref{e1_11_ecStokes_ecPoisson_variational_dem_grad_miu_r_h}) to
(\ref{e1_11_ecStokes_ecPoisson_variational_dem_grad_miu_h}), where
we use
(\ref{e1_11_ecStokes_ecPoisson_variational_dem_grad_miu_inlocuiri_ec_h}).

\begin{cor}
\label{observatia_ec_Poisson_grad_p_h} We have
\begin{eqnarray}
   && ( \alpha \textbf{z}_{h} - \textbf{t}_{h}, grad \, \bar{\lambda}_{h} )
            + (grad \, p_{h},grad \, \bar{\lambda}_{h})
            =(r_{h},\bar{\lambda}_{h})
          \label{VAR_mathcal_P2_e1_6_ecStokes_ecPoisson_variational_dem_q_compl_dem11_cont_alt_izom_PRESIUNE_dir_h} \\
   && +\langle \psi,\bar{\lambda}_{h} \rangle
            +\langle \hat{\psi},\bar{\lambda}_{h} \rangle
            -\langle \hat{\varrho},\bar{\lambda}_{h} \rangle , \
              \forall \, \bar{\lambda}_{h}  \, \in \, Y_{0h} \, ,
         \nonumber
\end{eqnarray}
where we use
(\ref{e1_11_ecStokes_ecPoisson_variational_dem_grad_miu_inlocuiri_ec_h}).
This equation approximates
(\ref{VAR_mathcal_P2_e1_6_ecStokes_ecPoisson_variational_dem_q_compl_dem11_cont_alt_izom_PRESIUNE_dir}).

\begin{eqnarray}
   && (\textbf{z}_{h} - \textbf{t}_{h} + grad \, p_{h},grad \, \varphi_{h})
         \label{e1_11_ecStokes_ecPoisson_variational_dem_grad_miu_inlocuiri_sum_2_L2_cor_h} \\
   && \qquad = (\textbf{f}_{\triangle,h} - \hat{\textbf{f}}_{h},grad \, \varphi_{h}), \
              \forall \, \varphi_{h}  \, \in \, Y_{h} \ominus Y_{0h} \, ,
         \nonumber
\end{eqnarray}
This equation approximates
(\ref{e1_11_ecStokes_ecPoisson_variational_dem_grad_miu_inlocuiri_sum_2_L2_cor}).
\end{cor}

\begin{proof}

If we add
(\ref{mathcal_P2_e1_6_ecStokes_ecPoisson_variational_dem_q_compl_dem11_cont_alt_izom_h})
and
(\ref{mathcal_P2_e1_6_ecStokes_ecPoisson_variational_dem_q_compl_hat_q_dem11_cont_alt_izom_h})
and subtract
(\ref{e1_11_ecStokes_ecPoisson_variational_dem_grad_miu_r_h}),
where the test functions $\varphi_{h}$ $\in$ $Y_{0h}$, we have
(\ref{VAR_mathcal_P2_e1_6_ecStokes_ecPoisson_variational_dem_q_compl_dem11_cont_alt_izom_PRESIUNE_dir_h}),
where we use
(\ref{e1_11_ecStokes_ecPoisson_variational_dem_grad_miu_inlocuiri_ec_h}).

(\ref{e1_11_ecStokes_ecPoisson_variational_dem_grad_miu_inlocuiri_sum_2_L2_h})
gives
(\ref{e1_11_ecStokes_ecPoisson_variational_dem_grad_miu_inlocuiri_sum_2_L2_cor_h}).

\qquad
\end{proof}

\begin{rem}
\label{observatia5_consistentFEM_compatibilitate}

As in the exact case, in
(\ref{e1_11_ecStokes_ecPoisson_variational_dem_diferenta}), we
obtain
\begin{eqnarray}
   && (grad \, \textbf{u}_{h},grad \, \textbf{w}_{h})
      - (\textbf{z}_{h},\textbf{w}_{h})
      = (\ell_{S},div \, \textbf{w}_{h})
         \label{e1_11_ecStokes_ecPoisson_variational_dem_diferenta_h} \\
   && \qquad +\langle \textbf{f},\textbf{w}_{h} \rangle-(\textbf{f}_{\triangle},\textbf{w}_{h}), \
              \forall \, \textbf{w}_{h}  \, \in \, \textbf{X}_{0h} \, ,
         \nonumber
\end{eqnarray}
\end{rem}

Approximate
(\ref{mathcal_P2_e1_6_ecStokes_ecPoisson_variational_dem_q_compl_p_dem11_cont_alt_izom_sist})
-
(\ref{e1_11_ecStokes_ecPoisson_variational_dem_grad_miu_ec_TAU_sist})
by the problem: find $(x_{h},U_{h})$ $\in$ $\widehat{\Gamma}_{h}$
such that
\begin{eqnarray}
   && p_{h} - (q_{h}+\hat{p}_{h}) = \hat{g} \, ,
         \label{mathcal_P2_e1_6_ecStokes_ecPoisson_variational_dem_q_compl_p_dem11_cont_alt_izom_h_sist} \\
   && p_{S,h} - p_{h} = \ell_{S} \, ,
         \label{e5_1_cond_ecStokes_ecPoisson_variational_dem_div_corelatie_aprox_sist} \\
   && div \, \textbf{u}_{h}-y_{h} = g \, ,
         \label{e5_1_cond_ecStokes_ecPoisson_variational_dem_div_h_sist} \\
   && \textbf{u}_{h} - T_{VPD,h}G_{\textbf{u}}(p_{S,h})
         = T_{VPD,h}\textbf{f} \, ,
         \label{e1_11_ecStokes_ecPoisson_variational_dem_q_compl4_h} \\
   && q_{h} - T_{PD,h}G_{q}(\textbf{t}_{h},\textbf{z}_{h})
         = T_{PD,h}\psi \, ,
         \label{e1_6_ecStokes_ecPoisson_variational_dem_q_compl4_h} \\
   && \hat{q}_{h} - T_{PD,h}G_{\hat{q}}(p_{h})
         = T_{PD,h}\hat{\psi} \, ,
         \label{e1_6_ecStokes_ecPoisson_variational_dem_q_compl_hat_q4_h} \\
   && \textbf{z}_{h} - \pi_{h} G_{\textbf{z}}(p_{h})
         = \pi_{h} \textbf{f}_{\triangle} \, ,
         \label{dem10_cont_alt_izom_instead2_delta_h_cont} \\
   && p_{h} - B_{1,h}G_{p}(p_{h},\hat{p}_{h},\hat{q}_{h},\textbf{z}_{h},\textbf{t}_{h},r_{h})
         = B_{1,h}\varrho \, ,
         \label{e1_11_ecStokes_ecPoisson_variational_dem_grad_miu_B0_h} \\
   && r_{h} - B_{1,h}G_{r}(q_{h},\hat{q}_{h},r_{h})
         = B_{1,h}\hat{\varrho} \, ,
         \label{e1_11_ecStokes_ecPoisson_variational_dem_grad_miu_inlocuiri_ec_op_sist_h} \\
   && \textbf{t}_{h} - \pi_{h} G_{\textbf{t}}(\textbf{u}_{h})
            =\pi_{h} \hat{\textbf{f}} \, ,
          \label{e1_11_ecStokes_ecPoisson_variational_dem_grad_miu_ec_TAU_h_sist}
\end{eqnarray}

\begin{lem}
\label{lema__izom_1_hhh} Problem
(\ref{mathcal_P2_e1_6_ecStokes_ecPoisson_variational_dem_q_compl_p_dem11_cont_alt_izom_h_sist})
-
(\ref{e1_11_ecStokes_ecPoisson_variational_dem_grad_miu_ec_TAU_h_sist})
is equivalent to problem
(\ref{mathcal_P2_e1_6_ecStokes_ecPoisson_variational_dem_q_compl_p_dem11_cont_alt_izom_h})
-
(\ref{e1_11_ecStokes_ecPoisson_variational_dem_grad_miu_ec_TAU_h}).
\end{lem}

\begin{proof}
The conclusion follows from the definition of operators
$T_{VPD,h}$, $T_{PD,h}$, $\pi_{h}$, $B_{1,h}$.

\qquad
\end{proof}

\begin{lem}
\label{lema__izom_1_hhh_th} Problem
(\ref{mathcal_P2_e1_6_ecStokes_ecPoisson_variational_dem_q_compl_p_dem11_cont_alt_izom_h_sist})
-
(\ref{e1_11_ecStokes_ecPoisson_variational_dem_grad_miu_ec_TAU_h_sist})
has the form of equation
(\ref{well_posed_e1_11_ecStokes_ecPoisson_variational_dem_q_compl4_ec_mathcal_X_h_E}),
so Theorem \ref{well_posed_teorema5_1_Stokes_math_B_L2}, Corollary
\ref{well_posed_corolarul_teorema5_1_Stokes_math_B_L2_Tdense},
Corollary \ref{well_posed_corolarul_teorema5_1_restrictie_dense}
can be used in order to analyze
(\ref{mathcal_P2_e1_6_ecStokes_ecPoisson_variational_dem_q_compl_p_dem11_cont_alt_izom_h_sist})
-
(\ref{e1_11_ecStokes_ecPoisson_variational_dem_grad_miu_ec_TAU_h_sist}).
\end{lem}

\begin{proof}

Consider the spaces $\mathcal{Q}$, $\mathcal{X}$, $\mathcal{Y}$,
$\mathcal{Z}$ and $\mathcal{X}_{h}$ defined in Section
\ref{sectiunea_ecuatii_remarks}. The operator
$\widehat{\mathcal{A}}_{h}$ defined by the left hand side of
(\ref{mathcal_P2_e1_6_ecStokes_ecPoisson_variational_dem_q_compl_p_dem11_cont_alt_izom_h_sist})
-
(\ref{e1_11_ecStokes_ecPoisson_variational_dem_grad_miu_ec_TAU_h_sist})
has the form of the operator $\widehat{\mathcal{A}}_{h}$ defined
in Section
\ref{well_posed_sectiunea_general_theorem_on_the_well_posedness}.

\qquad
\end{proof}

\begin{thm}
\label{teorema_izom_1_hhh_alfa} Assume the hypotheses of Theorem
\ref{well_posed_teorema5_1_Stokes_math_B_L2}. Taking $\alpha$
$\in$ $\mathbb{R}$, $\alpha \neq 1$, we obtain a family of well -
posed approximate formulations
(\ref{mathcal_P2_e1_6_ecStokes_ecPoisson_variational_dem_q_compl_p_dem11_cont_alt_izom_h})
-
(\ref{e1_11_ecStokes_ecPoisson_variational_dem_grad_miu_ec_TAU_h}).
\end{thm}

\begin{proof} The results follow from Lemma \ref{lema__izom_1_hhh_th}, Theorem
\ref{well_posed_teorema5_1_Stokes_math_B_L2} and Lemma
\ref{lema__izom_1_hhh}.

\qquad
\end{proof}

\begin{thm}
\label{teorema_izom_1_hhh} Assume the hypotheses of Corollaries
\ref{corolarul5_teorema_izom_1_cor10_div0} and
\ref{corolarul5_teorema_izom_1_cor10}. Assume the hypotheses of
Theorem \ref{well_posed_teorema5_1_Stokes_math_B_L2}. For each
$\alpha \neq 1$, the solution $(\textbf{u},p)$ of problem
(\ref{e1_11_ecStokes_ecPoisson_variational_dem_th_I_5_1}) -
(\ref{e5_1_cond_ecStokes_ecPoisson_variational_dem_th_I_5_1}) is
approximated by the components $(\textbf{u}_{h},p_{h})$ of the
solution of
(\ref{mathcal_P2_e1_6_ecStokes_ecPoisson_variational_dem_q_compl_p_dem11_cont_alt_izom_h})
-
(\ref{e1_11_ecStokes_ecPoisson_variational_dem_grad_miu_ec_TAU_h}).
\end{thm}

\begin{proof}
This follows from Theorem \ref{teorema_izom_1_hhh_alfa}.

\qquad
\end{proof}

\subsection{Approximation by the finite element method}
\label{sectiunea_ecuatii_Stokes_problema_aproximativa_fem}

In order to define the following operator $\pi_{0h}$ and the
discrete Laplace operator $-\triangle^{h}$, assume that the
approximate spaces are finite element spaces and the following
inverse inequality \cite{CLBichir_bib_Bochev_Gunzburger2004,
CLBichir_bib_Bochev_GunzburgerLSFEM2009, CLBichir_bib_Ciarlet2002,
CLBichir_bib_Gir_Rav1986} is satisfied
\begin{eqnarray}
   && \| \textbf{u}_{h} \|_{1} \leq C_{FEM} h^{-1} \| \textbf{u}_{h} \|_{0}, \
              \forall \, \textbf{u}_{h}  \, \in \, \textbf{X}_{h} \, ,
         \label{cont_SOL_prel_op_hat_pi_h_DEF3_ineg}
\end{eqnarray}

We refer to \cite{CLBichir_bib_Bochev_Gunzburger2004}. Assume
(\ref{cont_SOL_prel_op_hat_pi_h_DEF3_ineg}). Let $\pi_{0h}$ be the
$L^{2}$ projection operator onto $\textbf{X}_{0h}$ defined by
$\pi_{0h}:\textbf{L}^{2}(\Omega) \rightarrow \textbf{X}_{0h}$ such
that $\textbf{z}_{h}$ $\in$ $\textbf{X}_{0h}$, $\textbf{f}_{\ast}$
$\in$ $\textbf{L}^{2}(\Omega)$,
$\pi_{0h}\textbf{f}_{\ast}=\textbf{z}_{h}$ if and only if
\begin{eqnarray}
   && (\textbf{z}_{h}, \textbf{w}_{h})=(\textbf{f}_{\ast}, \textbf{w}_{h}), \
              \forall \, \textbf{w}_{h}  \, \in \, \textbf{X}_{0h} \, ,
         \label{SOL_prel_op_hat_pi_h_DEF1_zero}
\end{eqnarray}
$\mathcal{M}_{h,6}$ $=$ $\{ 6$, $\mathcal{M}_{5}$,
$(\ref{prel_op_hat_pi_h_DEF2H_1_EC_M9})$,
$(\ref{SOL_prel_op_hat_pi_h_DEF1_zero})$, $\textbf{X}_{0h}$,
$\textbf{z}_{h}$, $\gamma_{6}$, $\pi_{0h} \}$.

We refer to \cite{CLBichir_bib_Bochev_Gunzburger2004}. Assume
(\ref{cont_SOL_prel_op_hat_pi_h_DEF3_ineg}). The discrete Laplace
operator $-\triangle^{h}$ is defined to be the mapping
$-\triangle^{h}:\textbf{H}_{0}^{1}(\Omega) \rightarrow
\textbf{X}_{0h}$ be such that $-\triangle^{h}\textbf{u}$ $=$
$\xi_{h}$ if and only if
\begin{eqnarray}
   && (\xi_{h},\textbf{w}_{h})=(grad \, \textbf{u},grad \, \textbf{w}_{h}), \
              \forall \, \textbf{w}_{h}  \, \in \, \textbf{X}_{0h} \, ,
         \label{e1_11_ecStokes_ecPoisson_variational_dem_h_def_delta}
\end{eqnarray}

\begin{rem}
\label{observatia5_consistentFEM_frontiera} Assume that $\Omega$
is a polygon in $\mathbb{R}^{2}$. $N=2$. Let us consider a regular
triangulation $\Theta_{h}$ of $\Omega$ by triangles $K$ that are
triangles with three nodal points that coincide with the three
vertices.
\begin{equation}
\label{ec12_V}
   X_{h}=\{ v_{h} \in C^{0}(\bar{\Omega}),
      v_{h}|_{K} \in P_{1}(K),
      \forall K \in \Theta_{h} \} \, .
\end{equation}
We take $X_{h}$ $=$ $Y_{h}$. A basis functions of the space
$P_{1}(K)$ is defined by the three nodal points of $K$ (e.g.
subsection 10.2.1, page 394,
\cite{CLBichir_bib_Atkinson_Han2009}). The inverse inequality
(\ref{cont_SOL_prel_op_hat_pi_h_DEF3_ineg}) holds
(\cite{CLBichir_bib_Bochev_Gunzburger2004,
CLBichir_bib_Bochev_GunzburgerLSFEM2009, CLBichir_bib_Ciarlet2002,
CLBichir_bib_Gir_Rav1986}). Then, in equation
(\ref{e1_11_ecStokes_ecPoisson_variational_dem_grad_miu_inlocuiri_sum_2_L2_cor_h}),
we use \linebreak $Y_{h}$ $\ominus$ $Y_{0h}$ $=$ $\Xi_{h}$. The
space $\Xi_{h}$ denotes the complementary space $M_{h}$ from the
subsection 5.7.2.2, page 268, \cite{CLBichir_bib_Glowinski1984}.
From this reference, we also retain that $\mu_{h}$ $\in$ $\Xi_{h}$
is determined by its values at the boundary nodes.
\end{rem}

\begin{cor}
\label{observatia_ec_Poisson_grad_p_h_modalitate_zh_1} The
approximation $p_{h}$ of the pressure $p$ on the boundary depends
on the approximation $\textbf{z}_{h}$ of the $\textbf{z}$ on the
boundary. In order to obtain the approximate solution
$(\textbf{u}_{h},p_{h})$ as a component of the solution of
(\ref{e1_11_ecStokes_ecPoisson_variational_dem_h}),
(\ref{VAR_mathcal_P2_e1_6_ecStokes_ecPoisson_variational_dem_q_compl_dem11_cont_alt_izom_PRESIUNE_dir_h}),
(\ref{e1_11_ecStokes_ecPoisson_variational_dem_grad_miu_inlocuiri_sum_2_L2_cor_h}),
where we use
(\ref{e1_11_ecStokes_ecPoisson_variational_dem_grad_miu_ec_TAU_h}),
(\ref{e1_11_ecStokes_ecPoisson_variational_dem_grad_miu_inlocuiri_ec_h}),
$\textbf{z}_{h}$ can be determined from equation
(\ref{e1_11_ecStokes_ecPoisson_variational_dem_diferenta_h}) under
some additional conditions which are decided in every studied
particular case of approximation. One such condition is
$\textbf{z}_{h}$ $\in$ $\textbf{X}_{0h}$ in the case of the finite
element method, under the hypothesis
(\ref{cont_SOL_prel_op_hat_pi_h_DEF3_ineg}), that is,
$\textbf{z}_{h}$ $=$ $\xi_{h}$ from
(\ref{e1_11_ecStokes_ecPoisson_variational_dem_h_def_delta}),
where we use the approximation
$-\triangle^{h}:\textbf{H}_{0}^{1}(\Omega) \rightarrow
\textbf{X}_{0h}$ of the Laplace operator from
\cite{CLBichir_bib_Bochev_Gunzburger2004}.
\end{cor}

\begin{proof} The first affirmation follows from
(\ref{e1_11_ecStokes_ecPoisson_variational_dem_grad_miu_inlocuiri_sum_2_L2_cor_h}).
$\tilde{\psi}_{r,h}$ $=$ $0$ so $r_{h}$ $=$ $0$ from
(\ref{e1_11_ecStokes_ecPoisson_variational_dem_grad_miu_inlocuiri_ec_h}).
If we determine $u_{h}$, $p_{h}$, $\textbf{z}_{h}$,
$\textbf{t}_{h}$ as it is indicated in the Corollary
\ref{observatia_ec_Poisson_grad_p_h_modalitate_zh_1}, the rest of
the components of the solution $(x_{h},U_{h})$ are obtained from
the rest of the equations of
(\ref{mathcal_P2_e1_6_ecStokes_ecPoisson_variational_dem_q_compl_p_dem11_cont_alt_izom_h})
-
(\ref{e1_11_ecStokes_ecPoisson_variational_dem_grad_miu_ec_TAU_h}).
$u_{h}$, $p_{h}$, $\textbf{z}_{h}$, $\textbf{t}_{h}$ exist as
components of the solution $(x_{h},U_{h})$ of
(\ref{mathcal_P2_e1_6_ecStokes_ecPoisson_variational_dem_q_compl_p_dem11_cont_alt_izom_h})
-
(\ref{e1_11_ecStokes_ecPoisson_variational_dem_grad_miu_ec_TAU_h})
or
(\ref{mathcal_P2_e1_6_ecStokes_ecPoisson_variational_dem_q_compl_p_dem11_cont_alt_izom_h_sist})
-
(\ref{e1_11_ecStokes_ecPoisson_variational_dem_grad_miu_ec_TAU_h_sist}).
If we assume that they are not unique, then we contradict the
uniqueness of $(x_{h},U_{h})$.

\qquad
\end{proof}

\begin{cor}
\label{observatia_ec_Poisson_grad_p_h_modalitate_zh} Let us
continue the discussion from Corollary
\ref{observatia_ec_Poisson_grad_p_h_modalitate_zh_1} in the case
of the finite element method, under the hypothesis
(\ref{cont_SOL_prel_op_hat_pi_h_DEF3_ineg}).

We refer as to "the basic case" to the exact and approximated
problems studied till now,
(\ref{mathcal_P2_e1_6_ecStokes_ecPoisson_variational_dem_q_compl_p_dem11_cont_alt_izom})
- (\ref{e1_11_ecStokes_ecPoisson_variational_dem_grad_miu_ec_TAU})
(where we work with
(\ref{mathcal_P2_e1_6_ecStokes_ecPoisson_variational_dem_q_compl_dem11_cont_alt_izom_instead2_delta_hat_delta_aprox_ECH_var_modif_cont})
instead of
(\ref{mathcal_P2_e1_6_ecStokes_ecPoisson_variational_dem_q_compl_dem11_cont_alt_izom_instead2_delta_hat_delta_aprox})),
(\ref{mathcal_P2_e1_6_ecStokes_ecPoisson_variational_dem_q_compl_p_dem11_cont_alt_izom_sist})
-
(\ref{e1_11_ecStokes_ecPoisson_variational_dem_grad_miu_ec_TAU_sist}),
(\ref{mathcal_P2_e1_6_ecStokes_ecPoisson_variational_dem_q_compl_p_dem11_cont_alt_izom_h})
-
(\ref{e1_11_ecStokes_ecPoisson_variational_dem_grad_miu_ec_TAU_h}),
(\ref{mathcal_P2_e1_6_ecStokes_ecPoisson_variational_dem_q_compl_p_dem11_cont_alt_izom_h_sist})
-
(\ref{e1_11_ecStokes_ecPoisson_variational_dem_grad_miu_ec_TAU_h_sist}).

I. A modality to introduce the condition $\textbf{z}_{h}$ $\in$
$\textbf{X}_{0h}$ is obtained by taking two steps:

1. The basic case. For the study, we use Theorem
\ref{teorema5_1_partea2NOU} (or similar to this).

2. The approximate problem
(\ref{mathcal_P2_e1_6_ecStokes_ecPoisson_variational_dem_q_compl_p_dem11_cont_alt_izom_h})
-
(\ref{e1_11_ecStokes_ecPoisson_variational_dem_grad_miu_ec_TAU_h})
from the basic case is approximated by a problem of the same form,
with the same data, with a single modification: $\textbf{z}_{h}$
$\in$ $\textbf{X}_{h}$ from the basic case is approximated by the
new $\textbf{z}_{h}$ $\in$ $\textbf{X}_{0h}$; we replace
(\ref{mathcal_P2_e1_6_ecStokes_ecPoisson_variational_dem_q_compl_dem11_cont_alt_izom_instead2_delta_h_hat_delta_aprox}),
(\ref{dem10_cont_alt_izom_instead2_delta_h_cont}) by
(\ref{mathcal_P2_e1_6_ecStokes_ecPoisson_variational_dem_q_compl_dem11_cont_alt_izom_instead2_delta_h_hat_delta_aprox_caseIII}),
(\ref{dem10_cont_alt_izom_instead2_delta_h_cont_caseIII}),
respectively, where
\begin{eqnarray}
   && (\textbf{z}_{h},\tilde{\textbf{w}}_{h}) -(p_{h},div \, \tilde{\textbf{w}}_{h})=(\textbf{f}_{\triangle},\tilde{\textbf{w}}_{h}), \
              \forall \, \tilde{\textbf{w}}_{h}  \, \in \, \textbf{X}_{0h} \, ,
         \label{mathcal_P2_e1_6_ecStokes_ecPoisson_variational_dem_q_compl_dem11_cont_alt_izom_instead2_delta_h_hat_delta_aprox_caseIII}
\end{eqnarray}
\begin{eqnarray}
   && \textbf{z}_{h} - \pi_{0h} G_{\textbf{z}}(p_{h})
         = \pi_{0h} \textbf{f}_{\triangle} \, .
         \label{dem10_cont_alt_izom_instead2_delta_h_cont_caseIII}
\end{eqnarray}
For this second step of the study, for a fixed $h$, we use only
Theorem \ref{well_posed_teorema5_1_Stokes_math_B_L2}, without the
convergence $h \rightarrow 0$. In order to verify the estimate
(\ref{well_posed_e5_45p_conditia_q}), it is necessary only the use
of (\ref{e5_1_eroare_aprox_Poisson_HypH2_Td_2_STOKES_abc_3}) (or
similar to this) in order to compare
(\ref{dem10_cont_alt_izom_instead2_delta_h_cont}), from the basic
case, with (\ref{dem10_cont_alt_izom_instead2_delta_h_cont}) from
this second step.

Finally, for these two steps, the estimate
(\ref{well_posed_e5_45p_conditia_q}) must be satisfied comparing
(\ref{mathcal_P2_e1_6_ecStokes_ecPoisson_variational_dem_q_compl_p_dem11_cont_alt_izom_sist})
-
(\ref{e1_11_ecStokes_ecPoisson_variational_dem_grad_miu_ec_TAU_sist})
with
(\ref{mathcal_P2_e1_6_ecStokes_ecPoisson_variational_dem_q_compl_p_dem11_cont_alt_izom_h_sist})
-
(\ref{e1_11_ecStokes_ecPoisson_variational_dem_grad_miu_ec_TAU_h_sist}),
where
(\ref{mathcal_P2_e1_6_ecStokes_ecPoisson_variational_dem_q_compl_dem11_cont_alt_izom_instead2_delta_h_hat_delta_aprox}),
(\ref{dem10_cont_alt_izom_instead2_delta_h_cont}) are replaced by
(\ref{mathcal_P2_e1_6_ecStokes_ecPoisson_variational_dem_q_compl_dem11_cont_alt_izom_instead2_delta_h_hat_delta_aprox_caseIII}),
(\ref{dem10_cont_alt_izom_instead2_delta_h_cont_caseIII}) (and
where $\textbf{z}_{h}$ $\in$ $\textbf{X}_{0h}$). The convergence
$h \rightarrow 0$ from the first step and
(\ref{e5_1_eroare_aprox_Poisson_HypH2_Td_2_STOKES_abc_3}) (or
similar to this) in the second step lead us to the existence of an
index $h_{2}$ such that $\forall$ $h < h_{2}$, $q_{\ast} < 1$,
that is, (\ref{well_posed_e5_45p_conditia_q}).

II. Another modality to introduce the condition $\textbf{z}_{h}$
$\in$ $\textbf{X}_{0h}$ is obtained in the following manner. In
the basic case, we approximate
(\ref{mathcal_P2_e1_6_ecStokes_ecPoisson_variational_dem_q_compl_dem11_cont_alt_izom_instead2_delta_hat_delta_aprox})
by
(\ref{mathcal_P2_e1_6_ecStokes_ecPoisson_variational_dem_q_compl_dem11_cont_alt_izom_instead2_delta_h_hat_delta_aprox_caseIII})
that replaces
(\ref{mathcal_P2_e1_6_ecStokes_ecPoisson_variational_dem_q_compl_dem11_cont_alt_izom_instead2_delta_h_hat_delta_aprox})
in
(\ref{mathcal_P2_e1_6_ecStokes_ecPoisson_variational_dem_q_compl_p_dem11_cont_alt_izom_h})
-
(\ref{e1_11_ecStokes_ecPoisson_variational_dem_grad_miu_ec_TAU_h}),
and, in
(\ref{mathcal_P2_e1_6_ecStokes_ecPoisson_variational_dem_q_compl_p_dem11_cont_alt_izom_h_sist})
-
(\ref{e1_11_ecStokes_ecPoisson_variational_dem_grad_miu_ec_TAU_h_sist}),
we replace (\ref{dem10_cont_alt_izom_instead2_delta_h_cont}) by
(\ref{dem10_cont_alt_izom_instead2_delta_h_cont_caseIII}). For the
study, it can be used Theorem
\ref{well_posed_teorema5_1_Stokes_math_B_L2} and estimates such
those from Lemmas \ref{lema_eroare_aprox_toate_pr} and
\ref{lema_eroare_aprox_toate_pr_0} (or similar to these),
accompanied to laborious correlations between spaces, norms and
the components of $\mathcal{G}$.
\end{cor}

\begin{cor}
\label{observatia_ec_Poisson_grad_p_h_matrice_2} Assume the
hypotheses of Corollary
\ref{observatia_ec_Poisson_grad_p_h_modalitate_zh}. Then, the
problem has the following matrix form
\begin{equation}
\label{ecStokes_ecPoisson_variational_01_matr_3}
      \left[\begin{array}{ccc}
      A_{h} & 0 & B_{h}^{T} \\
      A_{h} & K_{h} & 0     \\
      0 & \left[\begin{array}{c}
      \alpha C_{h} \\
      C_{h,0}
      \end{array}\right] & D_{h}
      \end{array}\right]
      \left[\begin{array}{c}
      \bar{\textbf{u}}_{h} \\
      \bar{\textbf{z}}_{h} \\
      \bar{p}_{h}
      \end{array}\right]
      = \left[\begin{array}{c}
      \bar{\textbf{f}} \\
      \bar{0} \\
      \bar{\psi}
      \end{array}\right] \ ,
\end{equation}
with three block Poisson preconditioners $A_{h}$, $K_{h}$ and
$D_{h}$. $(\bar{\textbf{u}}_{h},\bar{\textbf{z}}_{h},\bar{p}_{h})$
is the coefficient vector corresponding to the finite element
solution $(\textbf{u}_{h},\textbf{z}_{h},p_{h})$.
\end{cor}

\begin{proof}
We use (\ref{e1_11_ecStokes_ecPoisson_variational_dem_h}),
(\ref{VAR_mathcal_P2_e1_6_ecStokes_ecPoisson_variational_dem_q_compl_dem11_cont_alt_izom_PRESIUNE_dir_h}),
(\ref{e1_11_ecStokes_ecPoisson_variational_dem_grad_miu_inlocuiri_sum_2_L2_cor_h}),
(\ref{e1_11_ecStokes_ecPoisson_variational_dem_diferenta_h}) and
some bases of $X_{h}$ and $Y_{h}$.

\qquad
\end{proof}

\begin{cor}
\label{observatia_ec_Poisson_grad_p_h_matrice_20} Assume that
$\bar{\textbf{z}}_{h}$ is evaluated in terms of
$\bar{\textbf{u}}_{h}$, $\bar{\textbf{z}}_{h}$ $=$ $-K_{h}^{-1}A
\bar{\textbf{u}}_{h}$. We denote $G_{h}(\alpha)$ $=$
$-\left[\begin{array}{c}
      \alpha C_{h} \\
      C_{h,0}
      \end{array}\right]K_{h}^{-1}A_{h}$. Then, the problem has the
      formulation (\ref{ecStokes_ecPoisson_variational_01})
with two block Poisson preconditioners $A_{h}$ and $D_{h}$.
\end{cor}

\begin{proof} This comes from (\ref{ecStokes_ecPoisson_variational_01_matr_3}).

\qquad
\end{proof}

\begin{cor}
\label{lema_AizomBCDizom_2_beta} Assume the hypotheses of
Corollaries \ref{corolarul5_teorema_izom_1_cor10_div0},
\ref{corolarul5_teorema_izom_1_cor10} and
\ref{observatia_ec_Poisson_grad_p_h_modalitate_zh}. It is expected
that $( \textbf{z}_{h}, grad \, \bar{\lambda}_{h} )$ $\neq$ $0$,
$\forall$ $\bar{\lambda}_{h}$ $\in$ $Y_{0h}$. Then
$(\textbf{u}_{h,1},\textbf{z}_{h,1},p_{h,1})$ $\neq$
$(\textbf{u}_{h,2},\textbf{z}_{h,2},p_{h,2})$, for $\alpha_{1}$
$\neq$ $\alpha_{2}$, where
$(\textbf{u}_{h,i},\textbf{z}_{h,i},p_{h,i})$ is the solution of
(\ref{e1_11_ecStokes_ecPoisson_variational_dem_h}),
(\ref{e1_11_ecStokes_ecPoisson_variational_dem_diferenta_h}),
(\ref{VAR_mathcal_P2_e1_6_ecStokes_ecPoisson_variational_dem_q_compl_dem11_cont_alt_izom_PRESIUNE_dir_h}),
(\ref{e1_11_ecStokes_ecPoisson_variational_dem_grad_miu_inlocuiri_sum_2_L2_cor_h})
for $\alpha$ $=$ $\alpha_{i}$.
\end{cor}

\begin{proof}
We obtain
\begin{eqnarray*}
   && (grad \, (\textbf{u}_{h,1}-\textbf{u}_{h,2}),grad \, \textbf{w}_{h})
      -((p_{h,1}-p_{h,2}),div \, \textbf{w}_{h})
            = 0, \ \forall \, \textbf{w}_{h}  \, \in \, \textbf{X}_{0h} \, ,
         \label{e1_11_ecStokes_ecPoisson_variational_dem_h_dif} \\
   && (grad \, (\textbf{u}_{h,1}-\textbf{u}_{h,2}),grad \, \widetilde{\textbf{w}}_{h})
      - ((\textbf{z}_{h,1}-\textbf{z}_{h,2}),\widetilde{\textbf{w}}_{h})
            = 0, \ \forall \, \widetilde{\textbf{w}}_{h}  \, \in \, \textbf{X}_{0h} \, ,
         \label{e1_11_ecStokes_ecPoisson_variational_dem_diferenta_h_dif} \\
   && \alpha_{1} ( \textbf{z}_{h,1}-\textbf{z}_{h,2}, grad \, \bar{\lambda}_{h} )
            + (grad \, (p_{h,1}-p_{h,2}),grad \, \bar{\lambda}_{h})
          \label{VAR_mathcal_P2_e1_6_ecStokes_ecPoisson_variational_dem_q_compl_dem11_cont_alt_izom_PRESIUNE_dir_h_dif} \\
   && = (\alpha_{2}-\alpha_{1})( \textbf{z}_{h,2}, grad \, \bar{\lambda}_{h} ), \
         \forall \, \bar{\lambda}_{h}  \, \in \, Y_{0h} \, ,
         \nonumber \\
   && ((\textbf{z}_{h,1}-\textbf{z}_{h,2}) + grad \, (p_{h,1}-p_{h,2}),grad \, \varphi_{h})
            = 0, \ \forall \, \varphi_{h}  \, \in \, Y_{h} \ominus Y_{0h} \, ,
         \label{e1_11_ecStokes_ecPoisson_variational_dem_grad_miu_inlocuiri_sum_2_L2_cor_h_dif}
\end{eqnarray*}

From Theorem \ref{teorema_izom_1_hhh_alfa}, we deduce that
$(\textbf{u}_{h,1},\textbf{z}_{h,1},p_{h,1})$ $=$
$(\textbf{u}_{h,2},\textbf{z}_{h,2},p_{h,2})$ if and only if $(
\textbf{z}_{h,2}, grad \, \bar{\lambda}_{h} )$ $=$ $0$, $\forall$
$\bar{\lambda}_{h}$ $\in$ $Y_{0h}$.

In the situation considered in Corollary
\ref{observatia_ec_Poisson_grad_p_h_modalitate_zh}, it is expected
that $( \textbf{z}_{h}, grad \, \bar{\lambda}_{h} )$ $=$ $- (div
\, \textbf{z}_{h},\bar{\lambda}_{h} )$ $\neq$ $0$.

\qquad
\end{proof}

\begin{rem}
\label{observatia5_nodalFEM} Assume the hypotheses of Corollary
\ref{observatia_ec_Poisson_grad_p_h_modalitate_zh}. We can
simplify the problem by replacing $(grad \, \textbf{u}_{h},grad \,
\textbf{w}_{h})$ from
(\ref{e1_11_ecStokes_ecPoisson_variational_dem_diferenta_h}) by
$(\textbf{z}_{h},\textbf{w}_{h})$ in
(\ref{e1_11_ecStokes_ecPoisson_variational_dem_h}). In other
words, we replace
(\ref{e1_11_ecStokes_ecPoisson_variational_dem_h}),
(\ref{e1_11_ecStokes_ecPoisson_variational_dem_diferenta_h}) by
(\ref{mathcal_P2_e1_6_ecStokes_ecPoisson_variational_dem_q_compl_dem11_cont_alt_izom_instead2_delta_h_hat_delta_aprox}).
The system is given by
(\ref{mathcal_P2_e1_6_ecStokes_ecPoisson_variational_dem_q_compl_dem11_cont_alt_izom_instead2_delta_h_hat_delta_aprox}),
(\ref{e1_11_ecStokes_ecPoisson_variational_dem_grad_miu_ec_TAU_h}),
(\ref{VAR_mathcal_P2_e1_6_ecStokes_ecPoisson_variational_dem_q_compl_dem11_cont_alt_izom_PRESIUNE_dir_h}),
(\ref{e1_11_ecStokes_ecPoisson_variational_dem_grad_miu_inlocuiri_sum_2_L2_cor_h})
with the unknowns $\textbf{z}_{h}$, $p_{h}$, $\textbf{t}_{h}$.
$\textbf{u}_{h}$ results from
(\ref{e1_11_ecStokes_ecPoisson_variational_dem_diferenta_h}).
\end{rem}

\section{An analysis using the finite element method}
\label{sectiunea_ecuatii_Stokes_problema_aproximativa_partea_2}

Assume that $\Omega$ is a polygon in $\mathbb{R}^{2}$ or a
polyhedron in $\mathbb{R}^{3}$. Let $X_{h}$ be a finite
dimensional subspace of $H^{1}(\Omega)$. $h$ is the mesh
parameter.

We take $X_{h}$ $=$ $Y_{h}$. In the case of Stokes problem, if the
approximation space of the velocity is the same as the
approximation space, then the discrete inf-sup (or the LBB)
condition is not satisfied (for instance,
\cite{CLBichir_bib_Barth_Bochev_Gunzburger_Shadid2004,
CLBichir_bib_Bochev_Gunzburger2004,
CLBichir_bib_Bochev_GunzburgerLSFEM2009}).

Taking a regular triangulation $\Theta_{h}$ of $\Omega$ by
triangles $K$, define
\begin{equation} \label{ec12}
   X_{h}=\{ v_{h} \in C^{0}(\bar{\Omega}),
      v_{h}|_{K} \in \mathbb{P}_{k_{Q}}(K),
      \forall K \in \Theta_{h} \},
      \quad k_{Q} \geq 1 \, ,   \\
\end{equation}
where $\mathbb{P}_{k}(K)$ is the space of polynomials in the
variables $x_{1}$, ... , $x_{N}$, of degree less than or equal to
$k$, defined on $K$, $k \geq 1$. Let us assume the following
hypotheses: there exist some operators $I_{1} \in
L(H^{1}(\Omega),X_{h})$, $I_{0} \in L(L^{2}(\Omega),X_{h})$, some
integers $j_{1}$, $j_{0}$, and some positive constants
$\delta_{1}$, $\delta_{0}$, independent of $h$, such that
\begin{eqnarray}
    && \| \mu-I_{1}\mu \|_{1} \leq \delta_{1} \cdot h^{m} \cdot | \mu |_{m+1}, \
              \forall \mu \in H^{m+1}(\Omega), \ 1 \leq m \leq j_{1} \, .
        \label{HypH1}
\end{eqnarray}
\begin{eqnarray}
   && \| \mu-I_{0}\mu \|_{0} \leq \delta_{0} \cdot h^{m+1} \cdot | \mu |_{m+1}, \
              \forall \mu \in H^{m+1}(\Omega), \ 1 \leq m \leq j_{0} \, .
         \label{HypH2}
\end{eqnarray}

\begin{lem}
\label{lema_teorema5_1_dense_prStokes_FEM} Let us consider
$\mathcal{M}_{h,i}$, $i=1,\ldots,5$. Assume Hypotheses
(\ref{HypH1}) - (\ref{HypH2}). Let $\mathcal{E}_{1}$ $=$
$(C^{\infty}(\overline{\Omega}) \cap H_{0}^{1})^{N}$,
$\mathcal{E}_{2}$ $=$ $C^{\infty}(\overline{\Omega}) \cap
H_{0}^{1}$, $\mathcal{E}_{3}$ $=$ $C^{\infty}(\overline{\Omega})$,
$\mathcal{E}_{4}$ $=$ $C^{\infty}(\overline{\Omega}) \cap
L_{0}^{2}(\Omega)$, $\mathcal{E}_{5}$ $=$
$C_{0}^{\infty}(\Omega)^{N}$. Then,

i) $\mathcal{E}_{i}$ is dense in $E_{i}$ and
\begin{equation}
\label{e5_45p_conditia_q_dense_general_FEM}
   \lim_{h \rightarrow 0} \inf_{v_{h} \in E_{h,i}} \| v-v_{h} \|_{E_{i}}=0 \, ,
      \ \forall \, v \in \mathcal{E}_{i} \, ,
\end{equation}

ii) \begin{equation} \label{e5_45p_U_Uh_convergenta_FEM}
   \| v_{i} - v_{h,i} \|_{E_{i}}
   \rightarrow 0 \ \textrm{as} \ h \rightarrow 0 \, ,
\end{equation}
\end{lem}

\begin{proof}
The proof follows from \cite{CLBichir_bib_Atkinson_Han2009,
CLBichir_bib_Brenner_Scott2008, CLBichir_bib_Ciarlet2002,
CLBichir_bib_Dautray_Lions_vol6_1988, CLBichir_bib_A_Ern2005,
CLBichir_bib_Quarteroni_Valli2008}.

\qquad
\end{proof}

\begin{lem}
\label{lema_eroare_aprox_toate_pr} Assume that the hypotheses of
Lemma \ref{lema_teorema5_1_dense_prStokes_FEM}. Let $m \geq 1$. If
$\textbf{u}$ $\in$ $\textbf{H}_{0}^{1}(\Omega) \cap
\textbf{H}^{m+1}(\Omega)$, $q$ $\in$ $H_{0}^{1}(\Omega) \cap
H^{m+1}(\Omega)$, $\phi$ $\in$ $H^{m+1}(\Omega)$, $\textbf{z}$
$\in$ $\textbf{H}^{m+1}(\Omega)$, then there exist some constants
$C_{i}
> 0$, $i=1,\ldots,4$, such that
\begin{eqnarray}
   && | (T_{VPD}-T_{VPD,h})T_{VPD}^{-1}\textbf{u} |_{1}
      \leq C_{1} \cdot h^{m} \cdot | \textbf{u} |_{m+1} \, .
         \label{e5_1_eroare_aprox_Poisson_Dirichlet_Ta_STOKES} \\
   && | (T_{PD}-T_{PD,h})T_{PD}^{-1}q |_{1}
      \leq C_{2} \cdot h^{m} \cdot | q |_{m+1} \, .
         \label{e5_1_eroare_aprox_Poisson_HypH3_Td_STOKES} \\
   && \| (B_{1}-B_{1,h})B_{1}^{-1}\phi \|_{1}
      \leq C_{3} \cdot h^{m} \cdot \| \phi \|_{m+1} \, .
         \label{e5_1_eroare_aprox_Poisson_HypH2_Td_STOKES} \\
   && \| (\pi-\pi_{h})\pi^{-1}\textbf{z} \|_{0}
      \leq C_{4} \cdot h^{m+1} \cdot \| \textbf{z} \|_{m+1} \, .
         \label{e5_1_eroare_aprox_Poisson_HypH3_L0_H_Y_VAR}
\end{eqnarray}
\end{lem}

\begin{proof} These are obtained by standard methods (\cite{CLBichir_bib_Atkinson_Han2009,
CLBichir_bib_Brenner_Scott2008, CLBichir_bib_Ciarlet2002,
CLBichir_bib_Dautray_Lions_vol6_1988, CLBichir_bib_A_Ern2005,
CLBichir_bib_Quarteroni_Valli2008}). We also use
$\textbf{f}=T_{VPD}^{-1}\textbf{u}$, $\psi=T_{PD}^{-1}q$,
$\psi_{1}=B_{1}^{-1}\phi$, $\textbf{f}_{\ast}=\pi^{-1}\textbf{z}$
and $\pi_{h}\pi = \pi_{h}$.

\qquad
\end{proof}

\begin{lem}
\label{lema_eroare_aprox_toate_pr_0}
(\cite{CLBichir_bib_Bochev_Gunzburger2004}, Theorem 5.1, $k=1$)
For $\pi_{0h}$ from (\ref{SOL_prel_op_hat_pi_h_DEF1_zero}) and for
any $\textbf{u}$ $\in$ $\textbf{L}^{2}(\Omega)$, we retain for
$k=1$ in Theorem 5.1, \cite{CLBichir_bib_Bochev_Gunzburger2004},
the following facts:
\begin{eqnarray}
   && \| (I-\pi_{0h})\textbf{u} \|_{-1}  = \sup_{\phi \in \textbf{H}_{0}^{1}(\Omega)}
      \frac{((I-\pi_{0h})\textbf{u},\phi)}{\| \phi \|_{1}} \, ,
         \label{e5_1_eroare_aprox_Poisson_HypH2_Td_2_STOKES_abc}
\end{eqnarray}
\begin{eqnarray}
   && \| (I-\pi_{0h})\textbf{u} \|_{-1} \leq C h \| \textbf{u} \|_{0} \, .
         \label{e5_1_eroare_aprox_Poisson_HypH2_Td_2_STOKES_abc_3}
\end{eqnarray}
\end{lem}

In order to verify
(\ref{well_posed_e5_45p_conditia_q_dense_general_GAMA0_elliptic_restrictie})
and (\ref{well_posed_e5_45p_conditia_Cea_GAMA0_ini_restrictie}),
let us remember:

\begin{lem}
\label{lema_eroare_aprox_toate_pr_two_sets}
(\cite{CLBichir_bib_Sir_AM1978}) Consider two sets $A$, $B$
$\subseteq$ $\mathbb{R}$. Then,
\begin{equation}
\label{well_posed_e5_45p_conditia_q_dense_general_GAMA0_elliptic_restrictie_real}
   \inf (A+B) = \inf A + \inf B  \, ,
\end{equation}
\end{lem}

\begin{thm}
\label{teorema5_1_partea2NOU} i) Consider the spaces
$\mathcal{Q}$, $\mathcal{X}$, $\mathcal{Y}$, $\mathcal{Z}$ and
$\mathcal{X}_{h}$ defined in Section
\ref{sectiunea_ecuatii_remarks}. In the definition of
$\mathcal{X}_{h}$, the spaces $X_{h}$ and $Y_{h}$ are those from
the present Section. Let $\widehat{\mathcal{A}}_{h}$ be the
operator defined by
(\ref{mathcal_P2_e1_6_ecStokes_ecPoisson_variational_dem_q_compl_p_dem11_cont_alt_izom_h_sist})
-
(\ref{e1_11_ecStokes_ecPoisson_variational_dem_grad_miu_ec_TAU_h_sist}).

Then, the conclusions of Theorem
\ref{well_posed_teorema5_1_Stokes_math_B_L2} and of Corollary
\ref{well_posed_corolarul_teorema5_1_Stokes_math_B_L2_Tdense} hold
for
(\ref{mathcal_P2_e1_6_ecStokes_ecPoisson_variational_dem_q_compl_p_dem11_cont_alt_izom_h_sist})
-
(\ref{e1_11_ecStokes_ecPoisson_variational_dem_grad_miu_ec_TAU_h_sist}).

ii) Assume that $\Omega$ is a bounded convex polygon in
$\mathbb{R}^{2}$ or a polyhedron in $\mathbb{R}^{3}$. Let $m \geq
1$ and $\textbf{u}$ $\in$ $\textbf{H}_{0}^{1}(\Omega) \cap
\textbf{H}^{m+1}(\Omega)$, $q$, $\hat{q}$ $\in$ $H_{0}^{1}(\Omega)
\cap H^{m+1}(\Omega)$, $\textbf{z}$, $\textbf{t}$ $\in$
$\textbf{H}^{m+1}(\Omega)$, $p$, $r$ $\in$ $H^{m+1}(\Omega)$. Let
$q_{\ast,1}$ be $q_{\ast}$ for some fixed $h = h_{1} \leq h_{0}$
such that $q_{\ast}$ $\leq$ $q_{\ast,1}$, $\forall$ $h < h_{1}$.
Then, there exists a constant $C_{0} > 0$ such that we have the
estimate
\begin{equation} \label{well_posed_e5_45p_U_Uh_eroare_h_h_h}
   \left\| \left[\begin{array}{c}
      x \\
      U
      \end{array}\right]
      -
      \left[\begin{array}{c}
      x_{h} \\
      U_{h}
      \end{array}\right] \right\|_{\widehat{\Gamma}}
   \leq
   (1-q_{\ast,1})^{-1}
   \| \widehat{\mathcal{A}}^{-1} \|_{L(\widehat{\Delta},\widehat{\Gamma})}
   C_{0} \cdot h^{m} \cdot ||| U ||| \, ,
\end{equation}
where $||| U |||$ $=$ $| \textbf{u} |_{m+1}$ $+$ $| q |_{m+1}$ $+$
$| \hat{q} |_{m+1}$ $+$ $\| \textbf{z} \|_{m+1}$ $+$ $\| p
\|_{m+1}$ $+$ $\| r \|_{m+1}$ $+$ $\| \textbf{t} \|_{m+1}$.
\end{thm}

\begin{proof}

We use Theorem \ref{well_posed_teorema5_1_Stokes_math_B_L2},
Corollary
\ref{well_posed_corolarul_teorema5_1_Stokes_math_B_L2_Tdense},
Corollary \ref{well_posed_corolarul_teorema5_1_restrictie_dense}.

i) $\mathcal{S}$ $=$ $\mathcal{E}_{1}$ $\times$ $\mathcal{E}_{2}$
$\times$ $\mathcal{E}_{2}$ $\times$ $\mathcal{E}_{4}$ $\times$
$\mathcal{E}_{5}$ $\times$ $\mathcal{E}_{3}$ $\times$
$\mathcal{E}_{4}$ and it is also given in Section
\ref{sectiunea_ecuatii_remarks}.

Using an enumeration different from that for $\mathcal{E}_{k}$ and
for $\mathcal{M}_{h,k}$, let us write $V$ $=$ $(V_{1}$, $\ldots$,
$V_{7})$ $\in$ $\mathcal{X}$ $=$ $\mathcal{X}_{1}$ $\times$
$\ldots$ $\times$ $\mathcal{X}_{7}$, $\mathcal{S}$ $=$
$\mathcal{S}_{1}$ $\times$ $\ldots$ $\times$ $\mathcal{S}_{7}$,
$\mathcal{X}_{h}$ $=$ $\mathcal{X}_{h,1}$ $\times$ $\ldots$
$\times$ $\mathcal{X}_{h,7}$, $\mathcal{T}$ $=$
$(\mathcal{T}_{1}$, $\ldots$, $\mathcal{T}_{7})$,
$\mathcal{T}_{h}$ $=$ $(\mathcal{T}_{h,1}$, $\ldots$,
$\mathcal{T}_{h,7})$, and so on.

$\| V \|_{\mathcal{X}}$ $=$ $\| V_{1} \|_{\mathcal{X}_{1}}$ $+$
$\ldots$ $+$ $\| V_{7} \|_{\mathcal{X}_{7}}$.

Let us define the following sets of real numbers: $A_{j}$ $=$ $\{
\| V_{j} - V_{h,j} \|_{\mathcal{X}_{j}} | V_{h,j} \in
\mathcal{X}_{h,j} \}$, $A$ $=$ $\{ \| V - V_{h} \|_{\mathcal{X}} |
V_{h} \in \mathcal{X}_{h} \}$.

We have $A$ $=$ $\{ \| V_{1} - V_{h,1} \|_{\mathcal{X}_{1}} +
\ldots + \| V_{n} - V_{h,7} \|_{\mathcal{X}_{7}} |  V_{h,1} \in
\mathcal{X}_{h,1}, \ldots,  V_{h,7} \in \mathcal{X}_{h,7} \}$, so
$A$ $=$ $A_{1}$ $+$ $\ldots$ $+$ $A_{5}$ and $\inf (A_{1} + \ldots
+ A_{5}) = \inf A_{1} + \ldots + \inf A_{5}$ (we use
(\ref{well_posed_e5_45p_conditia_q_dense_general_GAMA0_elliptic_restrictie_real})),
that is, $\inf_{V_{h} \in \mathcal{X}_{h}} \| V-V_{h}
\|_{\mathcal{X}}$ $=$ $\sum_{j \: = 1}^{7} \inf_{V_{h,j} \in
\mathcal{X}_{h,j}} \| V_{j} - V_{h,j} \|_{\mathcal{X}_{j}}$.

Let us verify
(\ref{well_posed_e5_45p_conditia_q_dense_general_GAMA0_elliptic_restrictie}).
$\forall$ $V$ $\in$ $\mathcal{S}$, we have:

$\lim_{h \rightarrow 0} \inf_{V_{h} \in \mathcal{X}_{h}}$ $\|
V-V_{h} \|_{\mathcal{X}}$ $=$ $\sum_{j \: = 1}^{7} \lim_{h
\rightarrow 0} \inf_{V_{h,j} \in \mathcal{X}_{h,j}} \| V_{j} -
V_{h,j} \|_{\mathcal{X}_{j}}$

Using Lemma \ref{lema_teorema5_1_dense_prStokes_FEM} i), it
results
(\ref{well_posed_e5_45p_conditia_q_dense_general_GAMA0_elliptic_restrictie}).

Let us verify
(\ref{well_posed_e5_45p_conditia_Cea_GAMA0_ini_restrictie}). We
denote $\gamma_{0}$ $=$ $\sup_{i \in \{ 1, \ldots, 7
\}}\gamma_{i}$. Let us use
(\ref{e5_45p_conditia_q_dense_general_Cea}).

$\| W-W_{h} \|_{\mathcal{X}}$ $=$ $\sum_{j \: = 1}^{7} \| W_{j} -
W_{h,j} \|_{\mathcal{X}_{j}}$ $\leq$ $\sum_{j \: = 1}^{7}
\gamma_{j} \inf_{V_{h,j} \in \mathcal{X}_{h,j}} \| W_{j} - V_{h,j}
\|_{\mathcal{X}_{j}}$ $\leq$ $\gamma_{0} \sum_{j \: = 1}^{7}
\inf_{V_{h,j} \in \mathcal{X}_{h,j}} \| W_{j} - V_{h,j}
\|_{\mathcal{X}_{j}}$ $=$ $\gamma_{0} \inf_{V_{h} \in
\mathcal{X}_{h}}\| W-V_{h} \|_{\mathcal{X}}$, so
(\ref{well_posed_e5_45p_conditia_Cea_GAMA0_ini_restrictie}).

ii) We have
(\ref{well_posed_e5_45p_conditia_q_restrictie_dense_lim_T}). So we
can take $h_{1}$ and $q_{\ast,1}$. $q_{\ast}$ $<$ $q_{\ast,1}$
implies $(1-q_{\ast})^{-1}$ $<$ $(1-q_{\ast,1})^{-1}$.

$\forall$ $h < h_{1}$, let us use
(\ref{well_posed_e5_45p_U_Uh_eroare}) and the estimates from Lemma
\ref{lema_eroare_aprox_toate_pr}. We denote $C_{0}$ $=$ $\sup \{
C_{1}, C_{2}, C_{3}, C_{4}h_{1} \}$.

We have $\| (\mathcal{T} - \mathcal{T}_{h})\mathcal{T}^{-1}U
\|_{\mathcal{X}}$ $=$ $\sum_{j \: = 1}^{7} \| (\mathcal{T}_{j} -
\mathcal{T}_{h,j})\mathcal{T}_{j}^{-1}U_{j} \|_{\mathcal{X}_{j}}$
$\leq$ $C_{1} \cdot h^{m} \cdot | \textbf{u} |_{m+1}$ $+$ $C_{2}
\cdot h^{m} \cdot | q |_{m+1}$ $+$ $C_{2} \cdot h^{m} \cdot |
\hat{q} |_{m+1}$ $+$ $C_{4} \cdot h^{m+1} \cdot \| \textbf{z}
\|_{m+1}$ $+$ $C_{3} \cdot h^{m} \cdot \| p \|_{m+1}$ $+$ $C_{3}
\cdot h^{m} \cdot \| r \|_{m+1}$ $+$ $C_{4} \cdot h^{m+1} \cdot \|
\textbf{t} \|_{m+1}$  $<$ $C_{0} \cdot h^{m} \cdot ($ $|
\textbf{u} |_{m+1}$ $+$ $| q |_{m+1}$ $+$ $| \hat{q} |_{m+1}$ $+$
$\| \textbf{z} \|_{m+1}$ $+$ $\| p \|_{m+1}$ $+$ $\| r \|_{m+1}$
$+$ $\| \textbf{t} \|_{m+1}$ $)$.

\qquad
\end{proof}

\section*{Acknowledgments}
\label{sectiunea_Acknowledgment}

The author owes gratitude to Mr. Costache Bichir and Mrs. Georgeta
Bichir for their moral and financial support without which this
paper would not have existed.

Many thanks to Mrs. Georgeta Stoenescu for her advice with the
English language.


\end{document}